\documentclass[11pt]{article}
 
\usepackage[margin=0.5in]{geometry} 
\usepackage{amsmath,amsthm,amssymb}
\usepackage{mathrsfs}
\usepackage[T1]{fontenc}
\usepackage[english]{babel}
\usepackage{graphicx}
\usepackage{color} 

\usepackage{hyperref}
\usepackage{afterpage}

\usepackage{color}

\usepackage{atbegshi}% http://ctan.org/pkg/atbegshi
\AtBeginDocument{\AtBeginShipoutNext{\AtBeginShipoutDiscard}}
\usepackage{pdfpages} % Para por pdf no inicio
\usepackage{longtable} % longtable's, duh
\usepackage{amsmath,amsthm,amsfonts,amssymb,relsize,bm} %maths stuff
\usepackage{xparse,xcolor} % cores
\usepackage{graphicx,float} %imagens, duh
\usepackage{etoolbox}
\usepackage[shortlabels]{enumitem} %enumerate environment
\usepackage{tikz-cd} %Diagramas Comutativos
\usepackage{cite} %Para aglomerar citacoes
\usepackage[normalem]{ulem}
\usepackage{comment}
\usepackage{tocloft}
\numberwithin{equation}{section}

\theoremstyle{plain}
\newtheorem{theorem}{Theorem}[section]
\newtheorem*{theorem*}{Theorem}
\newtheorem{lemma}{Lemma}[section]
\newtheorem{corollary}{Corollary}[section]

\theoremstyle{definition}
\newtheorem{definition}{Definition}[section]
\newtheorem*{definition*}{Definition}
\newtheorem{example}{Example}[section]
\newtheorem{remark}{Remark}[section]
\begin{document}
 
 \title{Certain Extensions of Results of Siegel, Wilton and Hardy}

\author{{Pedro Ribeiro,\quad    Semyon  Yakubovich} } %Department of Mathematics, Faculty of  Sciences,  University of Porto, \\ Rua do Campo Alegre,  687; 4169-007 Porto (Portugal)}}  
\thanks{


Department of Mathematics, Faculty of Sciences of University of Porto, Rua do Campo Alegre,  687; 4169-007 Porto (Portugal). 

\,\,\,\,\,E-mail of Corresponding author: pedromanelribeiro1812@gmail.com}
\date{}
 
\maketitle

\begin{abstract} Recently, Dixit et al. \cite{DKMZ} established a very elegant generalization of Hardy's Theorem concerning the infinitude of zeros that the Riemann zeta function possesses at its critical line.

By introducing a general transformation formula for the theta function involving the Bessel and Modified Bessel functions of the first kind, we extend their result to a class of Dirichlet series satisfying Hecke's functional equation. In the process, we also find new generalizations of classical identities in Analytic number theory.
\end{abstract}

\tableofcontents

\newpage

\begin{center}\section{Introduction and Main Results}\end{center}

Let $\zeta(s)$ denote the Riemann zeta function and define $\eta(s):=\pi^{-s/2}\Gamma(s/2)\,\zeta(s)$. A. Dixit, N. Robles,
A. Roy and A. Zaharescu \cite{DRRZ} proved the following theorem.
\\

\paragraph*{Theorem A:} Let $(c_{j})_{j\in\mathbb{N}}$ be a sequence of non-zeros real numbers
such that $\sum_{j=1}^{\infty}\,|c_{j}|<\infty$. Also, let $\left(\lambda_{j}\right)_{j\in\mathbb{N}}$
be a bounded sequence of distinct real numbers that attains its bounds.
Then the function
\[
F(s)=\sum_{j=1}^{\infty}c_{j}\,\eta\left(s+i\lambda_{j}\right):=\sum_{j=1}^{\infty}c_{j}\,\pi^{-\frac{s+i\lambda_{j}}{2}}\Gamma\left(\frac{s+i\lambda_{j}}{2}\right)\,\zeta\left(s+i\lambda_{j}\right)
\]
has infinitely many zeros on the critical line $\text{Re}(s)=\frac{1}{2}$.

\bigskip{}

Based on an integral representation of Jacobi's transformation formula
due to Dixit [\cite{Dixit_theta}, p. 374, eq. (1.13)],
A. Dixit, R. Kumar, B. Maji and A. Zaharescu \cite{DKMZ} later generalized the
aforementioned result and proved the more general theorem.

\paragraph*{Theorem B:} \label{DKMZ_result} Let $(c_{j})_{j\in\mathbb{N}}$ and $(\lambda_{j})_{j\in\mathbb{N}}$
be as in Theorem A. Also, let $\mathscr{R}$ denote the region of
the complex plane defined\footnote{The region indicated in \cite{DKMZ} is actually $|\text{Re}(z)-\text{Im}(z)|<\sqrt{\frac{\pi}{2}}-\sqrt{\frac{2}{\pi}}\,\text{Re}(z)\,\text{Im}(z)$,
which contains $\mathscr{R}$, but in this paper we only
focus on rectangular regions as the one indicated in the above statement.} by $\mathscr{R}:=\left\{ z\in\mathbb{C}\,:\,|\text{Re}(z)|<\sqrt{\frac{\pi}{2}},\,|\text{Im}(z)|<\sqrt{\frac{\pi}{2}}\right\} $.

\bigskip{}

Then, for any $z\in\mathscr{R}$, the function
\begin{equation}
F_{z}(s)=\sum_{j=1}^{\infty}c_{j}\,\eta\left(s+i\lambda_{j}\right)\,\left\{ _{1}F_{1}\left(\frac{1-(s+i\lambda_{j})}{2};\,\frac{1}{2};\,\frac{z^{2}}{4}\right)\,+\,_{1}F_{1}\left(\frac{1-(\overline{s}-i\lambda_{j})}{2};\,\frac{1}{2};\,\frac{\overline{z}^{2}}{4}\right)\right\} \label{def function theroem B intro}
\end{equation}
has infinitely many zeros on the critical line $\text{Re}(s)=\frac{1}{2}$.

\bigskip{}

The proofs of both theorems used a very elegant generalization of
Hardy's method of studying the moments of the real function $\eta\left(\frac{1}{2}+it\right)$
[\cite{hardy_note}, \cite{titchmarsh_zetafunction}, Chapter X]. One of the crucial
steps in Hardy's proof is the transformation formula for Jacobi's
theta function, usually defined as
\begin{equation}
\theta(x):=\sum_{n=-\infty}^{\infty}e^{-\pi n^{2}x}=2\psi(x)+1,\label{definition theta function in introduction}
\end{equation}
with $\psi(x)=\sum_{n=1}^{\infty}e^{-\pi n^{2}x}$. This transformation
formula is:
\begin{equation}
x^{1/2}\left(1+2\psi(x)\right)=1+2\psi\left(\frac{1}{x}\right).\label{Transformation formula Jacobi theta}
\end{equation}

\bigskip{}

It was also established by Jacobi that, if we consider a two-variable
generalization of $\psi(x)$ in the form,
\begin{equation}
\psi(x,z):=\sum_{n=1}^{\infty}e^{-\pi n^{2}x}\cos\left(\sqrt{\pi x}\,n\,z\right),\,\,\,\,\text{Re}(x)>0,\,\,\,z\in\mathbb{C},\label{generalization of Jacobi formula with z intro}
\end{equation}
the transformation formula takes place [\cite{DKMZ}, p. 312, eq.
(2.6)]
\begin{equation}
\sqrt{x}\left(1+2\psi(x,z)\right)=e^{-z^{2}/4}\left(1+2\,\psi\left(\frac{1}{x},iz\right)\right).\label{transformation formula Jacobi with parameter z}
\end{equation}

\bigskip{}

Dixit {[}\cite{Dixit_theta}, p. 374, eq. (1.13){]} realized
that (\ref{transformation formula Jacobi with parameter z}) could
be achieved through an integral representation involving the Riemann
zeta function, proving the elegant identity,
\begin{align}
2x^{1/4}\psi(x,z)-x^{-1/4}e^{-z^{2}/4} & =2e^{-z^{2}/4}x^{-1/4}\,\psi\left(\frac{1}{x},\,iz\right)-x^{1/4}=\nonumber \\
=\frac{1}{2\pi}\,\intop_{-\infty}^{\infty}\pi^{-\frac{1}{4}-\frac{it}{2}}\Gamma\left(\frac{1}{4}+\frac{it}{2}\right)\,\zeta\left(\frac{1}{2}+it\right) & \,_{1}F_{1}\left(\frac{1}{4}+\frac{it}{2};\,\frac{1}{2};\,-\frac{z^{2}}{4}\right)\,x^{-\frac{it}{2}}\,dt.\label{Dixit integral formula jacobi psi}
\end{align}

\bigskip{}

This formula played a fundamental role in proving Theorem B above (see also the very interesting survey \cite{Ramanujan_ingenious} for more integral representations of several modular-type transformation formulas).
The purpose of this paper is to see how Theorem B can be extended
to a class of Dirichlet series satisfying Hecke's functional equation.

Our method is an extension of the one employed in \cite{DKMZ}, but
in order to use it we need to develop a general formulation of the
integral identity (\ref{Dixit integral formula jacobi psi}).

\bigskip{}

Although we proceed in the set up of Chandrasekharan, Narasimhan and
Berndt \cite{arithmetical identities, dirichletserisI}, for the purpose of
introducing the main results of this paper, we just give some examples.

For any $\alpha>0$, define $r_{\alpha}(n)$ as the arithmetical function {[}\cite{suzuki}, p. 481, eq. (1.6){]}
\begin{equation}
\theta^{\alpha}(x)-1:=\sum_{n=1}^{\infty}r_{\alpha}(n)\,e^{-\pi nx}.\label{r alpha definition iiiintro}
\end{equation}

For $\text{Re}(s)$ sufficiently large, we can consider formally the Dirichlet series
\begin{equation}
\zeta_{\alpha}(s):=\sum_{n=1}^{\infty}\frac{r_{\alpha}(n)}{n^{s}},\label{zeta alpha definition intro}
\end{equation}
and motivate its study through the transformation formula for $\theta(x)$
(\ref{Transformation formula Jacobi theta}). The zeta function (\ref{zeta alpha definition intro})
seems to have been introduced in a paper by Lagarias and Rains \cite{lagarias_reins}
who studied, among several other things, how the distribution of its
zeros could vary with the index $\alpha$. Suzuki proved the analytic
continuation and the functional equation for $\zeta_{\alpha}(s)$,
with $0<\alpha<1$, and derived from it a formula similar to Selberg-Chowla's \cite{selberg_chowla} involving its coefficients.

\bigskip{}

When $\alpha=k\in\mathbb{N}$, it is effortless to see that $r_{\alpha}(n)$
reduces to the arithmetical function counting the number of representations
of $n$ as a sum of $k$ squares. Also, when $\alpha=1$, $\zeta_{1}(s)$
reduces to $2\zeta(2s)$. Like $\zeta_{k}(s)$ and $\zeta(2s)$, $\zeta_{\alpha}(s)$
satisfies the functional equation
\[
\eta_{\alpha}(s):=\pi^{-s}\Gamma(s)\,\zeta_{\alpha}(s)=\pi^{-\left(\frac{\alpha}{2}-s\right)}\Gamma\left(\frac{\alpha}{2}-s\right)\,\zeta_{\alpha}\left(\frac{\alpha}{2}-s\right):=\eta_{\alpha}\left(\frac{\alpha}{2}-s\right),
\]
from which it is possible to conclude that $\eta_{\alpha}(s)$ is
real on the critical line $\text{Re}(s)=\frac{\alpha}{4}$.

\bigskip{}

One of the first results of our paper is the extension of Theorem
B above to $\zeta_{\alpha}(s)$. Its statement is as follows.

\begin{theorem} \label{zeta alpha hypergeometric zeros}
Let $(c_{j})_{j\in\mathbb{N}}$ be a sequence of real numbers such
that $\sum_{j}|c_{j}|<\infty$ and $\left( \lambda_{j}\right) _{j\in\mathbb{N}}$
be a bounded sequence of distinct real numbers attaining its bounds. Then,
for any $z$ satisfying the condition:
\begin{equation}
z\in\mathscr{D}_{\alpha}:=\left\{ z\in\mathbb{C}\,:\,|\text{Re}(z)|<\sqrt{\frac{\pi\alpha}{2}},\,|\text{Im}(z)|<\sqrt{\frac{\pi\alpha}{2}}\right\} ,\label{condition rectangle z zeta alpha}
\end{equation}
the function
\begin{equation}
F_{z,\alpha}(s):=\sum_{j=1}^{\infty}c_{j}\,\pi^{-(s+i\lambda_{j})}\Gamma\left(s+i\lambda_{j}\right)\,\zeta_{\alpha}(s+i\lambda_{j})\,\left\{ _{1}F_{1}\left(\frac{\alpha}{2}-s-i\lambda_{j};\,\frac{\alpha}{2};\,\frac{z^{2}}{4}\right)+\,_{1}F_{1}\left(\frac{\alpha}{2}-\overline{s}+i\lambda_{j};\,\frac{\alpha}{2};\,\frac{\overline{z}^{2}}{4}\right)\right\} \label{function defining coooombinations}
\end{equation}
has infinitely many zeros on the critical line $\text{Re}(s)=\frac{\alpha}{4}$.

\end{theorem}

\bigskip{}

In particular, for $\alpha=1$ we can recover Theorem B above. For $\alpha=k\in\mathbb{N}$, our result generalizes
a Theorem of Siegel \cite{Siegel_Contributions} and Landau \cite{landau_hardy} (see also \cite{Hecke_middle_line}),
which says that the completed Dirichlet series
\[
\eta_{k}(s):=\pi^{-s}\Gamma(s)\,\zeta_{k}(s)
\]
possesses infinitely many zeros at its critical line\footnote{In fact, Siegel's result is very sharp and provides detailed information
about the zeros of $\zeta_{k}(s)$ on a fairly large strip of the complex plane.} $\text{Re}(s)=\frac{k}{4}$.

\bigskip{}

One of the main ingredients in the proof Theorem \ref{zeta alpha hypergeometric zeros} is a generalization
of Dixit's integral formula (\ref{Dixit integral formula jacobi psi}):
as it will be seen in Example \ref{example 2.1} of the present paper, it is possible
to prove the transformation formula
\begin{align}
\sqrt{x}\,e^{z^{2}/8}\,\sum_{n=1}^{\infty}r_{\alpha}(n)\,n^{\frac{1}{2}-\frac{\alpha}{4}}\,e^{-\pi n\,x}\,J_{\frac{\alpha}{2}-1}(\sqrt{\pi\,n\,x}\,z)-\left(\sqrt{\frac{\pi}{x}\,}\frac{z}{2}\right)^{\frac{\alpha}{2}-1}\frac{e^{-z^{2}/8}}{\Gamma\left(\frac{\alpha}{2}\right)\sqrt{x}}\nonumber \\
=\frac{e^{-z^{2}/8}}{\sqrt{x}}\,\sum_{n=1}^{\infty}r_{\alpha}(n)\,n^{\frac{1}{2}-\frac{\alpha}{4}}\,e^{-\frac{\pi n}{x}}\,I_{\frac{\alpha}{2}-1}\left(\sqrt{\frac{\pi n}{x}}\,z\right)-\left(\sqrt{\pi x}\,\frac{z}{2}\right)^{\frac{\alpha}{2}-1}\,\frac{\sqrt{x}\,e^{z^{2}/8}}{\Gamma\left(\frac{\alpha}{2}\right)}\nonumber \\
=\left(\frac{\sqrt{\pi}z}{2}\right)^{\frac{\alpha}{2}-1}\frac{e^{z^{2}/8}}{2\pi\,\Gamma\left(\frac{\alpha}{2}\right)}\,\intop_{-\infty}^{\infty}\pi^{-\frac{\alpha}{4}-it}\Gamma\left(\frac{\alpha}{4}+it\right)\zeta_{\alpha}\left(\frac{\alpha}{4}+it\right)\,_{1}F_{1}\left(\frac{\alpha}{4}+it;\,\frac{\alpha}{2};\,-\frac{z^{2}}{4}\right)\,x^{-it}\,dt,\label{integral formula ralpha given intro}
\end{align}
where $J_{\nu}(z)$ and $I_{\nu}(z)$ denote the Bessel functions
of the first kind. Formula (\ref{Dixit integral formula jacobi psi})
can be now recovered when $\alpha=1$.

\bigskip{}

Since the proof of (\ref{integral formula ralpha given intro}) was
achieved from a general point of view, identities akin to (\ref{integral formula ralpha given intro})
are also valid to other Dirichlet series.

For example, a version of it holds for Epstein zeta functions attached
to binary quadratic forms: if $Q(m,n)=Am^{2}+Bm\,n+Cn^{2}$ is a binary,
integral and positive definite quadratic form, we can consider the
Dirichlet series
\begin{equation}
\zeta\left(s,Q\right)=\sum_{m,n\neq0}\frac{1}{Q(m,n)^{s}},\,\,\,\,\,\text{Re}(s)>1,\label{epstein definition intro}
\end{equation}
known as the Epstein zeta function attached to $Q$.

In 1934, Potter and Titchmarsh \cite{Titchmarsh_Potter} extended Hardy's result to $\zeta(s,Q)$. Roughly one year later, Kober \cite{kober_zeros} furnished a considerable
simplification of their proof, more similar in spirit to Hardy's own
proof.

In this paper we are also able to extend the result given in Theorem B to a subclass of the zeta functions (\ref{epstein definition intro}).

\begin{theorem} \label{Epstein result}
Let $Q(x,y)=Ax^{2}+Bx\,y+Cy^{2}$ be a binary, integral and positive
definite quadratic form and let $\Delta:=4AC-B^{2}$. Assume also
that $Q$ is reduced, this is, $\text{gcd}(A,\,B,\,C)=1$ and that
$\sqrt{\Delta}\equiv2\mod4$.

Consider the Epstein zeta function attached to $Q$,
\begin{equation}
\zeta(s,\,Q)=\sum_{m,n\neq0}\frac{1}{\left(Am^{2}+Bm\,n+Cn^{2}\right)^{s}},\,\,\,\,\text{Re}(s)>1.
\end{equation}

\bigskip{}

If $(c_{j})_{j\in\mathbb{N}}$ is a sequence of non-zero real numbers
such that $\sum_{j=1}^{\infty}\,|c_{j}|<\infty$ and $\left(\lambda_{j}\right)_{j\in\mathbb{N}}$
is a bounded sequence of distinct real numbers attaining its bounds
and $z$ satisfies the condition:
\begin{equation}
z\in\mathscr{D}_{Q}:=\left\{ z\in\mathbb{C}\,:\,|\text{Re}(z)|<\frac{2\sqrt{\pi}}{\Delta^{3/4}},\,\,|\text{Im}(z)|<\frac{2\sqrt{\pi}}{\Delta^{3/4}}\right\} ;\label{Condition Quadratic form without characccter}
\end{equation}

Then the function
\[
F_{z,Q}(s)=\sum_{j=1}^{\infty}c_{j}\,\left(\frac{2\pi}{\sqrt{\Delta}}\right)^{-(s+i\lambda_{j})}\Gamma\left(s+i\lambda_{j}\right)\,\zeta(s+i\lambda_{j},\,Q)\,\left\{ _{1}F_{1}\left(1-s-i\lambda_{j};\,1;\,\frac{z^{2}}{4}\right)+\,_{1}F_{1}\left(1-\overline{s}+i\lambda_{j};\,1;\,\frac{\overline{z}^{2}}{4}\right)\right\} 
\]
has infinitely many zeros on the critical line $\text{Re}(s)=\frac{1}{2}$.
 
\end{theorem}

Another natural extension of Theorem B has to be concerned with Dirichlet
$L$-functions. At the end of their paper, Dixit, Kumar, Maji and Zaharescu \cite{DKMZ}
state a character analogue of the identity (\ref{Dixit integral formula jacobi psi})
and write that \textit{``it would be worthwhile (...) to find a character
analogue of Theorem 2 {[}Theorem B above{]}''}.   

\bigskip{}

Under certain restrictions on even characters, we have managed to prove a character
analogue of Theorem B. In fact, motivated by the study of our Theorem
\ref{zeta alpha hypergeometric zeros}, we establish something more general than that. 

\bigskip{}

In order to clarify our next statement, note that, in analogy with $\zeta_{\alpha}(s)$, we can define, just as Suzuki
does {[}\cite{suzuki}, p. 483, eq. (1.15){]}, a new Dirichlet series
attached to the powers of the theta function:
\begin{equation}
\theta(x,\chi):=\sum_{n\in\mathbb{Z}}n^{\delta}\chi(n)\,e^{-\frac{\pi n^{2}x}{q}},\,\,\,\,\delta=\begin{cases}
0 & \text{if }\chi(-1)=1\\
1 & \text{if }\chi(-1)=-1.
\end{cases}\label{character analogue theta}
\end{equation}

However, in order to define $\theta^{\alpha}(x,\chi)$ for every $\alpha>0$,
one would have to assure that $\chi$ is real and such that $\theta(x,\chi)>0$
for every $x>0$.

Whether this condition holds or not, it is always possible to
take an arbitrary integer power of $\theta(x,\chi)$, which results
in the following series:
\[
\theta^{k}(x,\chi)=\sum_{n_{1},...,n_{k}\in\mathbb{Z}}n_{1}^{\delta}\,\chi(n_{1})\,...\,n_{k}^{\delta}\,\chi(n_{k})\,e^{-\frac{\pi(n_{1}^{2}+...+n_{k}^{2})\,x}{q}}.
\]

Thus, in analogy to the Dirichlet series attached to the sum of $k-$squares,
we study the zeros of arbitrary combinations of the Dirichlet series\footnote{Note that the zero term $\left(n_{1},...,n_{k}\right)=(0,...,0)$
is not problematic in the expression of this Dirichlet series because
the non-principal characters vanish at zero.}
\begin{equation}
L_{k}\left(s,\chi\right):=\sum_{n_{1},...,n_{k}\in\mathbb{Z}}\frac{n_{1}^{\delta}\,\chi(n_{1})\,...\,n_{k}^{\delta}\,\chi(n_{k})}{\left(n_{1}^{2}+...+n_{k}^{2}\right)^{s}},\,\,\,\,\,\text{Re}(s)>\frac{k}{2}(1+\delta),\label{definition at intro of L k}
\end{equation}
which is clearly attached to the powers of $\theta(x,\chi)$. Note
that, if $k=1$, $L_{1}(s,\chi)=2\,L(2s-\delta,\chi)$. Also, $L_{k}\left(s,\chi\right)$ can be thought as a ``character
analogue'' of $\zeta_{k}(s)$. See Lemma \ref{functional equation L_k} below to see the functional equation satisfied by (\ref{definition at intro of L k}). 

\bigskip{}

Using a class of transformation formulas related to $L_{k}(s,\chi)$,
we will study the analytic continuation of (\ref{definition at intro of L k})
and establish the following Theorem.
\begin{theorem} \label{Dirichlet L functions result}
Let $\chi$ be a primitive Dirichlet character modulo $q$ and define
$\delta=0$ if $\chi$ is even and $\delta=1$ if $\chi$ is odd.

\bigskip{}

Moreover, if $\chi$ is even assume that $q\neq0\mod4$.

\bigskip{}

For $L_{k}\left(s,\chi\right)$ given in (\ref{definition at intro of L k}),
set 
\[
\eta_{k}\left(s,\chi\right):=\left(\frac{\pi}{q}\right)^{-s}\Gamma(s)\,L_{k}\left(s,\chi\right).
\]

\bigskip{}

Assume that $(c_{j})_{j\in\mathbb{N}}$ is a sequence of real numbers
such that $\sum_{j}|c_{j}|<\infty$ and $\left( \lambda_{j}\right) _{j\in\mathbb{N}}$
is a bounded sequence of distinct real numbers attaining its bounds.

\bigskip{}

Moreover, assume that $z$ satisfies the condition:
\begin{equation}
z\in\mathscr{D}_{q}:=\left\{ z\in\mathbb{C}\,:\,|\text{Re}(z)|<\sqrt{\frac{\pi}{2q}},\,|\text{Im}(z)|<\sqrt{\frac{\pi}{2q}}\right\} .\label{domain z character only k}
\end{equation}

\bigskip{}

Then the function $F_{z,k,\chi}(s)$ given by
\[
\sum_{j=1}^{\infty}c_{j}\,\eta_{k}\left(s+i\lambda_{j},\chi\right)\,\left\{ _{1}F_{1}\left(k\left(\frac{1}{2}+\delta\right)-s-i\lambda_{j};\,k\left(\frac{1}{2}+\delta\right);\,\frac{z^{2}}{4}\right)+{}_{1}F_{1}\left(k\left(\frac{1}{2}+\delta\right)-\overline{s}+i\lambda_{j};\,k\left(\frac{1}{2}+\delta\right);\,\frac{\overline{z}^{2}}{4}\right)\right\} 
\]
has infinitely many zeros on the critical line $\text{Re}(s)=\frac{k}{2}\left(\frac{1}{2}+\delta\right)$.

\end{theorem}

Finally, we consider an extension of Theorem B to Dirichlet series
attached to cusp forms for the full modular group. Wilton \cite{Wilton_tau}
used Hardy's method to prove that, if $p$ and $q$ are two integers
such that $p^{2}\equiv1 \mod q$, then the twisted $L-$function,
\begin{equation}
L_{\tau}\left(s,\frac{p}{q}\right)=\sum_{n=1}^{\infty}\frac{\tau(n)\,e^{\frac{2\pi ip}{q}n}}{n^{s}},\label{Dirichlet series tau twisted}
\end{equation}
where $\tau(n)$ is Ramanujan's $\tau-$function, has infinitely many
zeros on the line $\text{Re}(s)=6$. Since $\tau(n)$ defines the
Fourier coefficients of a cusp form with weight $k=12$, it is natural
that Wilton's proof extends to $L-$functions attached to different
cusp forms. Very recently, Meher, Pujahari and Shankhadhar \cite{meher_half}
employed Wilton's argument to prove that $L-$functions of cusp forms
having half-integral weight with level $4N^{2}$ have infinitely many
zeros on their critical lines.

\bigskip{}

Since any Dirichlet series of the form (\ref{Dirichlet series tau twisted})
satisfies Hecke's functional equation, our general formalism allows
to prove the following extension of Wilton's result.

\begin{theorem} \label{cusp form theorem}
Let $f(\tau)$ be a cusp form of weight $k$ for the full modular
group with real Fourier coefficients $a_{f}(n)$.

Consider the Dirichlet series:
\begin{equation}
L_{f}(s,p/q)=\sum_{n=1}^{\infty}\frac{a_{f}(n)\,e^{\frac{2\pi ip}{q}n}}{n^{s}},\,\,\,\,\,\,\,p^{2}\equiv1 \mod q.\label{L series cusp form p q}
\end{equation}

If $(c_{j})_{j\in\mathbb{N}}$ is a sequence of non-zero real numbers
such that $\sum_{j=1}^{\infty}\,|c_{j}|<\infty$, $\left(\lambda_{j}\right)_{j\in\mathbb{N}}$
is a bounded sequence of distinct real numbers attaining its bounds
and $z$ satisfies the condition:

\begin{equation}
z\in\mathscr{D}_{q}:=\left\{ z\in\mathbb{C}\,:\,|\text{Re}(z)|<2\sqrt{\frac{\pi}{q}},\,\,|\text{Im}(z)|<2\sqrt{\frac{\pi}{q}}\right\} ;\label{condition cusp case}
\end{equation}

Then the function

\begin{equation}
\sum_{j=1}^{\infty}c_{j}\,\left(\frac{2\pi}{q}\right)^{-s-i\lambda_{j}}\Gamma(s+i\lambda_{j})\,L_{f}\left(s+i\lambda_{j},\frac{p}{q}\right)\,\left\{ _{1}F_{1}\left(k-s-i\lambda_{j};\,k;\,\frac{z^{2}}{4}\right)+\,_{1}F_{1}\left(k-\overline{s}+i\lambda_{j};\,k;\,\frac{\overline{z}^{2}}{4}\right)\right\} \label{function cusp case}
\end{equation}
has infinitely many zeros on the critical line $\text{Re}(s)=\frac{k}{2}$.  
\end{theorem}

We would like to remark that it is possible to generalize Theorem B to
even more Dirichlet series.
Although we do not give the details in this paper, we have obtained similar results as the ones stated above in the following cases: 
\begin{enumerate}
\item Although $L_{k}(s,\chi)$ is a nice analogue of the Dirichlet series $\zeta_{k}(s)$,
there are other ways of constructing character analogues of Epstein
zeta functions. One such way is due to H. M. Stark {[}\cite{stark_epstein_characters_I, stark_epstein_characters_II}{]} (see also \cite{epstein_characters_berndt} for more constructions),
who introduced the character analogue of (\ref{epstein definition intro})
in the form
\begin{equation}
L\left(s,\chi,Q\right):=\sum_{m,n\neq0}\frac{\chi\left(Q(m,n)\right)}{Q(m,n)^{s}},\,\,\,\,\,\text{Re}(s)>1,\label{stark character analogue def}
\end{equation}
and studied its analytic properties when $\chi$ is primitive. One can prove an analogue of Theorem \ref{Epstein result} for this Dirichlet series, under the additional assumption that $(q,\Delta)=1$ and $\Delta$ satisfying the same conditions as in Theorem \ref{Epstein result}. 

\item Like in the case of (\ref{stark character analogue def}), we could get similar results to different analogues
of the Epstein zeta function such as
\[
\zeta\left(s,Q,\chi\right):=\sum_{m,n\neq0}\frac{\chi(m)\,\chi(n)}{Q(m,n)^{s}},\,\,\,\,\text{Re}(s)>1,
\]
as well as  
\[
\sum_{m_{1},...,m_{2k}\neq0}\frac{\chi\left(m_{1}^{2}+...+m_{2k}^{2}\right)}{\left(m_{1}^{2}+...+m_{2k}^{2}\right)^{s}},\,\,\,\,\text{Re}(s)>k.
\]
\item There are also extensions of Theorem \ref{cusp form theorem} to $L-$functions attached to cusp forms and twisted
by primitive Dirichlet characters. 
\item Instead of considering an Epstein zeta function in Theorem \ref{Epstein result}, one may prove 
prove a similar result for the Dirichlet series $\zeta(s)\,L(s,\chi)$,
where $\chi$ is an odd and primitive Dirichlet character whose modulus
is a perfect square.
\end{enumerate}
\bigskip{}

This paper is organized as follows. In the next subsections we introduce
some general definitions in the set up of Dirichlet series. We also
give the functional equations of the main $L$-functions here studied,
as well as some background on special functions.

Section \ref{section 2} is devoted to prove a general analogue of the theta transformation
formula, from which (\ref{Dixit integral formula jacobi psi}) is
derived as a particular case. We also provide examples which will
be helpful in establishing our main Theorems.

In the remaining sections of the paper we prove each particular Theorem.

\begin{center} \subsection{Definitions} \end{center}

To give a generalization of (\ref{Dixit integral formula jacobi psi})
we need the following definition.

\begin{definition}\label{Hecke and Bochner def}
Let $\left(\lambda_{n}\right)_{n\in\mathbb{N}}$ and $\left(\mu_{n}\right)_{n\in\mathbb{N}}$
be two sequences of positive numbers strictly increasing to $\infty$
and $\left(a(n)\right)_{n\in\mathbb{N}}$ and $\left(b(n)\right)_{n\in\mathbb{N}}$
two sequences of complex numbers not identically zero. Consider the
functions $\phi(s)$ and $\psi(s)$ representable as Dirichlet series 

\begin{equation}
\phi(s)=\sum_{n=1}^{\infty}\frac{a(n)}{\lambda_{n}^{s}}\,\,\,\,\,\,\text{and }\,\,\,\,\,\psi(s)=\sum_{n=1}^{\infty}\frac{b(n)}{\mu_{n}^{s}}\label{representable as Dirichlet series in first definition ever}
\end{equation}
with finite abscissas of absolute convergence $\sigma_{a}$ and $\sigma_{b}$
respectively. Let $\Delta(s)$ denote one of the following three gamma
factors: $\Gamma(s)$, $\Gamma\left(\frac{s}{2}\right)$ and $\Gamma\left(\frac{s+1}{2}\right)$
and $r$ be an arbitrary positive real number in the first case and
$1$ in the other two.

\bigskip{}

We say that $\phi(s)$ and $\psi(s)$ satisfy the functional equation

\begin{equation}
\Delta(s)\,\phi(s)=\Delta(r-s)\,\psi(r-s),\label{functional equation multi}
\end{equation}
if there exists a meromorphic function $\chi(s)$ with the following
properties:
\begin{enumerate}
\item $\chi(s)=\Delta(s)\,\phi(s)$ for $\text{Re}(s)>\sigma_{a}$ and $\chi(s)=\Delta(r-s)\,\psi(r-s)$
for $\text{Re}(s)<r-\sigma_{b}$;
\item $\lim_{|\text{Im}(s)|\rightarrow\infty}\chi(s)=0$ uniformly in every
interval $-\infty<\sigma_{1}\leq\text{Re}(s)\leq\sigma_{2}<\infty$. 
\item The singularities $\chi(s)$ are at most poles and are confined to
some compact set. 
\end{enumerate}
\end{definition}

We say that the pair of functions $\left(\phi,\,\psi\right)$ representable
as Dirichlet series (\ref{representable as Dirichlet series in first definition ever})
satisfy \textbf{Hecke's functional equation} if they satisfy the conditions
of the previous definition for $\Delta(s)=\Gamma(s)$. The particular
case of (\ref{functional equation multi}) reads
\begin{equation}
\Gamma(s)\,\phi(s)=\Gamma(r-s)\,\psi(r-s),\,\,\,\,\,\,r>0.\label{Hecke Dirichlet series Functional}
\end{equation}

\bigskip{}

Similarly, we say that the pair of functions $\left(\phi,\,\psi\right)$
representable as (\ref{representable as Dirichlet series in first definition ever})
with finite abscissas of absolute convergence $\sigma_{a}$
and $\sigma_{b}$ satisfy \textbf{Bochner's functional equation} if
they satisfy the functional equation
\begin{equation}
\Gamma\left(\frac{s+\delta}{2}\right)\phi(s)=\Gamma\left(\frac{1+\delta-s}{2}\right)\,\psi\left(1-s\right),\label{This is the first Bochner ever}
\end{equation}
in the sense of Definition \ref{Hecke and Bochner def}. 

\bigskip{}

Before proceeding further, we shall denote $\sum_{m,n\neq0}$ as the infinite sum over all integers
$m,\,n$ not simultaneously zero. The same notation is taken for multiple sums. Whenever we use the term ``critical line'' for a given Dirichlet
series $\phi(s)$ satisfying Definition \ref{Hecke and Bochner def} and the functional equation (\ref{Hecke Dirichlet series Functional}) we will be referring to $\text{Re}(s)=\frac{r}{2}$, which is the line of symmetry. 
\bigskip{}

The Dirichlet series considered in all of our Theorems have at most one simple pole as singularity. Therefore, it will be more convenient
for the analysis of this paper to consider subclasses of the Dirichlet
series defined above, with fewer singularities. We now give two
definitions introducing these. 

\begin{definition} \label{class A}
Let $\phi(s)$ be a Dirichlet series satisfying Definition \ref{Hecke and Bochner def} with
$\Delta(s)=\Gamma(s)$. We say that $\phi(s)$ belongs to the class $\mathcal{A}$
if additionally: 
\begin{enumerate}
\item $\phi(s)$ and $\psi(s)$ are analytic everywhere in $\mathbb{C}$
except for possible simple poles located at $s=r$ with residues $\rho$
and $\rho^{\star}$ respectively.
\end{enumerate} 
\end{definition}

\begin{remark} \label{properties class A}
By the functional equation (\ref{Hecke Dirichlet series Functional})
and the additional condition in the previous definition, $\phi(0)=-\rho^{\star}\Gamma(r)$,
while $\psi(0)=-\rho\,\Gamma(r)$. In particular, if $\psi(s)$ is
an entire Dirichlet series, then $\phi(0)=0$. It is also clear that
$\phi(s)\in\mathcal{A}$ if and only if $\psi(s)\in\mathcal{A}$.
The purpose of introducing this class is to mimic as much as possible
the class of Hecke Dirichlet series with a given signature, which is a subclass
of $\mathcal{A}$. A similar definition was also considered [\cite{hecke and identities Berndt}, p.221]. 

\end{remark}

\begin{remark} \label{zeros integers remark}
  It is clear from the definition above and the functional equation
(\ref{Hecke Dirichlet series Functional}) that $\phi(-n)=0$ for
every $n\in\mathbb{N}$.  
\end{remark}

Another class of Dirichlet series related to the previous one is given in the following definition: 

\begin{definition} \label{class B}
Let $\phi(s)$ be a Dirichlet series satisfying Definition \ref{Hecke and Bochner def} with
$\Delta(s)=\Gamma\left(\frac{s+\delta}{2}\right)$, $\delta\in\{0,1\}$,
$r=1$ (Bochner class). We say that $\phi(s)$ belongs to the class $\mathcal{B}$
if additionally: 
\begin{enumerate}
\item For $\delta=0$, $\phi(s)$ and $\psi(s)$ are analytic everywhere
in $\mathbb{C}$ except for possible simple poles located at $s=1$
with residues $\rho$ and $\rho^{\star}$ respectively.
\item For $\delta=1$, $\phi(s)$ and $\psi(s)$ can be analytically continued
as entire Dirichlet series.
\end{enumerate}
    
\end{definition}

\begin{remark} \label{Riemann Bochner Remark}
Note that definition \ref{class B} mimics in some way the properties that the
Dirichlet $L-$functions and the Riemann $\zeta-$function have. Note
that $\phi(s)=\pi^{-s/2}\zeta(s)$ is a Dirichlet series belonging
to the class $\mathcal{B}$ with $\delta=0$.

By the functional equation for Bochner Dirichlet series (\ref{This is the first Bochner ever}),
it is effortless to see that $\phi(0)=-\frac{\rho^{\star}}{2}\sqrt{\pi}$,
while $\psi(0)=-\frac{\rho}{2}\sqrt{\pi}$. In particular, if $\psi(s)$
is entire then $\phi(0)=0$.
    
\end{remark}

\begin{remark} \label{Bochner Hecke Remark}
 If $\phi(s)\in\mathcal{B}$ with $\delta=0$ and its residue at the
simple pole $s=1$ is $\rho$, then $\phi^{\prime}(s):=\phi(2s)\in\mathcal{A}$
having a simple pole located at $s=1/2$ with residue $\rho/2$. Moreover,
$\phi^{\prime}(s)$ satisfies Hecke's functional equation (\ref{Hecke Dirichlet series Functional})
with parameter $r=1/2$.

If $\phi(s)\in\mathcal{B}$ with $\delta=1$, then $\phi^{\prime}(s):=\phi(2s-1)\in\mathcal{A}$.
Since $\phi(s)$ is entire when $\delta=1$, $\phi^{\prime}(s)$ will
also be entire. Moreover, $\phi^{\prime}(s)$ satisfies Hecke's functional
equation (\ref{Hecke Dirichlet series Functional}) with parameter
$r=3/2$.
   
\end{remark}

\bigskip{}

The functional equations (\ref{Hecke Dirichlet series Functional})
and (\ref{This is the first Bochner ever}) can be translated into
arithmetical identities involving the sequences $a(n)$ and $b(n)$.
From now on, if $\phi(s)$ satisfies Hecke's functional equation in
the sense of Definition \ref{Hecke and Bochner def} and it is representable by the first
Dirichlet series in (\ref{representable as Dirichlet series in first definition ever}),
we shall use $\Phi(x)$ and $\Psi(x)$ to denote the generalized $\theta-$functions
(\ref{definition theta function in introduction})
\begin{equation}
\Phi(x):=\sum_{n=1}^{\infty}a(n)\,e^{-\lambda_{n}x},\,\,\,\,\Psi(x):=\sum_{n=1}^{\infty}b(n)\,e^{-\mu_{n}x}\,\,\,\,\,\,\text{Re}(x)>0.\label{definition generalized Theta function nnnn.}
\end{equation}

Following Bochner, for $\text{Re}(x)>0$, let $P(x)$
denote the residual function 
\begin{equation}
P(x)=\frac{1}{2\pi i}\,\intop_{C}\Gamma(s)\,\phi(s)\,x^{-s}ds,\label{residual function definition Bochner modular}
\end{equation}
where $C$ denotes a curve, or curves, encircling the singularities
of $\chi(s)$ given in Definition \ref{Hecke and Bochner def}. It was established in \cite{bochner_modular_relations}
that Hecke's functional equation for $\phi(s)$ and $\psi(s)$ (\ref{Hecke Dirichlet series Functional})
and the ``modular relation''
\begin{equation}
\Phi(x)=x^{-r}\Psi\left(\frac{1}{x}\right)+P(x)\label{Bochner Modular relation at intro}
\end{equation}
are equivalent.

It is interesting to note that the observation of this equivalence
has its genesis in Riemann's revolutionary memoir \cite{riemann}, where
one of the implications was proved for the first time.

Berndt {[}\cite{dirichlet and hecke}, \cite{dirichletserisIII}{]} established
an expansion involving Modified Bessel function which is equivalent
to (\ref{Bochner Modular relation at intro}). He proved that, if
$x>0$, then the curious formula takes place
\begin{equation}
\Gamma(s)\,\sum_{n=1}^{\infty}\frac{a(n)}{(\lambda_{n}+x^{2})^{s}}=R(s,x)+2x^{r-s}\,\sum_{n=1}^{\infty}b(n)\,\mu_{n}^{\frac{s-r}{2}}\,K_{s-r}\left(2\sqrt{\mu_{n}}\,x\right),\,\,\,\,\text{Re}(s)>\sigma_{a},\label{Berndt Bessel Expansion}
\end{equation}
where $R(s,x)$ denotes the sum of residues of the function $\Gamma(w)\,\phi(w)\,\Gamma(s-w)\,x^{2w-2s}$
at the poles of $\Gamma(w)\,\phi(w)$.

\bigskip{}

Since we will generalize (\ref{Bochner Modular relation at intro})
in the next section (see Theorem \ref{summation formula with 1F1} and Corollary \ref{integral representation theta} below), it
will be natural to seek a generalization of (\ref{Berndt Bessel Expansion}),
which will be given in Corollary \ref{generalized bessel expansion corollary 2.4}.

\begin{center}\subsection{Preliminary results}\end{center}

In several occasions throughout this paper, we shall need to estimate
the asymptotic order of certain integrals involving the Dirichlet
series $\phi(s)$ satisfying Definition \ref{Hecke and Bochner def}. To justify most of the
steps, we will often invoke the following version of Stirling's formula 
\begin{equation}
\Gamma(\sigma+it)=(2\pi)^{\frac{1}{2}}\,t^{\sigma+it-\frac{1}{2}}\,e^{-\frac{\pi t}{2}-it+\frac{i\pi}{2}(\sigma-\frac{1}{2})}\left(1+\frac{1}{12(\sigma+it)}+O\left(\frac{1}{t^{2}}\right)\right),\label{Stirling exact form on Introduction}
\end{equation}
as $t\rightarrow\infty$, uniformly for $-\infty<\sigma_{1}\leq\sigma\leq\sigma_{2}<\infty$.
A similar formula can be written for $t<0$ as $t$ tends to $-\infty$
by using the fact that $\Gamma(\overline{s})=\overline{\Gamma(s)}$.
Of course, a direct consequence of this exact version is
\begin{equation}
|\Gamma(\sigma+it)|=(2\pi)^{\frac{1}{2}}\,|t|^{\sigma-\frac{1}{2}}\,e^{-\frac{\pi}{2}|t|}\left(1+O\left(\frac{1}{|t|}\right)\right),\,\,\,\,\,|t|\rightarrow\infty.\label{preliminary stirling}
\end{equation}

\bigskip{}

To estimate the order of $\phi(s)$ at the line $\text{Re}(s)=\sigma$
we shall need a version of the classical Phragm\'en-Lindel\"of theorem
given in {[}\cite{titchmarsh_theory_of_functions}, p.180, 5.65{]} (see also {[}\cite{rysc_I}, p. 
11{]} for details). Adapted to our current needs, this version states that,
if $\phi(s)$ satisfies Hecke's functional equation (\ref{Hecke Dirichlet series Functional}) and has abscissa of absolute convergence $\sigma_{a}$,
then, as $|t|\rightarrow\infty$,
\begin{equation}
\phi(\sigma+it)=O\left(|t|^{\sigma_{a}+\delta-\sigma}\right),\,\,\,\,\,\,\,r-\sigma_{a}-\delta\leq\sigma\leq\sigma_{a}+\delta.\label{Lindelof Phragmen for Hecke}
\end{equation}

\bigskip{}

The next few lemmas are mainly concerned with the analytic continuation
of the Dirichlet series studied in this paper.

\bigskip{}

We start with $\zeta_{\alpha}(s)$. The theta function $\vartheta_{3}(\tau):=\sum_{n\in\mathbb{Z}}e^{\pi in^{2}\tau}$
is a modular form of weight $\frac{1}{2}$ with a multiplier system
with respect to the theta group $\Gamma_{\theta}$ (see [\cite{lagarias_reins},
pp. 15-16]). Therefore, for every $\alpha>0$, $\vartheta_{3}^{\alpha}(\tau)$
is a modular form of weight $\alpha/2$ with a multiplier system on
the same group.

\bigskip{}

By definition, $r_{\alpha}(n)$ are the Fourier coefficients of the
expansion of $\vartheta_{3}(\tau)$ at the cusp $i\infty$, this is:
\begin{equation}
\vartheta_{3}^{\alpha}(ix):=\theta^{\alpha}(x)=1+\sum_{n=1}^{\infty}r_{\alpha}(n)\,e^{-\pi n\,x},\,\,\,\,\,x>0.\label{second definition varthet coefficients}
\end{equation}

It is helpful to see how these coefficients show up: these are computable
by using the expansion
\begin{align}
\vartheta_{3}^{\alpha}(ix) & =\theta^{\alpha}(x)=\left(1+2\,\sum_{n=1}^{\infty}e^{-\pi n^{2}x}\right)^{\alpha}=\sum_{j=0}^{\infty}\left(\begin{array}{c}
\alpha\\
j
\end{array}\right)2^{j}\,\left(\sum_{n=1}^{\infty}e^{-\pi n^{2}x}\right)^{j}\nonumber \\
 & =1+\sum_{j=1}^{\infty}\left(\begin{array}{c}
\alpha\\
j
\end{array}\right)2^{j}\,\sum_{n_{1},...,n_{j}=1}^{\infty}e^{-\pi\left(n_{1}^{2}+...+n_{j}^{2}\right)x}.\label{power theta}
\end{align}

Defining a new summation variable, say $m=n_{1}^{2}+...+n_{j}^{2}$,
we have that the number of $(n_{i})_{1\leq i\leq j}$ decomposing
$n$ in this way is at most $m$. Therefore, the sum over $j$ is
finite (in fact, a polynomial of degree $m$ in $\alpha$), so we
may define $r_{\alpha}(m)$ as
\begin{equation}
\sum_{j=1}^{\infty}\left(\begin{array}{c}
\alpha\\
j
\end{array}\right)2^{j}\,\sum_{n_{1},...,n_{j}=1}^{\infty}e^{-\pi\left(n_{1}^{2}+...+n_{j}^{2}\right)x}:=\sum_{m=1}^{\infty}r_{\alpha}(m)\,e^{-\pi mx}.\label{construction power series ralpha}
\end{equation}

\bigskip{}

The order of growth of $r_{\alpha}(n)$ as $n\rightarrow\infty$ is
determined by classical estimates due to Petersson and Lehner \cite{lagarias_reins, suzuki, lehner} for the Fourier coefficients of arbitrary modular forms of
positive real weight with multiplier systems. These estimates show
that $r_{\alpha}(n)$ grows polynomially, more precisely in the form:
\begin{equation}
r_{\alpha}(n)\ll_{\alpha}\begin{cases}
n^{\alpha/2-1} & \alpha>4\\
n^{\alpha/2-1}\log(n) & \alpha=4\\
n^{\alpha/4} & 0<\alpha<4.
\end{cases}\label{estimate r alpha (n) useful}
\end{equation}

See also {[}\cite{lagarias_reins}, p. 18, Theorem 3.3.{]} for estimates
whose constants are independent of $\alpha$. These bounds determine
the existence of a finite abscissa of absolute convergence for $\zeta_{\alpha}(s)$
given in (\ref{zeta alpha definition intro}).

Our first lemma concerns the analytic continuation of (\ref{zeta alpha definition intro}),
which should resemble the analytic continuation of $\zeta_{k}(s)$.
For a proof see {[}\cite{lagarias_reins}, p. 11, Theorem 2.1.{]}.

\begin{lemma} \label{lemma 1.1}
Let $r_{\alpha}(n)$ be the sequence defined by (\ref{second definition varthet coefficients}).
Consider the Dirichlet series:
\[
\zeta_{\alpha}(s)=\sum_{n=1}^{\infty}\frac{r_{\alpha}(n)}{n^{s}},\,\,\,\,\,\,\text{Re}(s)>\sigma_{\alpha}:=\begin{cases}
\frac{\alpha}{2} & \alpha\geq4\\
1+\frac{\alpha}{4} & 0<\alpha<4
\end{cases}.
\]

Then $\zeta_{\alpha}(s)$ can be analytically continued to the entire
complex plane as a meromorphic function with a simple pole located
at $s=\frac{\alpha}{2}$ with residue $\text{Res}_{s=\alpha/2}\zeta_{\alpha}(s)=\frac{\pi^{\alpha/2}}{\Gamma(\alpha/2)}$.

Moreover, it satisfies Hecke's functional equation:
\begin{equation}
\pi^{-s}\Gamma(s)\,\zeta_{\alpha}(s)=\pi^{-\left(\frac{\alpha}{2}-s\right)}\Gamma\left(\frac{\alpha}{2}-s\right)\,\zeta_{\alpha}\left(\frac{\alpha}{2}-s\right).\label{Functional equation for zetaalpha}
\end{equation}
   
\end{lemma}

\bigskip{}

Our next lemma is due to Paul Epstein \cite{epstein_I, epstein_II} and
gives the analytic continuation and the functional equation for the
Epstein zeta function attached to any positive quadratic form. Before
stating it, let $\mathbf{x}$ denote a vector in $\mathbb{R}^{n}$
and $Q(\mathbf{x})$ a positive definite quadratic form in $n$ variables.
Write $Q(\mathbf{x})$ in the matrix form:
\begin{equation}
Q(\mathbf{x})=\frac{1}{2}\mathbf{x}^{T}\,A\,\mathbf{x},\label{matrix version quadratic}
\end{equation}
where $A$ is a symmetric $n\times n$ matrix of real numbers. Define
the discriminant of $Q$ to be
\[
D(Q):=\det(A),
\]
and the adjoint quadratic form $Q^{\dagger}(\mathbf{x})$ by
\begin{equation}
Q^{\dagger}(\mathbf{x})=\frac{1}{2}\mathbf{x}^{T}\,A^{\dagger}\,\mathbf{x},\label{adjoint quadratic defini}
\end{equation}
with $A^{\dagger}$ denoting the adjoint matrix of $A$. The Epstein
zeta functions attached to $Q$ are defined by the series:
\begin{equation}
\zeta\left(s,\,\mathbf{g},\,\mathbf{h},\,Q\right):=\sum_{\mathbf{m}\in\mathbb{Z}^{n},\,\mathbf{m}+\mathbf{g}\neq\mathbf{0}}\frac{\exp\left(2\pi i\,\mathbf{m}\cdot\mathbf{h}\right)}{Q\left(\mathbf{m}+\mathbf{g}\right)^{s}},\,\,\,\,\,\text{Re}(s)>\frac{n}{2},\label{deifnition epstein intor}
\end{equation}
where the $\mathbf{g}$ and $\mathbf{h}$ are vectors in $\mathbb{R}^{n}$.
The following lemma gives the analytic properties of (\ref{deifnition epstein intor})
(for an even more detailed statement, see {[}\cite{Siegel_Analytic Number Theory}, p.
54, Theorem 3{]}). 

\begin{lemma} \label{Epstein functional}
The Epstein zeta function $\zeta\left(s,\mathbf{g},\mathbf{h},\,Q\right)$
has an analytic continuation into the entire complex plane as:
\begin{enumerate}
\item An entire function if $\mathbf{h}\notin\mathbb{Z}^{n}$;
\item A meromorphic function with a simple pole at $s=\frac{n}{2}$ if $\mathbf{h}\in\mathbb{Z}^{n}$.
The residue that $\zeta\left(s,\mathbf{g},\mathbf{h},\,Q\right)$
possesses at $s=n/2$ is given by $(2\pi)^{n/2}\Gamma\left(\frac{n}{2}\right)/\sqrt{D(Q)}$.
\end{enumerate}
Moreover, $\zeta\left(s,\,\mathbf{g},\,\mathbf{h},\,Q\right)$ satisfies
Hecke's functional equation:
\begin{equation}
\left(\frac{2\pi}{D(Q)^{1/n}}\right)^{-s}\Gamma(s)\,\zeta\left(s,\,\mathbf{g},\,\mathbf{h},\,Q\right)=e^{-2\pi i\mathbf{g}\cdot\mathbf{h}}\,\left(\frac{2\pi}{D(Q^{\dagger})^{1/n}}\right)^{-(\frac{n}{2}-s)}\Gamma\left(\frac{n}{2}-s\right)\,\zeta\left(\frac{n}{2}-s,\mathbf{h},-\mathbf{g},Q^{\dagger}\right).\label{formula at intro}
\end{equation}    
\end{lemma}

\bigskip{}

The previous functional equation will be very useful in studying our next result which, although making its first appearance here as a lemma, will be proven in subsection \ref{proof of lemma 1.3.}. It concerns
the analytic continuation of the series (\ref{definition at intro of L k}).
\begin{lemma} \label{functional equation L_k}
Let $\chi$ be a primitive Dirichlet character modulo $q$ and set
$\delta=0$ if $\chi$ is even and $\delta=1$ if $\chi$ is odd.

\bigskip{}

If $L_{k}(s,\chi)$ is the Dirichlet series given by
\begin{equation}
L_{k}\left(s,\chi\right):=\sum_{n_{1},...,n_{k}\in\mathbb{Z}}\frac{n_{1}^{\delta}\,\chi(n_{1})\,...\,n_{k}^{\delta}\,\chi(n_{k})}{\left(n_{1}^{2}+...+n_{k}^{2}\right)^{s}},\,\,\,\,\,\text{Re}(s)>\frac{k}{2}(1+\delta),\label{sum of definition k squares L function}
\end{equation}
then $L_{k}(s,\chi)$ can be analytically continued as an entire function
and it satisfies the functional equation:
\begin{equation}
\left(\frac{\pi}{q}\right)^{-s}\Gamma(s)\,L_{k}\left(s,\chi\right)=\frac{(-i)^{\delta k}G^{k}(\chi)}{q^{k/2}}\,\left(\frac{\pi}{q}\right)^{-\left(k(\frac{1}{2}+\delta)-s\right)}\,\Gamma\left(k\left(\frac{1}{2}+\delta\right)-s\right)\,L_{k}\left(k\left(\frac{1}{2}+\delta\right)-s,\,\overline{\chi}\right),\label{Functional equation Lk (s, chi)}
\end{equation}
where $G(\chi)$ denotes the Gauss sum
\begin{equation}
G(\chi):=\sum_{r=0}^{q}\chi(r)\,e^{2\pi in\,r/q}.\label{Gauss sum definition}
\end{equation}   
\end{lemma}

\bigskip{}

We will prove this result on subsection \ref{proof of lemma 1.3.} because we could
not track any reference containing it, namely when $\delta=1$. Of
course, when $k=1$, (\ref{Functional equation Lk (s, chi)}) easily
reduces to the functional equation for Dirichlet $L$-functions and for
the case $\delta=0$ it actually
comes from a more general result of Berndt \cite{epstein_characters_berndt}.
For $\delta=1$, however, (\ref{Functional equation Lk (s, chi)}) seems to be absent in the literature.
\bigskip{}

To conclude our set of lemmas, let $f(\tau)$ be a cusp form of weight
$k\geq12$ for $\text{SL}(2,\mathbb{Z})$ on the upper half plane
$\mathbb{H}:=\left\{ \tau\in\mathbb{C}\,:\,\text{Im}(\tau)>0\right\} $
with Fourier expansion given by

\begin{equation}
f(\tau)=\sum_{n=1}^{\infty}a_{f}(n)\,e^{2\pi in\tau}.\label{Fourier expansion cusp fooorms}
\end{equation}

Then, for all $\left(\begin{array}{cc}
a & b\\
c & d
\end{array}\right)\in\text{SL}(2,\mathbb{Z})$ and $\tau\in\mathbb{H}$, $f(\tau)$ satisfies

\begin{equation}
f\left(\frac{a\tau+b}{c\tau+d}\right)=\left(c\tau+d\right)^{k}\,f(\tau),\,\,\,\,\,\lim_{\text{Im}(\tau)\rightarrow\infty}f(\tau)=0.\label{modular relation cusp forms}
\end{equation}

\bigskip{}

The following lemma, due to Wilton \cite{Wilton_tau}, gives the analytic
continuation of a twisted Dirichlet series attached to $f(\tau)$.
We quote Wilton's result in the same form as exposed in Jutila's text
{[}\cite{exponential_sums_jutilla}, p. 14, Lemma 1.2.{]}.

\begin{lemma} \label{functional equation cusp}
Let $f(\tau)$ be a holomorphic cusp form of weight $k$ for the full
modular group and let $a_{f}(n)$ be its Fourier coefficients. Also,
assume that $p,q$ are integers such that $(p,q)=1$ and consider
the Dirichlet series
\begin{equation}
L_{f}(s,p/q):=\sum_{n=1}^{\infty}\frac{a_{f}(n)\,e^{\frac{2\pi ip}{q}n}}{n^{s}},\,\,\,\,\,\text{Re}(s)>\frac{k+1}{2}.\label{periodic Cusp form L function}
\end{equation}

Then $L_{f}(s,p/q)$ can be continued analytically as an entire function
satisfying Hecke's functional equation 
\begin{equation}
\left(\frac{2\pi}{q}\right)^{-s}\Gamma(s)\,L_{f}\left(s,\frac{p}{q}\right)=(-1)^{k/2}\,\left(\frac{2\pi}{q}\right)^{-(k-s)}\Gamma(k-s)\,L_{f}\left(k-s,\,-\overline{p}/q\right),\label{Hecke functional equation periodic Cusp}
\end{equation}
where $\overline{p}$ is such that $p\,\overline{p}\equiv1 \mod\,q$.\bigskip{}
\end{lemma}

\bigskip{}

Throughout this paper we will also require the asymptotic formulas
for the Bessel functions. It can be seen in {[}\cite{watson_bessel}, pp. 199-202{]}
that $J_{\nu}(z)$, $I_{\nu}(z)$ and $K_{\nu}(z)$ satisfy the asymptotic
expansions:
\begin{equation}
J_{\nu}(z)\sim\left(\frac{2}{\pi z}\right)^{1/2}\left\{ \cos\left(z-\frac{\pi\nu}{2}-\frac{\pi}{4}\right)\,\sum_{n=0}^{\infty}\frac{(-1)^{n}(\nu,2n)}{(2z)^{2n}}-\sin\left(z-\frac{\pi\nu}{2}-\frac{\pi}{4}\right)\,\sum_{n=0}^{\infty}\frac{(-1)^{n}(\nu,2n+1)}{(2z)^{2n+1}}\right\} ,\label{asymptotic formula bessel function first kind}
\end{equation}

\begin{equation}
I_{\nu}(z)\sim\frac{e^{z}}{\sqrt{2\pi z}}\,\sum_{n=0}^{\infty}\frac{(-1)^{n}(\nu,n)}{(2z)^{n}}+\frac{e^{-z\pm(\nu+\frac{1}{2})\pi i}}{\sqrt{2\pi z}}\,\sum_{n=0}^{\infty}\frac{(\nu,n)}{(2z)^{n}},\label{asymptotic I}
\end{equation}

\begin{equation}
K_{\nu}(z)\sim\left(\frac{\pi}{2z}\right)^{\frac{1}{2}}e^{-z}\,\sum_{n=0}^{\infty}\frac{(\nu,n)}{(2z)^{n}}.\label{asymptotic K}
\end{equation}

Formulas (\ref{asymptotic formula bessel function first kind}) and
(\ref{asymptotic K}) are valid in the limit $|z|\rightarrow\infty$,
$|\arg(z)|<\pi$. The asymptotic expansion (\ref{asymptotic I}) is
taken with the plus sign if $-\frac{\pi}{2}<\arg(z)<\frac{3\pi}{2}$
and with minus sign if $-\frac{3\pi}{2}<\arg(z)<\frac{\pi}{2}$ (see
also [\cite{NIST}, pp. 223, 249 and 252, relations 10.7.3, 10.25.3,
10.30.4, 10.30.5]). Here, $(\nu,n)$ stands for the Hankel symbol:

\begin{equation}
(\nu,n):=\frac{(-1)^{n}\cos(\pi\nu)}{\pi\,n!}\,\Gamma\left(\frac{1}{2}+\nu+n\right)\,\Gamma\left(\frac{1}{2}-\nu+n\right).\label{Hankel symmmbol}
\end{equation}

\bigskip{}

In the sequel, we will employ at some points well-known integral representations
for the Bessel functions given above, as well as to Kummer's function.
For a matter of clarity in our exposition, we will invoke them solely
when they are needed.

\begin{center}\section{A general theta transformation formula involving Bessel functions} \label{section 2} \end{center}
To prove all the theorems given at the introduction, we need to develop
a new summation formula which generalizes Jacobi's formula (\ref{transformation formula Jacobi with parameter z})
with respect to the parameter $r$ in the functional equation (\ref{Hecke Dirichlet series Functional}).

\bigskip{}

In the same way that the ``modular relation'' (\ref{Bochner Modular relation at intro})
acts as a generalization of (\ref{Transformation formula Jacobi theta}),
in this section we will try to seek which kind of modular transformation
generalizes (\ref{transformation formula Jacobi with parameter z}).

\bigskip{}

We will restrict ourselves to Dirichlet series belonging to the class
$\mathcal{A}$ (see definition \ref{class A} above) although the study given
in this section also works for Dirichlet series with a more general
distribution of singularities. We will need to find a way to study
a summation formula for
\begin{equation}
\sum_{n=1}^{\infty}a(n)\,\lambda_{n}^{\frac{1-r}{2}}\,e^{-\alpha\lambda_{n}}\,J_{r-1}(\beta\,\sqrt{\lambda_{n}}),\,\,\,\,r>0,\label{Hecke series to evaluate}
\end{equation}
where $a(n)$ and $\lambda_{n}$ are given in definition \ref{Hecke and Bochner def}.  Here,
$r$ is the parameter appearing in Hecke's functional equation (\ref{Hecke Dirichlet series Functional}).

\bigskip{}

This study employs an integral representation, known as Hankel's formula (see [\cite{bucholz}, p. 8, eq. (15)], [\cite{handbook_marichev}, p.155, eq.(3.10.3.2)] and also [\cite{roy}, pp. 221-222] for a brief proof),
\begin{equation}
\intop_{0}^{\infty}t^{s-1}e^{-p^{2}t^{2}}J_{\nu}(at)\,dt=\frac{\Gamma\left(\frac{s+\nu}{2}\right)a^{\nu}}{2^{\nu+1}p^{s+\nu}\Gamma(\nu+1)}\,_{1}F_{1}\left(\frac{s+\nu}{2};\,\nu+1;\,-\frac{a^{2}}{4p^{2}}\right),\label{Hankel formula 1F1}
\end{equation}
valid for $\text{Re}(s)>-\text{Re}(\nu)$, $\text{Re}(p)>0$, $|\arg(a)|<\pi$.

\bigskip{}

Clearly, for $\nu=r-1$ and $\text{Re}(s)>1-r$, $\text{Re}(\alpha)>0$
and $|\arg(\beta)|<\pi$, the following particular case of the previous
formula holds
\begin{equation}
\intop_{0}^{\infty}t^{s-1}e^{-\alpha t}J_{r-1}(\beta\sqrt{t})\,dt=\frac{\Gamma\left(s+\frac{r-1}{2}\right)\beta^{r-1}}{2^{r-1}\alpha^{s+\frac{r-1}{2}}\Gamma(r)}\,_{1}F_{1}\left(s+\frac{r-1}{2};\,r;\,-\frac{\beta^{2}}{4\alpha}\right).\label{Direct Hankel}
\end{equation}

This suggests the inversion formula, now valid for $\text{Re}(\alpha)>0$,
$\beta\in\mathbb{C}$ and $\sigma,\,x>0$,
\begin{equation}
e^{-\alpha x}J_{r-1}(\beta\sqrt{x})=\frac{1}{\Gamma(r)}\,\left(\frac{\beta^{2}}{4\alpha}\right)^{\frac{r-1}{2}}\,\frac{1}{2\pi i}\,\intop_{\sigma-i\infty}^{\sigma+i\infty}\Gamma\left(s+\frac{r-1}{2}\right)\,_{1}F_{1}\left(s+\frac{r-1}{2};\,r;\,-\frac{\beta^{2}}{4\alpha}\right)\,(\alpha x)^{-s}ds,\label{Mellin representation J e}
\end{equation}
which might be proved by invoking directly the power series expansion
of Kummer's function and using the power series for the Bessel function
$J_{\nu}(z)$.

\bigskip{}

To represent the series given in (\ref{Hecke series to evaluate}),
we need an asymptotic formula for the confluent hypergeometric function on the right-hand side of (\ref{Mellin representation J e}) which is 
valid when $|s|\rightarrow\infty$. Following the reasoning in {[}\cite{Dixit_theta},
p. 379{]}, recall that the Whittaker function $M_{\lambda,\mu}(z)$
has the asymptotic formula {[}\cite{ERDELIY_TRANSCENDENTAL}, Vol. , p.274{]}, {[}\cite{NIST}, p.
341, eq. (13.21.1){]}
\[
M_{\lambda,\mu}(z)=\frac{z^{1/4}}{\sqrt{\pi}}\lambda^{-\mu-\frac{1}{4}}\Gamma\left(2\mu+1\right)\,\cos\left(2\sqrt{\lambda z}-\frac{\pi}{4}-\mu\pi\right)+O\left(|\lambda|^{-\mu-\frac{3}{4}}\right)
\]
as $|\lambda|\rightarrow\infty$ and $z$ such that $|\arg(\lambda z)|<2\pi$.
Furthermore, by its definition,
\begin{equation}
M_{\lambda,\mu}(z)=z^{\mu+\frac{1}{2}}e^{-z/2}\,_{1}F_{1}\left(\mu-\lambda+\frac{1}{2};\,2\mu+1;\,z\right),\label{definition whitttaker}
\end{equation}
we see that, once we replace $\lambda$ by $\frac{1}{2}-s$, $\mu$
by $\frac{r-1}{2}$ and $z$ by $-\frac{\beta^{2}}{4\alpha}$,
\begin{equation}
\,_{1}F_{1}\left(s+\frac{r-1}{2};\,r;\,-\frac{\beta^{2}}{4\alpha}\right)=\frac{e^{-\frac{\beta^{2}}{8\alpha}}\Gamma(r)}{\sqrt{\pi}}\,\left(\frac{\beta^{2}}{4\alpha}\left(s-\frac{1}{2}\right)\right)^{\frac{1}{4}-\frac{r}{2}}\,\cos\left(\frac{\beta}{\alpha}\sqrt{s-\frac{1}{2}}+\frac{\pi}{4}-\pi\frac{r}{2}\right)+O\left(|s|^{-\frac{r}{2}-\frac{1}{4}}\right),\label{asymptotic whittaker sense}
\end{equation}
as $|s|\rightarrow\infty$ and $|\arg\left(\frac{\beta^{2}}{4\alpha}\left(s-\frac{1}{2}\right)\right)|<2\pi$.

\bigskip{}

For the purposes of this paper, it will be enough to invoke a simpler bound of the
form,
\begin{equation}
\left|\,_{1}F_{1}\left(\sigma+it+\frac{r-1}{2};\,r;\,-\frac{\beta^{2}}{4\alpha}\right)\right|\leq A\,|t|^{\frac{1}{4}-\frac{r}{2}}\,e^{B\,\sqrt{|t|}}+O\left(|t|^{-\frac{r}{2}-\frac{1}{4}}\right),\,\,\,\,|t|\rightarrow\infty,
\end{equation}
where $A$ and $B$ are positive constants depending only on $\alpha,\,\beta$
and $r$ and $c\leq \sigma \leq d$.

\bigskip{}

Since we will use the Mellin inverse representation (\ref{Mellin representation J e}),
the confluent hypergeometric function lying in its kernel will have to satisfy
a transformation formula compatible with Hecke's functional equation
(\ref{Hecke Dirichlet series Functional}). The following formula
due to Kummer {[}\cite{roy}, p. 191, eq. (4.1.11){]} will be useful
\begin{equation}
_{1}F_{1}\left(a;\,c;\,x\right)=e^{x}\,_{1}F_{1}\left(c-a;\,c;\,-x\right).\label{Kummer confluent transformation}
\end{equation}

\newpage

\begin{center}
\subsection{General identities for Dirichlet series}\label{proof of summation}    
\end{center}

We are now ready to prove the following Theorem, which has independent
interest because it is an identity concerning general Dirichlet
series.

\begin{theorem} \label{summation formula with 1F1}
Let $\phi(s)$ and $\psi(s)$ be the pair of Dirichlet series satisfying
the functional equation (\ref{Hecke Dirichlet series Functional}).
Furthermore, assume that $\phi(s)\in\mathcal{A}$ and let us denote
by $\rho$ the residue that $\phi(s)$ possesses at its pole $s=r$.
    
\end{theorem}

Then, for $\text{Re}(\alpha)>0$ and $\beta\in\mathbb{C}$, the following transformation
formula holds
\begin{equation}
\sum_{n=1}^{\infty}a(n)\,\lambda_{n}^{\frac{1-r}{2}}\,e^{-\alpha\lambda_{n}}\,J_{r-1}(\beta\,\sqrt{\lambda_{n}})=\frac{\phi(0)\beta^{r-1}}{2^{r-1}\Gamma(r)}+\frac{\beta^{r-1}\rho}{2^{r-1}\alpha^{r}}\,e^{-\frac{\beta^{2}}{4\alpha}}+\frac{e^{-\frac{\beta^{2}}{4\alpha}}}{\alpha}\,\sum_{n=1}^{\infty}b(n)\,\mu_{n}^{\frac{1-r}{2}}\,e^{-\mu_{n}/\alpha} I_{r-1}\left(\frac{\beta\sqrt{\mu_{n}}}{\alpha}\right). \label{final formula for 1f1 theorem}
\end{equation}

\begin{remark}
The series on the left-hand side is obviously convergent under the
conditions established for the sequences $a(n)$ and $\lambda_{n}$
in definition \ref{Hecke and Bochner def}. The convergence of the series on the right-hand
side comes from the Hankel expansion for the Modified Bessel function
(\ref{asymptotic I}), which gives $|I_{\nu}(z)|\ll_{\nu}e^{|\text{Re}(z)|}/\sqrt{2\pi|z|}$
as $|z|\rightarrow\infty$.    
\end{remark}

\begin{proof}
Pick $\mu>\max\left\{ \sigma_{a}-\frac{r}{2}+\frac{1}{2},\sigma_{b}-\frac{r}{2}+\frac{1}{2},\frac{r+1}{2},\frac{3r}{2}-\frac{1}{2}\right\} $.
By (\ref{asymptotic whittaker sense}) and (\ref{Mellin representation J e}),
and arguing by absolute convergence, we can write the series (\ref{Hecke series to evaluate})
as
\begin{equation}
\sum_{n=1}^{\infty}a(n)\,\lambda_{n}^{\frac{1-r}{2}}\,e^{-\alpha\lambda_{n}}\,J_{r-1}(\beta\,\sqrt{\lambda_{n}})=\frac{1}{2\pi i}\,\intop_{\mu-i\infty}^{\mu+i\infty}\phi\left(s+\frac{r-1}{2}\right)\,\frac{\Gamma\left(s+\frac{r-1}{2}\right)\beta^{r-1}}{2^{r-1}\alpha^{s+\frac{r-1}{2}}\Gamma(r)}\,_{1}F_{1}\left(s+\frac{r-1}{2};\,r;\,-\frac{\beta^{2}}{4\alpha}\right)\,ds.\label{First representation in theorem}
\end{equation}

We will now integrate along a positively oriented rectangular contour
$\mathcal{R}$ containing the vertices $\mu\pm iT$ and $r-\mu\pm iT$
for $T>0$. By the choice of $\mu$, we know that the line $\text{Re}(s)=r-\mu$ is located at the left of the line $\text{Re}(s)=\mu$. By using the residue theorem, we have
\begin{align}
\frac{1}{2\pi i}\,\intop_{\mu-iT}^{\mu+iT}\phi\left(s+\frac{r-1}{2}\right)\,\frac{\Gamma\left(s+\frac{r-1}{2}\right)\beta^{r-1}}{2^{r-1}\alpha^{s+\frac{r-1}{2}}\Gamma(r)}\,_{1}F_{1}\left(s+\frac{r-1}{2};\,r;\,-\frac{\beta^{2}}{4\alpha}\right)\,ds & =\nonumber \\
=\frac{1}{2\pi i}\left[\intop_{r-\mu+iT}^{\mu+iT}+\intop_{r-\mu-iT}^{r-\mu+iT}+\intop_{\mu-iT}^{r-\mu-iT}\right]\,\phi\left(s+\frac{r-1}{2}\right)\,\frac{\Gamma\left(s+\frac{r-1}{2}\right)\beta^{r-1}}{2^{r-1}\alpha^{s+\frac{r-1}{2}}\Gamma(r)}\,_{1}F_{1}\left(s+\frac{r-1}{2};\,r;\,-\frac{\beta^{2}}{4\alpha}\right)\,ds\nonumber \\
+\sum_{\rho\in\mathcal{R}}\text{Res}_{s=\rho}\left\{ \phi\left(s+\frac{r-1}{2}\right)\,\frac{\Gamma\left(s+\frac{r-1}{2}\right)\beta^{r-1}}{2^{r-1}\alpha^{s+\frac{r-1}{2}}\Gamma(r)}\,_{1}F_{1}\left(s+\frac{r-1}{2};\,r;\,-\frac{\beta^{2}}{4\alpha}\right)\right\} .\label{Reisdue theorem at beginning}
\end{align}

Clearly, by Stirling's formula (\ref{preliminary stirling}), the convex estimates (\ref{Lindelof Phragmen for Hecke}) and (\ref{asymptotic whittaker sense}),
the integrals over the horizontal segments $\left[r-\mu\pm iT,\,\mu\pm iT\right]$
tend to zero as $T\rightarrow\infty$. Also, we know that $\phi(s)$
has a simple pole located at $s=r$ and, by Remark \ref{zeros integers remark}, has zeros located at the negative integers.

Since $_{1}F_{1}\left(z;\,r;\,-\frac{\beta^{2}}{4\alpha}\right)$
is an entire function in the variable $z$, the only poles that we
need to take into account in the sum above are located at the points $\rho=\frac{1-r}{2}$
and $\rho=\frac{r+1}{2}$. Denoting the integrand in (\ref{First representation in theorem})
by $\mathcal{I}_{\alpha,\beta}(s)$, we see that the residues at these
poles are respectively
\begin{equation}
\text{Res}_{s=\frac{r+1}{2}}\left\{ \mathcal{I}_{\alpha,\beta}(s)\right\} =\frac{\beta^{r-1}\rho}{2^{r-1}\alpha^{r}}\,e^{-\frac{\beta^{2}}{4\alpha}},\,\,\,\,\text{Res}_{s=\frac{1-r}{2}}\left\{ \mathcal{I}_{\alpha,\beta}(s)\right\} =\frac{\phi(0)\beta^{r-1}}{2^{r-1}\Gamma(r)}.\label{first residues first theorem}
\end{equation}

\bigskip{}

Letting $T\rightarrow\infty$ in (\ref{Reisdue theorem at beginning}),
we now need to evaluate the integral
\[
\frac{1}{\Gamma(r)}\,\left(\frac{\beta^{2}}{4\alpha}\right)^{\frac{r-1}{2}}\,\frac{1}{2\pi i}\,\intop_{r-\mu-i\infty}^{r-\mu+i\infty}\Gamma\left(s+\frac{r-1}{2}\right)\,\phi\left(s+\frac{r-1}{2}\right)\,_{1}F_{1}\left(s+\frac{r-1}{2};\,r;\,-\frac{\beta^{2}}{4\alpha}\right)\,\alpha^{-s}\,ds,
\]
which we do by appealing to the functional satisfied by $\phi(s)$ (\ref{Hecke Dirichlet series Functional}) and to Kummer's formula (\ref{Kummer confluent transformation}),
\begin{equation}
\,_{1}F_{1}\left(s+\frac{r-1}{2};\,r;\,-\frac{\beta^{2}}{4\alpha}\right)=e^{-\frac{\beta^{2}}{4\alpha}}\,_{1}F_{1}\left(\frac{r+1}{2}-s;\,r;\,\frac{\beta^{2}}{4\alpha}\right).\label{Kummer transformation}
\end{equation}

\bigskip{}

Performing such transformations and changing the variable back to
the region of absolute convergence of $\psi(s)$ (which is possible due to
the fact that we have chosen $\mu>\sigma_{b}-\frac{r}{2}+\frac{1}{2}$), we get
\begin{align}
\frac{1}{\Gamma(r)}\,\left(\frac{\beta^{2}}{4\alpha}\right)^{\frac{r-1}{2}}\,\frac{1}{2\pi i}\,\intop_{r-\mu-i\infty}^{r-\mu+i\infty}\Gamma\left(s+\frac{r-1}{2}\right)\,\phi\left(s+\frac{r-1}{2}\right)\,_{1}F_{1}\left(s+\frac{r-1}{2};\,r;\,-\frac{\beta^{2}}{4\alpha}\right)\,\alpha^{-s}\,ds\nonumber \\
=\frac{e^{-\frac{\beta^{2}}{4\alpha}}}{2\pi i\,\Gamma(r)}\,\left(\frac{\beta^{2}}{4\alpha}\right)^{\frac{r-1}{2}}\,\intop_{r-\mu-i\infty}^{r-\mu+i\infty}\Gamma\left(\frac{r+1}{2}-s\right)\,\psi\left(\frac{r+1}{2}-s\right)\,_{1}F_{1}\left(\frac{r+1}{2}-s;\,r;\,\frac{\beta^{2}}{4\alpha}\right)\,\alpha^{-s}\,ds\nonumber \\
=\frac{e^{-\frac{\beta^{2}}{4\alpha}}}{2\pi i\,\Gamma(r)}\,\left(\frac{\beta^{2}}{4\alpha}\right)^{\frac{r-1}{2}}\,\alpha^{-r}\,\intop_{\mu-i\infty}^{\mu+i\infty}\Gamma\left(s+\frac{1-r}{2}\right)\,\psi\left(s+\frac{1-r}{2}\right)\,_{1}F_{1}\left(s+\frac{1-r}{2};\,r;\,\frac{\beta^{2}}{4\alpha}\right)\,\alpha^{s}\,ds.\label{Final integral in the proof}
\end{align}

\bigskip{}

We now use the expansion of $\psi(s)$ as a Dirichlet series (\ref{representable as Dirichlet series in first definition ever})
together with the power series expansion for the confluent hypergeometric
function. Due to absolute convergence of both series, we can write
the right-hand side of (\ref{Final integral in the proof}) as
\begin{align}
e^{-\frac{\beta^{2}}{4\alpha}}\,\left(\frac{\beta^{2}}{4\alpha}\right)^{\frac{r-1}{2}}\,\alpha^{-r}\,\sum_{n=1}^{\infty}b(n)\,\mu_{n}^{\frac{r-1}{2}}\,\frac{1}{2\pi i}\,\intop_{\mu-i\infty}^{\mu+i\infty}\frac{\Gamma\left(s+\frac{1-r}{2}\right)}{\Gamma(r)}\,_{1}F_{1}\left(s+\frac{1-r}{2};\,r;\,\frac{\beta^{2}}{4\alpha}\right)\,(\alpha/\mu_{n})^{s}ds\nonumber \\
=e^{-\frac{\beta^{2}}{4\alpha}}\,\left(\frac{\beta^{2}}{4\alpha}\right)^{\frac{r-1}{2}}\,\alpha^{-r}\,\sum_{n=1}^{\infty}b(n)\,\mu_{n}^{\frac{r-1}{2}}\,\frac{1}{2\pi i}\,\intop_{\mu-i\infty}^{\mu+i\infty}\sum_{k=0}^{\infty}\frac{\Gamma\left(s+\frac{1-r}{2}+k\right)}{\Gamma(r+k)\,k!}\,\frac{\beta^{2k}}{4^{k}\alpha^{k}}\,(\alpha/\mu_{n})^{s}ds\nonumber \\
=e^{-\frac{\beta^{2}}{4\alpha}}\,\left(\frac{\beta}{2}\right)^{r-1}\,\alpha^{-r}\,\sum_{n=1}^{\infty}b(n)\,\sum_{k=0}^{\infty}\frac{1}{\Gamma(r+k)\,k!}\left(\frac{\beta^{2}\mu_{n}}{4\alpha^{2}}\right)^{k}\,\frac{1}{2\pi i}\,\intop_{\mu-i\infty}^{\mu+i\infty}\Gamma\left(z\right)\,(\alpha/\mu_{n})^{z}dz\nonumber \\
=e^{-\frac{\beta^{2}}{4\alpha}}\,\left(\frac{\beta}{2}\right)^{r-1}\,\alpha^{-r}\,\sum_{n=1}^{\infty}b(n)e^{-\mu_{n}/\alpha}\,\sum_{k=0}^{\infty}\frac{1}{\Gamma(r+k)\,k!}\left(\frac{\beta\sqrt{\mu_{n}}}{2\alpha}\right)^{2k}.\label{simplification almost final}
\end{align}

Recalling the definition of the modified Bessel function of the first
kind {[}\cite{temme_book}, p. 233, eq. (9.28){]}, 
\[
\left(\frac{x}{2}\right)^{-\nu}\,I_{\nu}(x)=\sum_{m=0}^{\infty}\frac{1}{m!\,\Gamma(m+\nu+1)}\left(\frac{x}{2}\right)^{2m},
\]
we see immediately that the latter infinite series in (\ref{simplification almost final})
is represented by the modified Bessel function of the first kind.
This proves (\ref{final formula for 1f1 theorem}).   
\end{proof}

\bigskip{}

Since the class $\mathcal{B}$ can be included in the class $\mathcal{A}$,
i.e., if $\phi(s)\in\mathcal{B}$ then $\phi(2s-\delta)\in\mathcal{A}$
(see Remark \ref{Bochner Hecke Remark}), we can deduce the following formula obtained in [\cite{rysc_I}, Lemma 3.1.].

\begin{corollary} \label{corollary about Bochner theta}
Let $\phi(s)$ be a Dirichlet series representable as (\ref{representable as Dirichlet series in first definition ever})
and belonging to the class $\mathcal{B}$ with $\delta=0$. Then,
for any $x>0$ and $\text{Re}(\alpha)>0$, $\beta\in\mathbb{C}$,
the following identity holds
\begin{equation}
\sum_{n=1}^{\infty}a(n)\,e^{-\alpha\lambda_{n}^{2}x^{2}}\,\cos\left(\beta\lambda_{n}x\right)=\phi(0)+\sqrt{\frac{\pi}{\alpha}}\frac{\rho}{2\,x}\,e^{-\frac{\beta^{2}}{4\alpha}}+\frac{e^{-\frac{\beta^{2}}{4\alpha}}}{\sqrt{\alpha}\,x}\,\sum_{n=1}^{\infty}b(n)\,e^{-\frac{\mu_{n}^{2}}{\alpha x^{2}}}\cosh\left(\frac{\beta\,\mu_{n}}{\alpha x}\right).\label{formula in RYSC I}
\end{equation}

Moreover, if $\phi(s)$ belongs to the class $\mathcal{B}$ with $\delta=1$,
the analogous formula takes place
\begin{equation}
\sum_{n=1}^{\infty}a(n)\,e^{-\alpha\lambda_{n}^{2}x^{2}}\,\sin\left(\beta\lambda_{n}x\right)=\frac{e^{-\frac{\beta^{2}}{4\alpha}}}{\sqrt{\alpha}x}\,\sum_{n=1}^{\infty}b(n)\,e^{-\frac{\mu_{n}^{2}}{\alpha\,x^{2}}}\,\sinh\left(\frac{\beta\mu_{n}}{\alpha\,x}\right).\label{bochner Odd theta relation}
\end{equation}

\end{corollary}

\begin{proof}
Let $\phi^{\prime}(s)=\phi(2s-\delta)$, $\delta\in\{0,1\}$. Then
$\phi^{\prime}(s)$ satisfies Hecke's functional equation
\begin{equation}
\Gamma\left(s\right)\,\phi^{\prime}(s)=\Gamma\left(\frac{1}{2}+\delta-s\right)\,\psi^{\prime}\left(\frac{1}{2}+\delta-s\right)\label{shifted Bochner Hecke Functional}
\end{equation}
and belongs to the class $\mathcal{A}$, with (at most) a simple pole
located at $s=\frac{1}{2}$ (if $\delta=0$) with residue $\rho/2$
(see Remark \ref{Bochner Hecke Remark} above).

\bigskip{}

In general, if $\phi(s)$ and $\psi(s)\in\mathcal{B}$ are representable
as (\ref{representable as Dirichlet series in first definition ever}),
$\phi^{\prime}(s)$ and $\psi^{\prime}(s)$ can be written in the
form
\[
\phi^{\prime}(s)=\sum_{n=1}^{\infty}\frac{a(n)\,\lambda_{n}^{\delta}}{\lambda_{n}^{2s}},\,\,\,\,\psi^{\prime}(s)=\sum_{n=1}^{\infty}\frac{b(n)\,\mu_{n}^{\delta}}{\mu_{n}^{2s}},\,\,\,\,\,\text{Re}(s)>\sigma_{a}/2.
\]

From (\ref{shifted Bochner Hecke Functional}), we can apply the summation
formula (\ref{final formula for 1f1 theorem}) replacing there $a(n)$
by $a(n)\,\lambda_{n}^{\delta}$, $\lambda_{n}$ by $\lambda_{n}^{2}$,
$b(n)$ by $b(n)\,\mu_{n}^{\delta}$, $\mu_{n}$ by $\mu_{n}^{2}$
and finally $\rho$ by $\rho/2$: this gives
\begin{equation}
\sum_{n=1}^{\infty}a(n)\,\lambda_{n}^{\frac{1}{2}}\,e^{-\alpha\lambda_{n}^{2}}\,J_{\delta-\frac{1}{2}}(\beta\,\lambda_{n}) =\sqrt{\frac{2}{\pi\beta}}\left\{ \phi(0)+\frac{\rho}{2}\sqrt{\frac{\pi}{\alpha}}\,e^{-\frac{\beta^{2}}{4\alpha}}\right\} (1-\delta)+\,\frac{e^{-\frac{\beta^{2}}{4\alpha}}}{\alpha}\,\sum_{n=1}^{\infty}b(n)\, \mu_{n}^{\frac{1}{2}}\,e^{-\mu_{n}^{2}/\alpha}\,I_{\delta-\frac{1}{2}}\left(\frac{\beta\mu_{n}}{\alpha}\right).\label{identity derived from straightforward substitution}
\end{equation}

Using now the particular values for the Bessel functions {[}\cite{temme_book}, p.248{]}, 
\begin{equation}
J_{\delta-\frac{1}{2}}(x)=\begin{cases}
\sqrt{\frac{2}{\pi x}}\,\cos(x) & \text{if }\delta=0\\
\sqrt{\frac{2}{\pi x}}\,\sin(x) & \text{if }\delta=1
\end{cases},\,\,\,\,I_{\delta-\frac{1}{2}}(x)=\begin{cases}
\sqrt{\frac{2}{\pi x}}\,\cosh(x) & \text{if }\delta=0\\
\sqrt{\frac{2}{\pi x}}\,\sinh(x) & \text{if }\delta=1
\end{cases},\label{particular cases formula Cosine}
\end{equation}
we obtain immediately the identities (\ref{formula in RYSC I}) and
(\ref{bochner Odd theta relation}) after replacing $\alpha$ and
$\beta$ in (\ref{identity derived from straightforward substitution})
respectively by $\alpha x^{2}$ and $\beta x$, $x>0$.    
\end{proof}

\begin{remark}\label{heading in right direction}
In particular, since $\phi(s):=\pi^{-s/2}\zeta(s)\in\mathcal{B}$
by Remark \ref{Riemann Bochner Remark}, an application of the identity (\ref{formula in RYSC I})
gives Jacobi's formula (\ref{transformation formula Jacobi with parameter z}).
This seems to be the first indication that the general formula (\ref{final formula for 1f1 theorem})
will prove to be the rightful analogue of (\ref{transformation formula Jacobi with parameter z}).
\end{remark}

Identities (\ref{formula in RYSC I}) and (\ref{bochner Odd theta relation})
were employed in [\cite{rysc_I}, Theorem 3.1.] to derive a generalization
of the Selberg-Chowla formula for Dirichlet series in the class $\mathcal{B}$.
Based on the ``theta-like'' features of the series appearing in
(\ref{formula in RYSC I}) and (\ref{bochner Odd theta relation}),
we now give a definition which introduces a new general analogue of
the theta function.

\bigskip{}

\begin{definition} \label{generalized theta function}
Let $\text{Re}(x)>0$, $y\in\mathbb{C}$ and assume that $\phi(s)$
is a Dirichlet series satisfying Hecke's functional equation (\ref{Hecke Dirichlet series Functional}).
We define the generalized $\theta-$function as the following series
\begin{equation}
\Phi(x,y)=2^{r-1}\Gamma(r)\,y^{1-r}\,\sum_{n=1}^{\infty}a(n)\,\lambda_{n}^{\frac{1-r}{2}}\,e^{-\lambda_{n}x}\,J_{r-1}\left(\sqrt{\lambda_{n}}\,y\right).\label{generalized theta as definition}
\end{equation}
    \end{definition}

\bigskip{}

Our next corollary, which actually consists in rewriting (\ref{final formula for 1f1 theorem})
in a compact form, gives a transformation formula for $\Phi(x,y)$.
Moreover, just like Dixit's formula (\ref{Dixit integral formula jacobi psi}),
we are able to compare this modular relation with an integral involving
$\phi(s)$ and the confluent hypergeometric function. 

\begin{corollary}\label{integral representation theta}
Let $\text{Re}(x)>0$, $y\in\mathbb{C}$ and assume that $\phi(s)\in\mathcal{A}$.
Let us consider the generalized Theta functions:
\begin{equation}
\Phi(x,y):=2^{r-1}\Gamma(r)\,y^{1-r}\,\sum_{n=1}^{\infty}a(n)\,\lambda_{n}^{\frac{1-r}{2}}\,e^{-\lambda_{n}x}\,J_{r-1}\left(\sqrt{\lambda_{n}}\,y\right)\label{Definition generalized Theta}
\end{equation}
and
\begin{equation}
\Psi\left(x,y\right):=2^{r-1}\Gamma(r)\,y^{1-r}\,\sum_{n=1}^{\infty}b(n)\,\mu_{n}^{\frac{1-r}{2}}\,e^{-\mu_{n}x}\,J_{r-1}\left(\sqrt{\mu_{n}}\,y\right).\label{Definition generalized auxiliary Theta}
\end{equation}

Then $\Phi(x,y)$ (resp. $\Psi(x,y)$) has the following properties:
\begin{enumerate}
\item It is entire in $y\in\mathbb{C}$ and analytic in $\text{Re}(x)>0$;
\item It satisfies the transformation formula 
\begin{equation}
\Phi(x,y)=\phi(0)+\frac{\rho}{x^{r}}\Gamma(r)\,e^{-\frac{y^{2}}{4x}}+\frac{e^{-\frac{y^{2}}{4x}}}{x^{r}}\,\Psi\left(\frac{1}{x},\,\frac{iy}{x}\right).\label{generalized Theta reflection}
\end{equation}
\end{enumerate}
Moreover, (\ref{generalized Theta reflection})
can be written in terms of an integral involving $\Gamma(s)\,\phi(s)\,_{1}F_{1}(s;\,r;\,-y^{2}/4x)$
at the critical line of $\phi(s)$, this is,
\begin{align}
x^{r/2}\Phi\left(x,y\right)-\frac{\rho\Gamma(r)}{x^{r/2}}\,e^{-\frac{y^{2}}{4x}} & =e^{-\frac{y^{2}}{4x}}x^{-r/2}\Psi\left(\frac{1}{x},\,i\,\frac{y}{x}\right)+\phi(0)\,x^{r/2}\nonumber \\
=\frac{1}{2\pi}\,\intop_{-\infty}^{\infty}\Gamma\left(\frac{r}{2}+it\right)\,\phi\left(\frac{r}{2}+it\right) & \,_{1}F_{1}\left(\frac{r}{2}+it;\,r;\,-\frac{y^{2}}{4x}\right)x^{-it}\,dt.\label{Integral representation Dixit Style}
\end{align}

\end{corollary}

\begin{proof}
We first show that $\Phi(x,y)$ satisfies item 1., i.e., that is entire
in $y\in\mathbb{C}$ and analytic in the half-plane $\text{Re}(x)>0$.
We divide this proof in two cases: $r>\frac{1}{2}$ or $0<r\leq\frac{1}{2}$.
For the first case, if $\text{Re}(x)=\sigma\ge\sigma_{0}>0$ and $y$
belongs to a bounded subset of $\mathbb{C}$, contained in $|y|\leq M$
say, then the series defining $\Phi(x,y)$ is absolutely and uniformly
convergent, since, for $r>\frac{1}{2}$, 
\begin{align*}
|\Phi(x,y)| & \leq\Gamma\left(r\right)\,\sum_{n=1}^{\infty}|a(n)|e^{-\lambda_{n}\sigma_{0}}\left|\left(\frac{\sqrt{\lambda_{n}}\,y}{2}\right)^{1-r}\,J_{r-1}\left(\sqrt{\lambda_{n}}\,y\right)\right|\\
 & \leq\frac{2\Gamma(r)}{\sqrt{\pi}\Gamma\left(r-\frac{1}{2}\right)}\,\sum_{n=1}^{\infty}|a(n)|e^{-\lambda_{n}\sigma_{0}}\,\intop_{0}^{1}\left(1-t^{2}\right)^{r-\frac{3}{2}}\,\left|\cos\left(\sqrt{\lambda_{n}}\,y\,t\right)\right|\,dt\\
 & \leq\sum_{n=1}^{\infty}|a(n)|\,\exp\left(-\lambda_{n}\sigma_{0}+\sqrt{\lambda_{n}}\,M\right)<\infty,
\end{align*}
where in the second inequality we have used the well-known Fourier
representation (due to Poisson) [\cite{NIST}, p. 224, eq. (10.9.4)]
\begin{equation}
\left(\frac{z}{2}\right)^{-\nu}\,J_{\nu}(z)=\frac{2}{\sqrt{\pi}\Gamma\left(\nu+\frac{1}{2}\right)}\,\intop_{0}^{1}\left(1-t^{2}\right)^{\nu-\frac{1}{2}}\cos(zt)\,dt,\,\,\,\,\,\,\,\text{Re}(\nu)>-\frac{1}{2},\,\,z\in\mathbb{C},\label{Poisson Bessel}
\end{equation}
as well as the fact that the Dirichlet series $\phi(s)$ converges
in some half-plane. 

\bigskip{}

Note that (\ref{Poisson Bessel}) can only be invoked for $J_{r-1}(z)$
if $r>\frac{1}{2}$. To address the case where $0<r\leq\frac{1}{2}$,
we just bound trivially $\Phi(x,y)$, which gives
\begin{align*}
|\Phi(x,y)| & \leq\Gamma\left(r\right)\,\sum_{n=1}^{\infty}|a(n)|e^{-\lambda_{n}\sigma_{0}}\left|\left(\frac{\sqrt{\lambda_{n}}\,y}{2}\right)^{1-r}\,J_{r-1}\left(\sqrt{\lambda_{n}}\,y\right)\right|\\
 & \leq\Gamma\left(r\right)\,\sum_{n=1}^{\infty}|a(n)|e^{-\lambda_{n}\sigma_{0}}\,\left(\frac{\sqrt{\lambda_{n}}\,|y|}{2}\right)^{1-r}I_{r-1}\left(\sqrt{\lambda_{n}}\,|y|\right)<\infty,
\end{align*}
where we have used the well-known bound $|J_{\nu}(z)|\leq I_{\nu}(|z|)$,
valid for $\nu\in\mathbb{R}$. The finitude of the last series comes
from the well-known asymptotic estimate (\ref{asymptotic I}) for $I_{\nu}(x)$, $x>0$ {[}\cite{NIST}, p. 252, eq. (10.30(ii)){]}, which gives 
\begin{equation}
\left(\frac{\sqrt{\lambda_{n}}\,|y|}{2}\right)^{1-r}\,I_{r-1}\left(\sqrt{\lambda_{n}}\,|y|\right)\ll\,\exp\left(\sqrt{\lambda_{n}}\,M\right)\,\lambda_{n}^{\frac{1}{4}-\frac{r}{2}}|y|^{\frac{1}{2}-r},\,\,\,\,\,n\rightarrow\infty.\label{asymptotic useful convergence}
\end{equation}

\bigskip{}

Therefore, for each fixed $x$ in the half-plane $\text{Re}(x)>0$,
the function $\Phi(x,\cdot)$ is entire and for each fixed $y\in\mathbb{C}$
the function $\Phi(\cdot,y)$ is holomorphic in the right half-plane
$\text{Re}(x)>0$. This proves item 1. above. 

\bigskip{}

Now, we prove the integral representation (\ref{Integral representation Dixit Style}),
which immediately gives (\ref{generalized Theta reflection}). We
start our proof with the integral on the right-hand side of (\ref{Integral representation Dixit Style}),
transforming it into a contour integral along the critical line $\text{Re}(s)=\frac{r}{2}$,
\begin{align*}
\frac{y^{r-1}x^{-r/2}}{2^{r}\pi\,\Gamma(r)}\,\intop_{-\infty}^{\infty}\Gamma\left(\frac{r}{2}+it\right)\,\phi\left(\frac{r}{2}+it\right)\,_{1}F_{1}\left(\frac{r}{2}+it;\,r;\,-\frac{y^{2}}{4x}\right)x^{-it}\,dt=\\
=\frac{y^{r-1}}{2^{r-1}\,\Gamma(r)}\,\frac{1}{2\pi i}\,\intop_{\frac{r}{2}-i\infty}^{\frac{r}{2}+i\infty}\Gamma\left(s\right)\,\phi\left(s\right)\,_{1}F_{1}\left(s;\,r;\,-\frac{y^{2}}{4x}\right)x^{-s}\,ds=\\
=\frac{y^{r-1}}{2^{r-1}\,\Gamma(r)}\,\frac{e^{-\frac{y^{2}}{4x}}}{2\pi i}\,\intop_{\frac{r}{2}-i\infty}^{\frac{r}{2}+i\infty}\Gamma\left(s\right)\,\phi\left(s\right)\,_{1}F_{1}\left(r-s;\,r;\,\frac{y^{2}}{4x}\right)x^{-s}\,ds,
\end{align*}
where we have used Kummer's formula (\ref{Kummer transformation}).
We now change the line of integration to $\text{Re}(s)=\sigma_{a}+\delta$,
for some positive $\delta$. Doing so we integrate over a (positively
oriented) rectangular contour whose vertices are $r/2\pm iT$ and
$\sigma_{a}+\delta\pm iT$, where $T$ is a (sufficiently large) positive
real number. The integrals along the horizontal segments vanish as
$T\rightarrow\infty$ due to Stirling's formula (\ref{preliminary stirling})
and (\ref{asymptotic whittaker sense}).

\bigskip{}

Doing this shift, due to the simple pole of $\phi(s)$ located at
$s=r$, we find the residue
\[
\text{Res}_{s=r}\left\{ \Gamma\left(s\right)\,\phi\left(s\right)\,_{1}F_{1}\left(r-s;\,r;\,\frac{y^{2}}{4x}\right)x^{-s}\right\} =\Gamma(r)\rho\,x^{-r}.
\]

Henceforth, since $\phi(s)$ converges absolutely for $\text{Re}(s)>\sigma_{a}$,
we have by absolute convergence
\begin{align*}
\frac{y^{r-1}}{2^{r-1}\,\Gamma(r)}\,\frac{e^{-\frac{y^{2}}{4x}}}{2\pi i}\,\intop_{\frac{r}{2}-i\infty}^{\frac{r}{2}+i\infty}\Gamma\left(s\right)\,\phi\left(s\right)\,_{1}F_{1}\left(r-s;\,r;\,\frac{y^{2}}{4x}\right)x^{-s}\,ds\\
=\frac{y^{r-1}}{2^{r-1}\,\Gamma(r)}\,\frac{e^{-\frac{y^{2}}{4x}}}{2\pi i}\,\intop_{\sigma_{a}+\delta-i\infty}^{\sigma_{a}+\delta+i\infty}\Gamma\left(s\right)\,\phi\left(s\right)\,_{1}F_{1}\left(r-s;\,r;\,\frac{y^{2}}{4x}\right)x^{-s}\,ds-\frac{y^{r-1}\rho}{2^{r-1}x^{r}}\,e^{-\frac{y^{2}}{4x}}\\
=\frac{y^{r-1}}{2^{r-1}\,\Gamma(r)}\,\frac{e^{-\frac{y^{2}}{4x}}}{2\pi i}\,\sum_{n=1}^{\infty}a(n)\,\intop_{\sigma_{a}+\delta-i\infty}^{\sigma_{a}+\delta+i\infty}\Gamma\left(s\right)\,_{1}F_{1}\left(r-s;\,r;\,\frac{y^{2}}{4x}\right)(x\lambda_{n})^{-s}\,ds-\frac{y^{r-1}\rho}{2^{r-1}x^{r}}\,e^{-\frac{y^{2}}{4x}}\\
=\sum_{n=1}^{\infty}a(n)\,e^{-\lambda_{n}x}J_{r-1}\left(\sqrt{\lambda_{n}}\,y\right)-\frac{y^{r-1}\rho}{2^{r-1}x^{r}}\,e^{-\frac{y^{2}}{4x}},
\end{align*}
where in the last step we have used the integral representation (\ref{Mellin representation J e}).
 This proves 
\begin{align}
\frac{y^{r-1}x^{-r/2}}{2^{r}\pi\,\Gamma(r)}\,\intop_{-\infty}^{\infty}\Gamma\left(\frac{r}{2}+it\right)\,\phi\left(\frac{r}{2}+it\right)\,_{1}F_{1}\left(\frac{r}{2}+it;\,r;\,-\frac{y^{2}}{4x}\right)x^{-it}\,dt\nonumber \\
=\,\sum_{n=1}^{\infty}a(n)\,e^{-\lambda_{n}x}J_{r-1}\left(\sqrt{\lambda_{n}}\,y\right)-\frac{y^{r-1}\rho}{2^{r-1}x^{r}}\,e^{-\frac{y^{2}}{4x}}.\label{auxiliary identity almost at the end}
\end{align}

After multiplying both sides of (\ref{auxiliary identity almost at the end})
by the factor $\Gamma(r)\,2^{r-1}y^{1-r}x^{r/2}$ and appealing to
the definition of the generalized theta function (\ref{Definition generalized Theta}),
we derive 
\begin{equation}
x^{r/2}\Phi\left(x,y\right)-\frac{\rho\Gamma(r)}{x^{r/2}}\,e^{-\frac{y^{2}}{4x}}=\frac{1}{2\pi}\,\intop_{-\infty}^{\infty}\Gamma\left(\frac{r}{2}+it\right)\,\phi\left(\frac{r}{2}+it\right)\,_{1}F_{1}\left(\frac{r}{2}+it;\,r;\,-\frac{y^{2}}{4x}\right)x^{-it}\,dt.\label{thing as above with theta}
\end{equation}

Now look at the integral on the right-hand side of (\ref{thing as above with theta}):
by the functional equation for $\phi(s)$ (\ref{Hecke Dirichlet series Functional}) and Kummer's formula (\ref{Kummer confluent transformation}), we can rewrite it as
\begin{align}
\frac{1}{2\pi}\,\intop_{-\infty}^{\infty}\Gamma\left(\frac{r}{2}+it\right)\,\phi\left(\frac{r}{2}+it\right)\,&_{1}F_{1}\left(\frac{r}{2}+it;\,r;\,-\frac{y^{2}}{4x}\right)x^{-it}\,dt=\frac{1}{2\pi}\,\intop_{-\infty}^{\infty}\Gamma\left(\frac{r}{2}+it\right)\,\psi\left(\frac{r}{2}+it\right)\,_{1}F_{1}\left(\frac{r}{2}-it;\,r;\,-\frac{y^{2}}{4x}\right)x^{it}\,dt\nonumber \\
&=\frac{e^{-\frac{y^{2}}{4x}}}{2\pi}\,\intop_{-\infty}^{\infty}\Gamma\left(\frac{r}{2}+it\right)\,\psi\left(\frac{r}{2}+it\right)\,_{1}F_{1}\left(\frac{r}{2}+it;\,r;\,\frac{y^{2}}{4x}\right)x^{it}\,dt.\label{equation final integral}
\end{align}

But the last integral in (\ref{equation final integral}) is actually
the same as the one on the right side of (\ref{thing as above with theta}) if we replace there $x$ by $1/x$, $y$ by $iy/x$
and $\phi(s)$ by $\psi(s)$: hence (\ref{thing as above with theta})
implies
\begin{equation}
x^{-r/2}\Psi\left(\frac{1}{x},\,i\,\frac{y}{x}\right)-\rho^{\star}\Gamma(r)\,x^{r/2}\,e^{\frac{y^{2}}{4x}}=\frac{1}{2\pi}\,\intop_{-\infty}^{\infty}\Gamma\left(\frac{r}{2}+it\right)\,\psi\left(\frac{r}{2}+it\right)\,_{1}F_{1}\left(\frac{r}{2}+it;\,r;\,\frac{y^{2}}{4x}\right)x^{it}\,dt,\label{second identity dual}
\end{equation}
where $\rho^{\star}$ is the residue that $\psi(s)$ possesses at
its simple pole at $s=r$. From the simple properties of the class
$\mathcal{A}$ (see Remark \ref{properties class A} above), we know that $\rho^{\star}\Gamma(r)=-\phi(0)$,
which gives, after using (\ref{second identity dual}), (\ref{equation final integral})
and (\ref{thing as above with theta}),
\begin{align*}
e^{-\frac{y^{2}}{4x}}x^{-r/2}\Psi\left(\frac{1}{x},\,i\,\frac{y}{x}\right)+\phi(0)\,x^{r/2} & =\frac{e^{-\frac{y^{2}}{4x}}}{2\pi}\,\intop_{-\infty}^{\infty}\Gamma\left(\frac{r}{2}+it\right)\,\psi\left(\frac{r}{2}+it\right)\,_{1}F_{1}\left(\frac{r}{2}+it;\,r;\,\frac{y^{2}}{4x}\right)x^{it}\,dt\\
=\frac{1}{2\pi}\,\intop_{-\infty}^{\infty}\Gamma\left(\frac{r}{2}+it\right)\,\phi\left(\frac{r}{2}+it\right) & \,_{1}F_{1}\left(\frac{r}{2}+it;\,r;\,-\frac{y^{2}}{4x}\right)x^{-it}\,dt=x^{r/2}\Phi\left(x,y\right)-\frac{\rho\Gamma(r)}{x^{r/2}}\,e^{-\frac{y^{2}}{4x}},
\end{align*}
completing our proof.    
\end{proof}

\bigskip{}

Since the previous ``modular'' relation (\ref{generalized Theta reflection}) is an analogue of Bochner's
``modular'' relation (\ref{Bochner Modular relation at intro}),
in the next corollary we note that (\ref{Bochner Modular relation at intro})
can indeed be obtained from (\ref{final formula for 1f1 theorem}).

\begin{corollary} \label{Bochner as corollary}
Assume that $\phi(s)\in\mathcal{A}$: then for any $\text{Re}(x)>0$,
Bochner's formula (\ref{Bochner Modular relation at intro})
takes place
\begin{equation}
\sum_{n=1}^{\infty}a(n)\,e^{-\lambda_{n}x}=\phi(0)+\rho\Gamma(r)\,x^{-r}+x^{-r}\,\sum_{n=1}^{\infty}b(n)\,e^{-\mu_{n}/x}.\label{Bochner class A}
\end{equation}   
\end{corollary}

\begin{proof}
The transformation formula (\ref{generalized Theta reflection}) implies,
for $\text{Re}(x)>0$ and $y\in\mathbb{C}$,
\begin{align*}
\Gamma(r)\,2^{r-1}y^{1-r}\,\sum_{n=1}^{\infty}a(n)\,\lambda_{n}^{\frac{1-r}{2}}\,e^{-\lambda_{n}x}\,J_{r-1}(\sqrt{\lambda_{n}}\,y) & =\phi(0)+\frac{\rho}{x^{r}}\Gamma(r)\,e^{-\frac{y^{2}}{4x}}+\\
+\frac{e^{-\frac{y^{2}}{4x}}}{x}\,\Gamma(r)\,2^{r-1}y^{1-r}\,\sum_{n=1}^{\infty}b(n)\,\mu_{n}^{\frac{1-r}{2}} & e^{-\frac{\mu_{n}}{x}}\,I_{r-1}\left(\frac{\sqrt{\mu_{n}}y}{x}\right).
\end{align*}

We have seen above that the series on both sides converge absolutely
and uniformly with respect to $y$ contained in any bounded subset
of $\mathbb{C}$. Therefore, we can take $y\rightarrow0^{+}$ and interchange
the orders of limit and summation on both sides. From the limiting
relations for the Bessel functions [\cite{NIST}, p. 223, eq. (10.7.3)],
\begin{equation}
\lim_{y\rightarrow0}y^{-\nu}J_{\nu}(y)=\frac{2^{-\nu}}{\Gamma(\nu+1)},\,\,\,\,\lim_{y\rightarrow0}y^{-\nu}I_{\nu}(y)=\frac{2^{-\nu}}{\Gamma(\nu+1)},\label{limiting relation Bessel}
\end{equation}
(\ref{Bochner class A}) is immediately obtained.      
\end{proof}

\bigskip{}

By the equivalence between the identities (\ref{Bochner Modular relation at intro}),
(\ref{Berndt Bessel Expansion}) and Hecke's functional equation,
the fact that (\ref{generalized Theta reflection}) generalizes (\ref{Bochner Modular relation at intro})
puts now an interesting question: is it possible to find a generalization 
of (\ref{Berndt Bessel Expansion}) in this context? The following corollary furnishes a positive answer to this question. 

\begin{corollary} \label{generalized bessel expansion corollary 2.4}
Let $\phi(s)$ and $\psi(s)$ be two Dirichlet series satisfying Definition
\ref{class A} and let $\sigma_{a}$ be the abscissa of absolute convergence
of $\phi(s)$. Furthermore, let $x,\,y$ be two positive real numbers
and $s$ be any complex number such that $\text{Re}(s)>\sigma_{a}$.

\bigskip{}

Then the following identity takes place  
\begin{align}
\frac{x^{s-r}}{2}\,\frac{\Gamma(s)\,y^{r-1}}{\Gamma(r)}\,\sum_{n=1}^{\infty}\frac{a(n)}{\left(\lambda_{n}+x^{2}+y^{2}\right)^{s}}\,_{2}F_{1}\left(\frac{s}{2},\,\frac{s+1}{2};\,r;\,\frac{4y^{2}\lambda_{n}}{\left(\lambda_{n}+x^{2}+y^{2}\right)^{2}}\right)\nonumber \\
=\frac{\rho\,y^{r-1}\,x^{r-s}\,\Gamma(s-r)}{2}+\frac{x^{s-r}y^{r-1}\phi(0)\,\Gamma(s)}{2\Gamma(r)\,\left(x^{2}+y^{2}\right)^{s}}+\sum_{n=1}^{\infty}\,b(n)\,\mu_{n}^{\frac{s+1}{2}-r}\,J_{r-1}(2\sqrt{\mu_{n}}\,y)\, & K_{s-r}\left(2\sqrt{\mu_{n}}\,x\right).\label{Formula generalized Bessel as Corollary}
\end{align}

\bigskip{}

Moreover, if $x$ and $y$ are two positive numbers such that $x>y$
and $\text{Re}(s)>\sigma_{a}$, then the analogous identity holds: 
\begin{align}
\frac{x^{s-r}}{2}\,\frac{\Gamma(s)\,y^{r-1}}{\Gamma(r)}\,\sum_{n=1}^{\infty}\frac{a(n)}{\left(\lambda_{n}+x^{2}-y^{2}\right)^{s}}\,_{2}F_{1}\left(\frac{s}{2},\,\frac{s+1}{2};\,r;\,-\frac{4y^{2}\lambda_{n}}{(\lambda_{n}+x^{2}-y^{2})^{2}}\right)\nonumber \\
=\frac{\rho\,y^{r-1}x^{r-s}\,\Gamma(s-r)}{2}+\frac{x^{s-r}y^{r-1}\phi(0)\,\Gamma(s)}{2\Gamma(r)\,(x^{2}-y^{2})^{s}}+\sum_{n=1}^{\infty}\,b(n)\,\mu_{n}^{\frac{s+1}{2}-r}\,I_{r-1}(2\sqrt{\mu_{n}}\,y)\, & K_{s-r}\left(2\sqrt{\mu_{n}}\,x\right).\label{Formula with Modified Bessel for generalized Bessel}
\end{align}

\end{corollary}

\begin{proof}
Since $I_{\nu}(x)=i^{-\nu}J_{\nu}(ix)$ for $x>0$, it is simple to
see that (\ref{Formula with Modified Bessel for generalized Bessel})
can be derived from (\ref{Formula generalized Bessel as Corollary})
once we assure the convergence of both sides of (\ref{Formula with Modified Bessel for generalized Bessel}).
Indeed, the series on the right-hand side of (\ref{Formula with Modified Bessel for generalized Bessel})
converges absolutely for every $s\in\mathbb{C}$ because of the fact
that $x>y$ and the asymptotic formulas (\ref{asymptotic I}) and
(\ref{asymptotic K}) {[}\cite{watson_bessel}, p. 199-203{]},
\begin{equation}
K_{\nu}(x)=\sqrt{\frac{\pi}{2x}}\,e^{-x}\left(1+O\left(\frac{1}{x}\right)\right),\,\,\,\,I_{\nu}(x)=\frac{e^{x}}{\sqrt{2\pi x}}\left(1+O\left(\frac{1}{x}\right)\right),\,\,\,\,x\rightarrow\infty.\label{asymptotic two modified Compared}
\end{equation}

To justify the convergence of the series on the left-hand side of
(\ref{Formula with Modified Bessel for generalized Bessel}), note
that, as $n\rightarrow\infty$,
\[
\frac{1}{\left(\lambda_{n}+x^{2}-y^{2}\right)^{s}}\,_{2}F_{1}\left(\frac{s}{2},\,\frac{s+1}{2};\,r;\,-\frac{4y^{2}\lambda_{n}}{\left(\lambda_{n}+x^{2}-y^{2}\right)^{2}}\right)\sim\frac{1}{\left(\lambda_{n}+x^{2}-y^{2}\right)^{s}}
\]
since the hypergeometric function factor tends to $1$ as $\lambda_{n}\rightarrow\infty$.
Thus, if $\text{Re}(s)>\sigma_{a}$, the series on the left of (\ref{Formula with Modified Bessel for generalized Bessel})
converges absolutely. Similar comments can be made about the series on both sides of (\ref{Formula generalized Bessel as Corollary})
without the additional hypothesis that $x>y$. 

\bigskip{}

From the previous comments, we just need to prove the identity (\ref{Formula generalized Bessel as Corollary}).
To that end, we look first at the infinite series on its right-hand
side and invoke the well-known representation for the Modified Bessel
function of the second kind {[}\cite{ryzhik}, eq. 3.471.9, p. 368{]},
\begin{equation}
\intop_{0}^{\infty}x^{s-1}e^{-\beta x}\,e^{-\frac{\gamma}{x}}\,dx=2\left(\frac{\gamma}{\beta}\right)^{s/2}\,K_{s}\left(2\sqrt{\beta\,\gamma}\right),\,\,\,\,\text{Re}(\beta),\,\text{Re}(\gamma)>0.\label{representation Bessel convolution}
\end{equation}

Using (\ref{representation Bessel convolution}) on the infinite series appearing on the right side
of (\ref{Formula generalized Bessel as Corollary}), we obtain 
\begin{equation*}
\sum_{n=1}^{\infty}\,b(n)\,\mu_{n}^{\frac{s+1-2r}{2}}\,J_{r-1}(2\sqrt{\mu_{n}}\,y)\,K_{s-r}\left(2\sqrt{\mu_{n}}\,x\right) =\frac{x^{s-r}}{2}\,\intop_{0}^{\infty}t^{s-r-1}e^{-x^{2}t}\,\sum_{n=1}^{\infty}\,b(n)\,\mu_{n}^{\frac{1-r}{2}}\,e^{-\frac{\mu_{n}}{t}}J_{r-1}\left(2\sqrt{\mu_{n}}\,y\right)\,dt,
\end{equation*}
where interchanging the orders of summation and integration is possible
due to absolute convergence. Now we invoke the transformation formula (\ref{final formula for 1f1 theorem})
with the roles of $\phi(s)$ and $\psi(s)$ being reversed and we
get 
\begin{align}
\sum_{n=1}^{\infty}\,b(n)\,\mu_{n}^{\frac{s+1-2r}{2}}\,J_{r-1}(2\sqrt{\mu_{n}}\,y)\,K_{s-r}\left(2\sqrt{\mu_{n}}\,x\right)
=\frac{x^{s-r}}{2}\,\intop_{0}^{\infty}t^{s-r-1}e^{-x^{2}t}\,\cdot\,\left\{ \frac{\psi(0)y^{r-1}}{\Gamma(r)}+y^{r-1}\rho^{\star}\,t^{r}\,e^{-y^{2}t}\right\} \,dt+\nonumber \\
+\frac{x^{s-r}}{2}\,\intop_{0}^{\infty}t^{s-r-1}e^{-x^{2}t}\,\cdot\,\left\{ t\,e^{-y^{2}t}\,\sum_{n=1}^{\infty}a(n)\,\lambda_{n}^{\frac{1-r}{2}}\,e^{-\lambda_{n}t}\,I_{r-1}\left(2y\,t\,\sqrt{\lambda_{n}}\right)\right\} \,dt.\label{formula after summation}
\end{align}

The computation of the first integral is simple and it is equal to: 
\begin{align}
\frac{x^{s-r}}{2}\,\intop_{0}^{\infty}t^{s-r-1}e^{-x^{2}t}\left\{ \frac{\psi(0)y^{r-1}}{\Gamma(r)}+y^{r-1}\rho^{\star}\,t^{r}\,e^{-y^{2}t}\right\} \,dt & =\frac{\psi(0)\,y^{r-1}\,x^{r-s}\,\Gamma(s-r)}{2\Gamma(r)}+\frac{x^{s-r}y^{r-1}\rho^{\star}\,\Gamma(s)}{2\left(x^{2}+y^{2}\right)^{s}}\nonumber \\
=-\frac{\rho\,y^{r-1}\,x^{r-s}\,\Gamma(s-r)}{2} & -\frac{x^{s-r}y^{r-1}\phi(0)\,\Gamma(s)}{2\Gamma(r)\,\left(x^{2}+y^{2}\right)^{s}},\label{computation first integrals for residual}
\end{align}
where we have used the simple properties of the class $\mathcal{A}$,
namely that $\psi(0)=-\rho\,\Gamma(r)$ and $\rho^{\star}=-\phi(0)/\Gamma(r)$ (see Remark \ref{properties class A}). 
To evaluate the second integral, we use absolute convergence (assured
by item 1. of Corollary \ref{integral representation theta} above) and the integral given in {[}\cite{handbook_marichev},
p. 191, eq. 3.13.2.1{]},
\begin{equation}
\intop_{0}^{\infty}x^{z-1}\,e^{-Ax}I_{\nu}(Bx)\,dx=A^{-z-\nu}\,\left(\frac{B}{2}\right)^{\nu}\,\frac{\Gamma(z+\nu)}{\Gamma(\nu+1)}\,_{2}F_{1}\left(\frac{z+\nu}{2},\,\frac{z+\nu+1}{2};\,\nu+1;\,\frac{B^{2}}{A^{2}}\right),\label{marichev I integral}
\end{equation}
valid for $\text{Re}(A)>|\text{Re}(B)|$ and $\text{Re}(z)>-\text{Re}(\nu)$.

\bigskip{}

Invoking (\ref{marichev I integral}) on the last integral in (\ref{formula after summation}),
we obtain
\begin{align}
\intop_{0}^{\infty}t^{s-r}e^{-x^{2}t}\,e^{-y^{2}t}\,\sum_{n=1}^{\infty}a(n)\,\lambda_{n}^{\frac{1-r}{2}}\,e^{-\lambda_{n}t}\,I_{r-1}\left(2y\,t\,\sqrt{\lambda_{n}}\right)\,dt=\sum_{n=1}^{\infty}a(n)\,\lambda_{n}^{\frac{1-r}{2}}\,\intop_{0}^{\infty}t^{s-r}\,e^{-\left(x^{2}+y^{2}+\lambda_{n}\right)t}\,I_{r-1}\left(2yt\,\sqrt{\lambda_{n}}\right)\,dt\nonumber \\
=\frac{\Gamma(s)\,y^{r-1}}{\Gamma(r)}\,\sum_{n=1}^{\infty}\frac{a(n)}{\left(\lambda_{n}+x^{2}+y^{2}\right)^{s}}\,_{2}F_{1}\left(\frac{s}{2},\,\frac{s+1}{2};\,r;\,\frac{4y^{2}\lambda_{n}}{\left(\lambda_{n}+x^{2}+y^{2}\right)^{2}}\right),\label{simplifying integral}
\end{align}
which is legitimate through (\ref{marichev I integral}) because $\text{Re}(s)>\sigma_{a}\geq\frac{r}{2}>0$ by hypothesis
and $x^{2}+y^{2}+\lambda_{n}>y^{2}+\lambda_{n}>2y\sqrt{\lambda_{n}}$.
Combining (\ref{simplifying integral}) with (\ref{formula after summation})
and (\ref{computation first integrals for residual}) proves our formula
(\ref{Formula generalized Bessel as Corollary}).
    
\end{proof}

\begin{remark} \label{sum of squares product}
A formula similar to (\ref{Formula with Modified Bessel for generalized Bessel})
involving the product $I_{\nu}(y)\,K_{\nu}(z)$ was found by Berndt,
Dixit, Kim and Zaharescu in \cite{sum_of_squares_berndt} and later generalized
to a class of Dirichlet series by three of the authors and Gupta in
{[}\cite{berndt_general_bessel}, Theorem 11.1{]}.

Restricted to the case where the Dirichlet series is attached to the
sum of $k-$squares, $\zeta_{k}(s)$, their identity reads {[}\cite{sum_of_squares_berndt},
p. 315, Theorem 1.6.{]}
\begin{align}
\sum_{n=1}^{\infty}r_{k}(n)\,I_{\nu}\left(\pi\sqrt{n}\left(\sqrt{\alpha}-\sqrt{\beta}\right)\right)\,K_{\nu}\left(\pi\sqrt{n}\left(\sqrt{\alpha}+\sqrt{\beta}\right)\right)\nonumber \\
=-\frac{1}{2\nu}\left(\frac{\sqrt{\alpha}-\sqrt{\beta}}{\sqrt{\alpha}+\sqrt{\beta}}\right)^{\nu}+\frac{\Gamma\left(\frac{k}{2}+\nu\right)}{\pi^{k/2}2^{k-1}\Gamma(\nu+1)}\,\sum_{n=0}^{\infty}\frac{r_{k}(n)}{\sqrt{n+\alpha}\,\sqrt{n+\beta}}\,\left(\frac{\sqrt{n+\alpha}-\sqrt{n+\beta}}{\sqrt{n+\alpha}+\sqrt{n+\beta}}\right)^{\nu}\times\nonumber \\
\times\left(\frac{1}{\sqrt{n+\alpha}}+\frac{1}{\sqrt{n+\beta}}\right)^{k-2}\,_{2}F_{1}\left(1-\frac{k}{2}+\nu,\,1-\frac{k}{2};\,\nu+1;\,\left(\frac{\sqrt{n+\alpha}-\sqrt{n+\beta}}{\sqrt{n+\alpha}+\sqrt{n+\beta}}\right)^{2}\right),\label{Berndt Kim Dixit Zaharescu Identity}
\end{align}
where $\text{Re}(\sqrt{\alpha})\geq\text{Re}(\sqrt{\beta})>0$ and
$\text{Re}(\nu)>0$. Compare formula (\ref{Berndt Kim Dixit Zaharescu Identity})
with (\ref{analogue watson for r alpha zeta}) below for $\alpha=k$.

Our identity (\ref{Formula with Modified Bessel for generalized Bessel})
and its consequences (c.f. (\ref{analogue watson for r alpha zeta})
below) are somewhat akin to (\ref{Berndt Kim Dixit Zaharescu Identity})
but they seem to be independent because in (\ref{Berndt Kim Dixit Zaharescu Identity})
it is required that the indices of the Bessel functions are the same.
In our formula (\ref{Formula with Modified Bessel for generalized Bessel})
we have a free complex parameter $s$, although the index of the modified
Bessel function $I_{\nu}(x)$ is fixed by the functional equation,
being equal to $r-1$. In the case of (\ref{Berndt Kim Dixit Zaharescu Identity})
the indices of the Bessel functions are kept equal but they are independent
of the functional equation for $\zeta_{k}(s)$, this is, they do not
depend on $k$.    
\end{remark}

\bigskip{}

The formula obtained by Berndt, Dixit, Kim and Zaharescu (\ref{Berndt Kim Dixit Zaharescu Identity})
can be used to extend known formulas due to Dixon and Ferrar \cite{dixon_ferrar_lattice, dixon_ferrar_circle(i), dixon_ferrar_circle(ii)}. For example,
by using (\ref{Berndt Kim Dixit Zaharescu Identity}), they were able
to establish the beautiful identity (see {[}\cite{sum_of_squares_berndt}, p. 329, Corollary 4.6.{]})
\begin{equation}
\frac{\beta^{\nu/2}\Gamma\left(\nu+\frac{k}{2}\right)}{2\pi^{\nu+\frac{k}{2}}}\,\sum_{n=0}^{\infty}\frac{r_{k}(n)}{(n+\beta)^{\nu+\frac{k}{2}}}=\sum_{n=0}^{\infty}r_{k}(n)\,n^{\nu/2}K_{\nu}\left(2\pi\sqrt{n\beta}\right),\label{ferrar berndt sum of squares}
\end{equation}
which is valid for any positive integer $k>1$ and $\text{Re}(\sqrt{\beta}),\,\text{Re}(\nu)>0$.
On the other hand, (\ref{ferrar berndt sum of squares}) can be established
immediately from Berndt's general Bessel expansion (\ref{Berndt Bessel Expansion})
because $\zeta_{k}(s)$ satisfies Hecke's functional equation.

\bigskip{}

In a general context, we now prove that (\ref{Formula generalized Bessel as Corollary})
implies (\ref{Berndt Bessel Expansion}) in the same way that the
similar formula (\ref{Berndt Kim Dixit Zaharescu Identity}) gives
(\ref{ferrar berndt sum of squares}).

\begin{corollary} \label{berndt as particular case}
Let $\text{Re}(s)>\sigma_{a}$ and $\phi(s)$ be a Dirichlet series
belonging to the class $\mathcal{A}$. Then, for every $x>0$, the
following Bessel expansion takes place
\begin{equation}
\sum_{n=1}^{\infty}\frac{a(n)}{\left(\lambda_{n}+x^{2}\right)^{s}}=\frac{\rho\,\Gamma(r)\,x^{2r-2s}\,\Gamma(s-r)}{\Gamma(s)}+x^{-2s}\phi(0)+\frac{2x^{r-s}}{\Gamma(s)}\,\sum_{m=1}^{\infty}\,b(m)\,\mu_{m}^{\frac{s-r}{2}}\,K_{s-r}\left(2\sqrt{\mu_{m}}\,x\right).\label{Berndt formula Bessel expansion 1967}
\end{equation}
   
\end{corollary}

\begin{proof}
 Multiply both sides of (\ref{Formula generalized Bessel as Corollary})
by $\Gamma(r)\,y^{1-r}$: we obtain
\begin{align*}
\Gamma(r)\,y^{1-r}\,\sum_{m=1}^{\infty}\,b(m)\,\mu_{m}^{\frac{s+1-2r}{2}}\,J_{r-1}(2\sqrt{\mu_{m}}\,y)\,K_{s-r}\left(2\sqrt{\mu_{m}}\,x\right)+\frac{\rho\,\Gamma(r)\,x^{r-s}\,\Gamma(s-r)}{2}+\,\frac{x^{s-r}\phi(0)\,\Gamma(s)}{2\,\left(x^{2}+y^{2}\right)^{s}}\\
=\frac{x^{s-r}}{2}\,\Gamma(s)\,\sum_{n=1}^{\infty}\frac{a(n)}{\left(\lambda_{n}+x^{2}+y^{2}\right)^{s}}\,_{2}F_{1}\left(\frac{s}{2},\,\frac{s+1}{2};\,r;\,\frac{4y^{2}\lambda_{n}}{\left(\lambda_{n}+x^{2}+y^{2}\right)^{2}}\right).
\end{align*}

We have argued at the beginning of the proof of corollary \ref{generalized bessel expansion corollary 2.4}
that both series above converge absolutely and uniformly with respect to
$y$ for $0\leq y<M$. Thus, letting $y\rightarrow0^{+}$, using the
first limiting relation in (\ref{limiting relation Bessel}) and recalling
that $_{2}F_{1}(a,b;\,c;\,0)=1$, we obtain the identity
\[
\sum_{n=1}^{\infty}\frac{a(n)}{\left(\lambda_{n}+x^{2}\right)^{s}}=\frac{\rho\,\Gamma(r)\,x^{2r-2s}\,\Gamma(s-r)}{\Gamma(s)}+x^{-2s}\phi(0)+\frac{2x^{r-s}}{\Gamma(s)}\,\sum_{m=1}^{\infty}\,b(m)\,\mu_{m}^{\frac{s-r}{2}}\,K_{s-r}\left(2\sqrt{\mu_{m}}\,x\right),
\]
which is Berndt's formula (\ref{Berndt Bessel Expansion}) restricted
to the class $\mathcal{A}$.   
\end{proof}

\bigskip{}

Since (\ref{Berndt formula Bessel expansion 1967})
can be thought as a generalization of Watson's formula [\cite{watson_reciprocal}, p. 299, eq. (4)], we now see that Corollary \ref{generalized bessel expansion corollary 2.4} also gives another analogue of Watson's
formula for the class $\mathcal{B}$. \footnote{Corollary \ref{Watson with cosine twist} is essentially given in {[}\cite{rysc_I}, Example 5.1.,
eq. (5.2), (5.3){]}, although in this reference there are some additional assumptions on the arithmetical function $a(n)$.}

\begin{corollary} \label{Watson with cosine twist}
Let $Q(x,y)=Ax^{2}+Bxy+Cy^{2}$ be a positive definite and real quadratic
form, i.e., $\Delta:=4AC-B^{2}>0$ and $A>0$. Also, define
$k:=\sqrt{\Delta}/2A$.

\bigskip{}

If $\phi(s)$ is a Dirichlet series belonging to the class $\mathcal{B}$
and $s$ is a complex number such that $\text{Re}(s)>\frac{\sigma_{a}}{2}$,
then the following formulas take place:
\begin{align}
-2\phi(0)\,C^{-s}+\sum_{n=1}^{\infty}\frac{a(n)}{\left(A\lambda_{n}^{2}+B\,\lambda_{n}+C\right)^{s}}+\sum_{n=1}^{\infty}\frac{a(n)}{\left(A\lambda_{n}^{2}-B\,\lambda_{n}+C\right)^{s}} & =\rho\sqrt{\pi}\,\frac{\Gamma(s-\frac{1}{2})}{\Gamma(s)}\,A^{-s}k^{1-2s}+\nonumber \\
+\frac{4k^{\frac{1}{2}-s}A^{-s}}{\Gamma(s)}\,\sum_{n=1}^{\infty}b(n)\mu_{n}^{s-\frac{1}{2}}\,\cos\left(\frac{B\mu_{n}}{A}\right) & K_{s-\frac{1}{2}}\left(2k\,\mu_{n}\right),\,\,\,\,\,\delta=0,\label{generalized Watson Kober}
\end{align}
\begin{equation}
\sum_{n=1}^{\infty}\frac{a(n)}{\left(A\lambda_{n}^{2}+B\,\lambda_{n}+C\right)^{s}}-\sum_{n=1}^{\infty}\frac{a(n)}{\left(A\lambda_{n}^{2}+B\,\lambda_{n}+C\right)^{s}}=-\frac{4k^{\frac{1}{2}-s}A^{-s}}{\Gamma(s)}\,\sum_{n=1}^{\infty}b(n)\,\mu_{n}^{s-\frac{1}{2}}\sin\left(\frac{B\mu_{n}}{A}\right)\,K_{s-\frac{1}{2}}\left(2k\,\mu_{n}\right),\,\,\,\delta=1.\label{generalized Watson kober odd}
\end{equation}

\end{corollary}

\begin{proof}
Starting with (\ref{generalized Watson Kober}), we use the fact that
if $\phi(s)\in\mathcal{B}$ with $\delta=0$, then $\phi(2s)\in\mathcal{A}$ with $r=\frac{1}{2}$.
Thus, to obtain (\ref{generalized Watson Kober}), we start by using
(\ref{Formula generalized Bessel as Corollary}) with $\phi(s)$ being replaced by $\phi(2s)$ and $r=\frac{1}{2}$. 

Appealing to the well-known case for Gauss Hypergeometric function
{[}\cite{MARICHEV}, vol. 3, p. 461, eq. (7.3.1.106){]},
\begin{equation}
_{2}F_{1}\left(a,a+\frac{1}{2};\,\frac{1}{2};\,z\right)=\frac{1}{2}\left[\left(1+\sqrt{z}\right)^{-2a}+\left(1-\sqrt{z}\right)^{-2a}\right],\label{particular case hypergeometric gauss}
\end{equation}
the left-hand side of (\ref{Formula generalized Bessel as Corollary})
can be simplified to 
\begin{align}
\frac{x^{s-\frac{1}{2}}}{2}\,\Gamma(s)\,\sum_{n=1}^{\infty}\frac{a(n)}{\left(\lambda_{n}^{2}+x^{2}+y^{2}\right)^{s}}\,_{2}F_{1}\left(\frac{s}{2},\,\frac{s+1}{2};\,\frac{1}{2};\,\frac{4y^{2}\lambda_{n}^{2}}{\left(\lambda_{n}^{2}+x^{2}+y^{2}\right)^{2}}\right)\nonumber \\
\frac{x^{s-\frac{1}{2}}}{4}\,\Gamma(s)\,\sum_{n=1}^{\infty}a(n)\,\left\{ \left(\lambda_{n}^{2}+2y\,\lambda_{n}+x^{2}+y^{2}\right)^{-s}+\left(\lambda_{n}^{2}-2y\,\lambda_{n}+x^{2}+y^{2}\right)^{-s}\right\}  & ,\,\,\,\text{Re}(s)>\frac{\sigma_{a}}{2}.\label{1st simplification}
\end{align}

On the other hand, for every $s\in\mathbb{C}$, we can write the
right-hand side of (\ref{Formula generalized Bessel as Corollary})
as follows:
\begin{equation}
\,\sum_{m=1}^{\infty}\,b(m)\,\mu_{m}^{s-\frac{1}{2}}\,\cos\left(2\mu_{m}y\right)\,K_{s-\frac{1}{2}}\left(2\,\mu_{m}\,x\right)+\frac{\sqrt{\pi}\,\rho\,x^{\frac{1}{2}-s}\,\Gamma(s-\frac{1}{2})}{4}+\,\frac{x^{s-\frac{1}{2}}\,\phi(0)\,\Gamma(s)}{2\,\left(x^{2}+y^{2}\right)^{s}},\label{second simplification}
\end{equation}
where we have used the particular case $J_{-1/2}(x)=\sqrt{\frac{2}{\pi x}}\,\cos(x)$
(\ref{particular cases formula Cosine}) and the fact that $\phi(2s)$
has residue $\rho/2$ at $s=1/2$ (see Remark \ref{Bochner Hecke Remark}).

\bigskip{}

Comparing (\ref{1st simplification}) with (\ref{second simplification})
gives the formula:
\begin{align*}
-\frac{2\phi(0)}{\left(x^{2}+y^{2}\right)^{s}}+\sum_{n=1}a(n)\,\left\{ \left(\lambda_{n}^{2}+2y\,\lambda_{n}+x^{2}+y^{2}\right)^{-s}+\left(\lambda_{n}^{2}-2y\,\lambda_{n}+x^{2}+y^{2}\right)^{-s}\right\}  & =\\
=\frac{\sqrt{\pi}\,\rho\,x^{1-2s}\,\Gamma(s-\frac{1}{2})}{\Gamma(s)}+\frac{4x^{\frac{1}{2}-s}}{\Gamma(s)}\sum_{m=1}^{\infty}\,b(m)\,\mu_{m}^{s-\frac{1}{2}}\cos\left(2\mu_{m}y\right)\,K_{s-\frac{1}{2}}\left(2\,\mu_{m}\,x\right).
\end{align*}
Multiplying both sides of the previous equality by $A^{-s}$ and defining
$y=\frac{B}{2A}$ and $x=\frac{\sqrt{\Delta}}{2A}:=k$, we get immediately
(\ref{generalized Watson Kober}). 

\bigskip{}

The proof of (\ref{generalized Watson kober odd})
is analogous: in this case, note that $\phi(2s-1)\in\mathcal{A}$
with $r=\frac{3}{2}$ (see Remark \ref{Bochner Hecke Remark}), so we just need to use (\ref{Formula generalized Bessel as Corollary})
with $r=\frac{3}{2}$. Appealing to the formula {[}\cite{MARICHEV},
vol. 3, p. 461, eq. (7.3.1.107){]},
\[
_{2}F_{1}\left(a,a+\frac{1}{2};\,\frac{3}{2};\,z\right)=\frac{1}{2(2a-1)\sqrt{z}}\left\{ \left(1-\sqrt{z}\right)^{1-2a}-\left(1+\sqrt{z}\right)^{1-2a}\right\} ,
\]
we see after some straightforward simplifications that the left-hand
side of (\ref{Formula generalized Bessel as Corollary}) is reduced
to

\begin{equation}
\frac{x^{s-\frac{3}{2}}}{4\,\sqrt{\pi y}}\,\Gamma(s-1)\,\sum_{n=1}^{\infty}a(n)\,\left\{ \frac{1}{\left(\lambda_{n}^{2}-2y\lambda_{n}+x^{2}+y^{2}\right)^{s-1}}-\frac{1}{\left(\lambda_{n}^{2}+2y\lambda_{n}+x^{2}+y^{2}\right)^{s-1}}\right\} ,\label{left hand side odd}
\end{equation}
for $\text{Re}(s)>\frac{\sigma_{a}}{2}$. 

In what concerns the right-hand side of (\ref{Formula generalized Bessel as Corollary}),
note that when $\delta=1$ in the class $\mathcal{B}$, we have $\rho=\phi(0)=0$,
so we just need to simplify the series involving the Bessel functions.

\bigskip{}

For every $s\in\mathbb{C}$, this series can be written as

\begin{equation}
\sum_{n=1}^{\infty}\,b(n)\,\mu_{n}^{s-1}\,J_{\frac{1}{2}}(2\mu_{n}\,y)\,K_{s-\frac{3}{2}}(2\mu_{n}\,x)=\frac{1}{\sqrt{\pi y}}\,\sum_{n=1}^{\infty}\,b(n)\,\mu_{n}^{s-\frac{3}{2}}\,\sin(2\mu_{n}\,y)\,K_{s-\frac{3}{2}}(2\mu_{n}\,x),\label{right side odd case}
\end{equation}
where we have used $J_{1/2}(x)=\sqrt{\frac{2}{\pi x}}\,\sin(x)$ (\ref{particular cases formula Cosine}).
Comparing both sides (\ref{left hand side odd}) and (\ref{right side odd case})
and replacing $s$ by $s+1$ and defining $y=\frac{B}{2A}$ and $x=\frac{\sqrt{\Delta}}{2A}:=k$,
we obtain (\ref{generalized Watson kober odd}).
\end{proof}

\bigskip{}

\begin{remark}\label{New Epstein direction}
The previous corollary is somewhat remindful of the classical Selberg-Chowla
formula for the Epstein zeta function \cite{Bateman_Epstein, selberg_chowla}. Being
a generalization of Watson's formula, it is expected that (\ref{Formula generalized Bessel as Corollary})
can be used to derive new analogues of Guinand and Koshliakov's formulas \cite{koshliakov_ramanujan_character, koshliakov_berndt, Guinand_Ramanujan, guinand_rapidly_convergent, koshliakov_Voronoi, relations_equivalent, Koshliakov_Soni},
as well as new extensions of the classical Epstein zeta function.
This study will be given elsewhere and for the purposes of the present
paper it is enough to work with the formula (\ref{Integral representation Dixit Style}).

\end{remark}

\begin{center}\subsection{Some Useful Examples} \label{examples} \end{center}
In their paper [\cite{DKMZ}, p. 312] Dixit, Kumar, Maji and Zaharescu
use a variant of Jacobi's $\psi-$function of the form
\begin{equation}
\psi(x,z):=\sum_{n=1}^{\infty}e^{-\pi n^{2}x}\cos(\sqrt{\pi\,x}\,nz),\,\,\,\,\,\text{Re}(x)>0,\,\,\,\,z\in\mathbb{C}.\label{jacobi theta 1 parameter}
\end{equation}
The parameter $z$ in this consideration is exactly the same parameter
that appears in the statement of their Theorem (see eq. (\ref{def function theroem B intro})
above).
So, in order to extend their result, we need to convert the generalized
theta function (\ref{Definition generalized Theta}) into the more
symmetric version (\ref{jacobi theta 1 parameter}) depending on a
new parameter $z$.
\\

This is done in the following definition:
\\

\begin{definition} \label{generalized Jacobi psi def}
Let $\phi(s)$ be any Dirichlet series belonging to the class $\mathcal{A}$
in the sense of Definition \ref{class A}. Then, for $\text{Re}(x)>0$ and $z\in\mathbb{C}$,
we define the generalized Jacobi's $\psi-$function attached to $\phi(s)$
as follows:
\begin{equation}
\psi_{\phi}(x,z):=2^{r-1}\,\Gamma(r)\left(\sqrt{x}\,z\right)^{1-r}\,\sum_{n=1}^{\infty}a(n)\,\lambda_{n}^{\frac{1-r}{2}}\,e^{-\lambda_{n}x}\,J_{r-1}\left(\sqrt{\lambda_{n}x}\,z\right)=\Phi\left(x,\sqrt{x}\,z\right),\label{generalized psi for Hecke}
\end{equation}
where $\Phi(x,y)$ is given in (\ref{Definition generalized Theta}).
Analogously, one may also define:
\begin{equation}
\tilde{\psi}_{\phi}(x,z):=2^{r-1}\Gamma(r)\,\left(\sqrt{x}\,z\right)^{1-r}\,\sum_{n=1}^{\infty}b(n)\,\mu_{n}^{\frac{1-r}{2}}\,e^{-\mu_{n}x}\,J_{r-1}\left(\sqrt{\mu_{n}x}\,z\right)=\Psi\left(x,\sqrt{x}\,z\right).\label{generalized psi for psi dual}
\end{equation}
\end{definition}

\bigskip{}

With this new parameter $z\in\mathbb{C}$, we note that the transformation
formula (\ref{generalized Theta reflection}) can be rewritten as
\begin{align}
\psi_{\phi}(x,z) & :=2^{r-1}\,\Gamma(r)\,x^{\frac{1-r}{2}}z^{1-r}\,\sum_{n=1}^{\infty}a(n)\,\lambda_{n}^{\frac{1-r}{2}}\,e^{-\lambda_{n}x}\,J_{r-1}\left(\sqrt{\lambda_{n}x}\,z\right)\nonumber \\
=\phi(0)+\frac{\rho}{x^{r}}\Gamma(r)\,e^{-\frac{z^{2}}{4}}+2^{r-1}\Gamma(r) & \,z^{1-r}\,x^{-\frac{r+1}{2}}\,e^{-z^{2}/4}\,\sum_{n=1}^{\infty}b(n)\,\mu_{n}^{\frac{1-r}{2}}\,e^{-\frac{\mu_{n}}{x}}\,I_{r-1}\left(\sqrt{\frac{\mu_{n}}{x}}\,z\right):=\label{transformation formula psi function general}\\
 & :=\phi(0)+\frac{\rho\Gamma(r)}{x^{r}}\,e^{-\frac{z^{2}}{4}}+\frac{e^{-\frac{z^{2}}{4}}}{x^{r}}\,\tilde{\psi}_{\phi}\left(\frac{1}{x},iz\right).\nonumber 
\end{align}

Also, identity (\ref{Integral representation Dixit Style}) takes
the new form: 
\begin{align}
x^{r/2}\psi_{\phi}(x,z)-\frac{\rho\Gamma(r)}{x^{r/2}}\,e^{-\frac{z^{2}}{4}} & =e^{-\frac{z^{2}}{4}}\,x^{-r/2}\,\tilde{\psi}_{\phi}\left(\frac{1}{x},iz\right)+\phi(0)\,x^{r/2}=\nonumber \\
=\frac{1}{2\pi}\,\intop_{-\infty}^{\infty}\Gamma\left(\frac{r}{2}+it\right)\,\phi\left(\frac{r}{2}+it\right) & \,_{1}F_{1}\left(\frac{r}{2}+it;\,r;\,-\frac{z^{2}}{4}\right)x^{-it}\,dt.\label{triple identity integral with psi}
\end{align}
which is valid for every $z\in\mathbb{C}$ and $\text{Re}(x)>0$.

\bigskip{}

\begin{example} \label{example 2.1}
Let $\alpha>0$ and $r_{\alpha}(n)$ be defined as the coefficients
of the expansion (\ref{second definition varthet coefficients}).
Recall the Dirichlet series attached to it: 
\begin{equation}
\zeta_{\alpha}(s)=\sum_{n=1}^{\infty}\frac{r_{\alpha}(n)}{n^{s}},\,\,\,\,\,\text{Re}(s)>\sigma_{\alpha}:=\begin{cases}
\alpha/2 & \text{if }\alpha\geq4\\
1+\frac{\alpha}{4} & \text{if }0<\alpha<4 \label{sigma alpha}
\end{cases}.
\end{equation}

We know by Lemma \ref{lemma 1.1} that $\phi(s)=\pi^{-s}\zeta_{\alpha}(s)$ satisfies
Hecke's functional equation (\ref{Hecke Dirichlet series Functional})
with parameter $r=\frac{\alpha}{2}$ and belongs to the class $\mathcal{A}$. 

For this particular example, the analogue of Jacobi's $\psi-$function (\ref{generalized psi for Hecke}) is
\begin{equation}
\psi_{\alpha}(x,z):=2^{\frac{\alpha}{2}-1}\,\Gamma\left(\frac{\alpha}{2}\right)\,\left(\sqrt{\pi x}\,z\right)^{1-\frac{\alpha}{2}}\,\sum_{n=1}^{\infty}r_{\alpha}(n)\,n^{\frac{1}{2}-\frac{\alpha}{4}}\,e^{-\pi n\,x}\,J_{\frac{\alpha}{2}-1}(\sqrt{\pi\,n\,x}\,z),\,\,\text{Re}(x)>0,\,\,\,\,z\in\mathbb{C}\label{Definition Jacobi final version} 
\end{equation}

\bigskip{}

Note that we can recover the classical Jacobi's function (\ref{jacobi theta 1 parameter})
from (\ref{Definition Jacobi final version}). Using the particular
cases
\[
J_{-\frac{1}{2}}(x)=\sqrt{\frac{2}{\pi x}}\,\cos(x),\,\,\,\,\,r_{1}(n)=\begin{cases}
2 & \text{if }n\,\,\text{is a square}\\
0 & \text{otherwise}
\end{cases},
\]
we obtain from (\ref{Definition Jacobi final version}) and (\ref{particular cases formula Cosine})
\begin{equation*}
\psi_{1}(x,z):=\sqrt{2\pi z}\,(\pi x)^{1/4}\,\sum_{n=1}^{\infty}\,\sqrt{n}\,e^{-\pi n^{2}\,x}\,J_{-\frac{1}{2}}(\sqrt{\pi\,x}\,n\,z)=2\,\sum_{n=1}^{\infty}e^{-\pi n^{2}x}\,\cos\left(\sqrt{\pi x}\,n\,z\right):=2\,\psi(x,z).
\end{equation*}

\bigskip{}

From the transformation formula (\ref{transformation formula psi function general}),
we see that the following identity takes place
\begin{equation}
\psi_{\alpha}(x,z)=-1+\frac{e^{-z^{2}/4}}{x^{\alpha/2}}+\frac{e^{-z^{2}/4}}{x^{\alpha/2}}\,\psi_{\alpha}\left(\frac{1}{x},\,iz\right).\label{Transformation formula psi alpha}
\end{equation}
Note again that, when $\alpha=1$, we recover
(\ref{transformation formula Jacobi with parameter z}).

\bigskip{}

Furthermore, we can write (\ref{Transformation formula psi alpha})
in the suitable integral form (\ref{triple identity integral with psi}).
Note that the critical line for $\phi(s):=\pi^{-s}\zeta_{\alpha}(s)$
is the line $\text{Re}(s)=\frac{\alpha}{4}$, so that (\ref{triple identity integral with psi})
yields the representation
\begin{equation}
x^{\frac{\alpha}{4}}\psi_{\alpha}(x,z)-x^{-\alpha/4}\,e^{-\frac{z^{2}}{4}}=e^{-\frac{z^{2}}{4}}x^{-\alpha/4}\,\psi_{\alpha}\left(\frac{1}{x},\,i\,z\right)-x^{\alpha/4}=\frac{1}{2\pi}\,\intop_{-\infty}^{\infty}\eta_{\alpha}\left(\frac{\alpha}{4}+it\right) \,_{1}F_{1}\left(\frac{\alpha}{4}+it;\,\frac{\alpha}{2};\,-\frac{z^{2}}{4}\right)\,x^{-it}\,dt,\label{equation in the setting of zeta alpha}
\end{equation}
where 
\begin{equation}
    \eta_{\alpha}(s)=\pi^{-s}\,\Gamma(s)\,\zeta_{\alpha}(s).
\end{equation}

Representation (\ref{equation in the setting of zeta alpha}) will be of great use in proving Theorem \ref{zeta alpha hypergeometric zeros} below.

\bigskip{}

From Corollary \ref{generalized bessel expansion corollary 2.4} we can derive the curious formula 
\begin{align}
\frac{1}{(x^{2}+y^{2})^{s}}+\sum_{n=1}^{\infty}\frac{r_{\alpha}(n)}{\left(n+x^{2}+y^{2}\right)^{s}}\,_{2}F_{1}\left(\frac{s}{2},\,\frac{s+1}{2};\,\frac{\alpha}{2};\,\frac{4n\,y^{2}}{\left(n+x^{2}+y^{2}\right)^{2}}\right)\nonumber \\
=\frac{\pi^{\frac{\alpha}{2}}\,\Gamma(s-\frac{\alpha}{2})\,x^{\alpha-2s}}{\Gamma(s)}+\frac{2\,\Gamma \left(\frac{\alpha}{2}\right)\,\pi^{s+1-\frac{\alpha}{2}}x^{\frac{\alpha}{2}-s}y^{1-\frac{\alpha}{2}}}{\Gamma(s)}\,\sum_{n=1}^{\infty}\,r_{\alpha}(n)\,n^{\frac{s+1-\alpha}{2}}\,J_{\frac{\alpha}{2}-1}(2\pi\sqrt{n}y)\,K_{s-\frac{\alpha}{2}}\left(2\pi\sqrt{n}\,x\right),\label{analogue watson for r alpha zeta}
\end{align}
valid for every positive pair of positive numbers $x$ and $y$ and
$s$ such $\text{Re}(s)>\sigma_{\alpha}$, with $\sigma_{\alpha}$ being given by (\ref{sigma alpha}).

When $\alpha=1$ we obtain an extension of Watson's formula due to
Kober {[}\cite{kober_epstein}, p.614{]}: since $r_{1}(n)=2$ only when
$n$ is a perfect square, we can use the steps leading to the proof
of Corollary \ref{Watson with cosine twist} to obtain the formula
\begin{equation}
\sum_{n\in\mathbb{Z}}\,\frac{1}{\left(n^{2}+2y\,n+x^{2}+y^{2}\right)^{s}}=\frac{\sqrt{\pi}x^{1-2s}\,\Gamma(s-\frac{1}{2})}{\Gamma(s)}+\frac{4\pi^{s}}{x^{s-\frac{1}{2}}\,\Gamma(s)}\,\sum_{n=1}^{\infty}\,n^{s-\frac{1}{2}}\,\cos(2\pi n\,y)\,K_{s-\frac{1}{2}}\left(2\pi\,n\,x\right),\label{generalized watson here example 1}
\end{equation}
which is valid for every $\text{Re}(s)>\frac{1}{2}$ and positive
$x$ and $y$. By letting $y=0$ in (\ref{generalized watson here example 1}),
one can deduce Watson's formula \cite{watson_reciprocal}. 
\end{example}

\begin{example}\label{example_Epstein}
Let $Q$ be a real, binary and positive definite quadratic form and let $\Delta:=4AC-B^2$ be its discriminant. The
Epstein zeta function attached to $Q$ (\ref{epstein definition intro})
satisfies Hecke's functional equation
\[
\eta_{Q}(s):=\left(\frac{2\pi}{\sqrt{\Delta}}\right)^{-s}\Gamma(s)\,\zeta(s,Q)=\left(\frac{2\pi}{\sqrt{\Delta}}\right)^{-(1-s)}\Gamma(1-s)\,\zeta(1-s,Q):=\eta_{Q}(1-s).
\]

Thus, we can apply the previous formalism to the pair of Dirichlet
series
\begin{equation}
\phi(s)=\psi(s)=\left(\frac{2\pi}{\sqrt{\Delta}}\right)^{-s}\sum_{n=1}^{\infty}\frac{r_{Q}(n)}{n^{s}},\,\,\,\,\text{Re}(s)>1,\label{definition phi(s)}
\end{equation}
where $r_{Q}(n)$ denotes the representation number of $n$ by $Q$.
Since $\zeta(s,Q)$ has a simple pole at $s=1$ with residue $2\pi/\sqrt{\Delta}$,
$\phi(s)$ has a simple pole at $s=1$ with residue $1$. Moreover,
$\phi(0)=-1$.

\bigskip{}

From (\ref{generalized psi for Hecke}), we can write the analogue
of Jacobi's $\psi-$function (\ref{generalized psi for Hecke}) as
follows:
\begin{equation}
\psi_{Q}(x,z)=\sum_{n=1}^{\infty}r_{Q}(n)\,e^{-\frac{2\pi nx}{\sqrt{\Delta}}}\,J_{0}\left(\sqrt{\frac{2\pi nx}{\Delta^{1/2}}}\,z\right),\,\,\,\,\,\,\text{Re}(x)>0,\,\,\,z\in\mathbb{C}.\label{Jacobi Epstein definition}
\end{equation}

\bigskip{}

Using (\ref{generalized Theta reflection}), the transformation formula
for (\ref{Jacobi Epstein definition}) reads
\begin{equation}
\sqrt{x}\,\psi_{Q}(x,z)-\frac{e^{-\frac{z^{2}}{4}}}{\sqrt{x}} =\frac{e^{-\frac{z^{2}}{4}}}{\sqrt{x}}\,\psi_{Q}\left(\frac{1}{x},\,iz\right)-\,\sqrt{x} =\frac{1}{2\pi}\,\intop_{-\infty}^{\infty}\eta_{Q}\left(\frac{1}{2}+it\right) \,_{1}F_{1}\left(\frac{1}{2}+it;\,1;\,-\frac{z^{2}}{4}\right)x^{-it}\,dt,\label{Integral representation Epstein binary}
\end{equation}
where 
\begin{equation}
    \eta_{Q}(s):=\left(\frac{2\pi}{\sqrt \Delta}\right)^{-s}\,\Gamma(s)\,\zeta(s,Q).
\end{equation}
Like the previous formula (\ref{equation in the setting of zeta alpha}),
(\ref{Integral representation Epstein binary}) will be of great importance
in the proof of Theorem \ref{Epstein result}.

Further identities can be obtained. From a straightforward application
of Corollary \ref{generalized bessel expansion corollary 2.4} we can derive the formula,
\begin{align*}
\frac{1}{\left(x^{2}+y^{2}\right)^{s}}+\sum_{n=1}^{\infty}\frac{r_{Q}(n)}{\left(n+x^{2}+y^{2}\right)^{s}}\,_{2}F_{1}\left(\frac{s}{2},\,\frac{s+1}{2};\,1;\,\frac{4ny^{2}}{\left(n+x^{2}+y^{2}\right)^{2}}\right)\\
=\frac{2\pi}{\sqrt{\Delta}}\cdot\,\frac{x^{2-2s}}{s-1}+\frac{2x^{1-s}}{\Gamma(s)}\,\left(\frac{2\pi}{\sqrt{\Delta}}\right)^{s}\,\sum_{n=1}^{\infty}\,r_{Q}(n)\,n^{\frac{s-1}{2}}\,J_{0}\left(4\pi\sqrt{\frac{n}{\Delta}}\,y\right)\,K_{s-1}\left(4\pi\,\sqrt{\frac{n}{\Delta}}\,x\right) & ,
\end{align*}
which seems to be novel.
    
\end{example}

\begin{example} \label{character case example}

Let $\chi$ be a primitive Dirichlet character modulo $q$. If $\chi$
is even, we define $\delta=0$ and if $\chi$ is odd, we take $\delta=1$.
In this example we consider the Dirichlet series
\[
L_{k}\left(s,\chi\right):=\sum_{n_{1},...,n_{k}\in\mathbb{Z}}\frac{n_{1}^{\delta}\,\chi(n_{1})\,...\,n_{k}^{\delta}\,\chi(n_{k})}{\left(n_{1}^{2}+...+n_{k}^{2}\right)^{s}},\,\,\,\,\text{Re}(s)>\frac{k}{2}(1+\delta).
\]

First, let us note that $L_{k}\left(s,\chi\right)$ can be rewritten as
\[
L_{k}\left(s,\chi\right):=\sum_{n_{1},...,n_{k}\in\mathbb{Z}}\frac{n_{1}^{\delta}\,\chi(n_{1})\,...\,n_{k}^{\delta}\,\chi(n_{k})}{\left(n_{1}^{2}+...+n_{k}^{2}\right)^{s}}=\sum_{n=1}^{\infty}\frac{r_{k,\chi}(n)}{n^{s}},\,\,\,\,\,\,\text{Re}(s)>\frac{k}{2}(1+\delta),
\]
where $r_{k,\chi}(n)$ is the character analogue of $r_{k}(n)$,
\begin{equation}
r_{k,\chi}(n):=\sum_{n_{1}^{2}+...+n_{k}^{2}=n}n_{1}^{\delta}\chi(n_{1})\,\cdot...\,\cdot n_{k}^{\delta}\,\chi(n_{k}),\label{general rk(n)}
\end{equation}
where the sum runs over any decomposition of $n$ as a sum of $k$
squares.

\bigskip{}

It will be shown in Subsection \ref{proof of lemma 1.3.} of this paper that $L_{k}(s,\chi)$
can be analytically continued to an entire function satisfying Hecke's
functional equation
\begin{equation}
\left(\frac{\pi}{q}\right)^{-s}\Gamma(s)\,L_{k}\left(s,\chi\right)=\frac{(-i)^{\delta k}G^{k}(\chi)}{q^{k/2}}\,\left(\frac{\pi}{q}\right)^{-\left(k(\frac{1}{2}+\delta)-s\right)}\,\Gamma\left(k\left(\frac{1}{2}+\delta\right)-s\right)\,L_{k}\left(k\left(\frac{1}{2}+\delta\right)-s;\,\overline{\chi}\right).\label{functional equation for Lk s chi}
\end{equation}

\bigskip{}

Therefore, we can apply the previous formulas to the pair of Dirichlet
series
\begin{equation}
\phi(s):=\left(\frac{\pi}{q}\right)^{-s}\sum_{n=1}^{\infty}\frac{r_{k,\chi}(n)}{n^{s}},\,\,\,\,\,\psi(s):=\frac{(-i)^{\delta k}G^{k}(\chi)}{q^{k/2}}\,\left(\frac{\pi}{q}\right)^{-s}\,\sum_{n=1}^{\infty}\frac{r_{k,\overline{\chi}}(n)}{n^{s}}.\label{To apply formalism Dirichlet series}
\end{equation}

From the definition (\ref{generalized psi for Hecke}), we can consider
the analogue of Jacobi's $\psi-$function attached to $L_{k}(s,\chi)$,
\begin{equation}
\psi_{\chi,k}(x,z):=2^{\frac{k}{2}+k\delta-1}\,\Gamma\left(\frac{k}{2}+k\delta\right)\left(\sqrt{\frac{\pi}{q}x}\,z\right)^{1-\frac{k}{2}-k\delta}\,\sum_{n=1}^{\infty}r_{k,\chi}(n)\,n^{\frac{1}{2}-\frac{k}{4}-\frac{k\delta}{2}}\,e^{-\frac{\pi n}{q}\,x}\,J_{\frac{k}{2}+k\delta-1}\left(\sqrt{\frac{\pi}{q}n\,x}\,z\right).\label{analogues character Jacobi psi!}
\end{equation}

When $k=1$, (\ref{analogues character Jacobi psi!}) reduces
to the character analogue of Jacobi's $\psi-$function given at the
end of the paper {[}\cite{DKMZ}, p. 321{]} and in {[}\cite{Riesz_type_criteria},
Theorem 1.3.{]}. Indeed, since $r_{1,\chi}(n)=2\sqrt{n}^{\delta}\chi(\sqrt{n})$
only if $n$ is a perfect square, we see that (\ref{analogues character Jacobi psi!}) gives
\begin{align*}
\psi_{\chi,1}(x,z) & :=2^{\frac{1}{2}+\delta}\,\Gamma\left(\frac{1}{2}+\delta\right)\left(\sqrt{\frac{\pi}{q}x}\,z\right)^{\frac{1}{2}-\delta}\,\sum_{n=1}^{\infty}\chi(n)\,\sqrt{n}\,e^{-\frac{\pi n^{2}}{q}\,x}\,J_{\delta-\frac{1}{2}}\left(\sqrt{\frac{\pi}{q}x}\,n\,z\right)\\
 & =\begin{cases}
2\,\sum_{n=1}^{\infty}\chi(n)\,e^{-\frac{\pi n^{2}}{q}\,x}\,\cos\left(\sqrt{\frac{\pi}{q}x}\,n\,z\right) & \delta=0\\
\frac{2}{z}\,\sqrt{\frac{q}{\pi x}}\,\sum_{n=1}^{\infty}\chi(n)\,e^{-\frac{\pi n^{2}}{q}\,x}\sin\left(\sqrt{\frac{\pi}{q}x}\,n\,z\right) & \delta=1.
\end{cases}
\end{align*}
\bigskip{}

Since $\phi(0)=\rho=0$ for $\phi(s)$ given by (\ref{To apply formalism Dirichlet series}), we can
deduce from our general formulas, 
\begin{equation*}
\,\sqrt{x}\,\sum_{n=1}^{\infty}r_{k,\chi}(n)\,n^{\frac{1}{2}-\frac{k}{4}-\frac{k\delta}{2}}\,e^{-\frac{\pi n}{q}\,x}\,J_{\frac{k}{2}+k\delta-1}\left(\sqrt{\frac{\pi}{q}n\,x}\,z\right)=\frac{(-i)^{\delta k}G^{k}(\chi)}{q^{k/2}}\,\frac{e^{-z^{2}/4}}{\sqrt{x}}\,\sum_{n=1}^{\infty}\,r_{k,\overline{\chi}}(n)\,n^{\frac{1}{2}-\frac{k}{4}-\frac{k\delta}{2}}\,e^{-\frac{\pi n}{q\,x}}\,I_{\frac{k}{2}+k\delta-1}\left(\sqrt{\frac{\pi n}{qx}}\,z\right),
\end{equation*}
which, by virtue of relation (\ref{triple identity integral with psi}),
can be rewritten as
\begin{align}
x^{\frac{k}{4}+\frac{k\delta}{2}}\,\psi_{\chi,k}\left(x,z\right) & =\frac{(-i)^{\delta k}G^{k}(\chi)\,e^{-\frac{z^{2}}{4}}}{q^{k/2}}\,x^{-\frac{k}{4}-\frac{k}{2}\delta}\,\psi_{\overline{\chi},k}\left(\frac{1}{x},iz\right)=\nonumber \\
=\frac{1}{2\pi}\,\intop_{-\infty}^{\infty}\eta_{k}\left(\frac{k}{4}+\frac{k\delta}{2}+it,\chi\right) & \cdot{}_{1}F_{1}\left(\frac{k}{4}+\frac{k\delta}{2}+it;\,\frac{k}{2}+k\delta;\,-\frac{z^{2}}{4}\right)x^{-it}\,dt,\label{great identity for charackter k}
\end{align}
where
\begin{equation}
\eta_{k}(s,\chi):=\left(\frac{\pi}{q}\right)^{-s}\Gamma(s)\,L_{k}\left(s,\chi\right).\label{definition eta chi k}
\end{equation}
\\

From Corollary \ref{generalized bessel expansion corollary 2.4}, we may derive a character analogue of (\ref{analogue watson for r alpha zeta}),
namely: 
\begin{align}
\frac{q^{k/2}i^{\delta k}}{G^{k}(\chi)}\sum_{n=1}^{\infty}\frac{r_{k,\chi}(n)}{\left(n+x^{2}+y^{2}\right)^{s}}\,_{2}F_{1}\left(\frac{s}{2},\,\frac{s+1}{2};\,\frac{k}{2}+k\delta;\,\frac{4y^{2}n}{\left(n+x^{2}+y^{2}\right)^{2}}\right) & =\nonumber \\
=\frac{2\,\Gamma\left(\frac{k}{2}+k\delta\right)}{\Gamma(s)\,y^{\frac{k}{2}+k\delta-1}\,x^{s-\frac{k}{2}-k\delta}}\,\left(\frac{\pi}{q}\right)^{1+s-\frac{k}{2}-k\delta}\,\sum_{n=1}^{\infty}r_{k,\overline{\chi}}(n)\,n^{\frac{s+1-k}{2}-k\delta}\,J_{\frac{k}{2}+k\delta-1}\left(\frac{2\pi\sqrt{n}}{q}y\right)\, & K_{s-\frac{k}{2}-k\delta}\left(\frac{2\pi\sqrt{n}}{q}\,x\right).\label{Watson like for rk chi}
\end{align}

If $k=1$ and $\delta=0$, we can also appeal to the particular case
(\ref{particular case hypergeometric gauss}) to obtain the character
analogue of Watson's formula, valid for $\text{Re}(s)>\frac{1}{2}$,
\[
\sum_{n\in\mathbb{Z}}\frac{\chi(n)}{\left(n^{2}+2yn+x^{2}+y^{2}\right)^{s}}=\frac{4\,x^{\frac{1}{2}-s}\,G(\chi)}{\Gamma(s)\sqrt{\pi}}\,\left(\frac{\pi}{q}\right)^{\frac{1}{2}+s}\,\sum_{n=1}^{\infty}\overline{\chi}(n)\,n^{s-\frac{1}{2}}\,\cos\left(\frac{2\pi n}{q}y\right)\,K_{s-\frac{1}{2}}\left(\frac{2\pi n}{q}x\right).
\]

By letting $y=0$ above, we rederive a formula due to Berndt,
Dixit and Sohn {[}\cite{koshliakov_ramanujan_character}, p. 56, eq. (2.9.){]}.

\bigskip{}

When $k=1$ and $\delta=1$, we see that $L_{1}\left(s,\chi\right)=2\,L(2s-1,\chi)$,
with $\chi$ being an odd Dirichlet character. In this case, (\ref{Watson like for rk chi})
implies the identity:
\[
\sum_{n\in\mathbb{Z}}\frac{\chi(n)}{\left(n^{2}+2yn+x^{2}+y^{2}\right)^{s}}=-\frac{4i\,x^{\frac{1}{2}-s}\,G(\chi)}{\Gamma(s)\sqrt{\pi}}\,\left(\frac{\pi}{q}\right)^{s+\frac{1}{2}}\,\sum_{n=1}^{\infty}\overline{\chi}(n)\,n^{s-\frac{1}{2}}\sin\left(\frac{2\pi n}{q}y\right)\,K_{s-\frac{1}{2}}\left(\frac{2\pi n}{q}x\right),
\]
valid for $\text{Re}(s)>\frac{1}{2}$. This is another analogue of Watson's formula. 

\end{example}
\bigskip{}

\begin{example}\label{example cusp forms}
Let $f(\tau)$ be a holomorphic cusp form with weight $k\geq12$ for
the full modular group and $L_{f}(s)$ the associated $L-$function,
\begin{equation}
L_{f}(s)=\sum_{n=1}^{\infty}\frac{a_{f}(n)}{n^{s}},\,\,\,\,\,\,\text{Re}(s)>\frac{k+1}{2}.\label{cusp form holomorphic def}
\end{equation}

From Lemma \ref{functional equation cusp} we know $L_{f}(s)$ can be analytically continued
as an entire function with functional equation 
\begin{equation}
(2\pi)^{-s}\Gamma(s)\,L_{f}(s)=(-1)^{k/2}\,\left(2\pi\right)^{-(k-s)}\,\Gamma(k-s)\,L_{f}(k-s).\label{functional equation cusp form}
\end{equation}

From the previous results we can conceive the analogue of Jacobi's $\psi-$function
(\ref{generalized psi for Hecke}) in the following form 
\begin{equation}
\psi_{f}(x,z):=(k-1)!\,\left(\sqrt{\frac{\pi x}{2}}\,z\right)^{1-k}\sum_{n=1}^{\infty}a_{f}(n)n^{\frac{1-k}{2}}\,e^{-2\pi n\,x}\,J_{k-1}\left(\sqrt{2\pi n\,x}\,z\right),\,\,\text{Re}(x)>0,\,\,z\in\mathbb{C}.\label{Jacobi theta function cusp forms}
\end{equation}

Moreover, we can establish the formula
\begin{equation}
x^{\frac{k}{2}}\psi_{f}(x,z)=(-1)^{k/2}\,e^{-\frac{z^{2}}{4}}x^{-k/2}\,\psi_{f}\left(\frac{1}{x},\,iz\right)=\frac{1}{2\pi}\,\intop_{-\infty}^{\infty}\eta_{f}\left(\frac{k}{2}+it\right)\,_{1}F_{1} \left(\frac{k}{2}+it;\,k;\,-\frac{z^{2}}{4}\right)\,x^{-it}dt,\label{Reflection formula for cusp forms}
\end{equation}
where
\begin{equation}
\eta_{f}(s):=(2\pi)^{-s}\Gamma(s)\,L_{f}(s).\label{xi function cusp forms}
\end{equation}

Also, for $\text{Re}(s)>\frac{k+1}{2}$ and $x,y>0$, the following particular case of Corollary \ref{generalized bessel expansion corollary 2.4} holds: 
\begin{align}
\sum_{n=1}^{\infty}\frac{a_{f}(n)}{\left(n+x^{2}+y^{2}\right)^{s}}\,_{2}F_{1}\left(\frac{s}{2},\,\frac{s+1}{2};\,k;\,\frac{4n\,y^{2}}{\left(n+x^{2}+y^{2}\right)^{2}}\right)\nonumber \\
=\frac{2\,(2\pi)^{s+1-k}\,(k-1)!\,(-1)^{k/2}}{\,\Gamma(s)\,y^{k-1}x^{s-k}}\,\sum_{n=1}^{\infty}\,a_{f}(n)\,n^{\frac{s+1}{2}-k}\,J_{k-1}(4\pi\sqrt{n}\,y)\, & K_{s-k}\left(4\pi\sqrt{n}\,x\right).\label{Cusp form Watson Identity}
\end{align}

For example, when $L_{f}(s)$ is the Dirichlet series associated with
Ramanujan's $\tau-$function,
\[
L_{\tau}(s)=\sum_{n=1}^{\infty}\frac{\tau(n)}{n^{s}},\,\,\,\,\,\,\text{Re}(s)>\frac{13}{2},
\]
we obtain from (\ref{Cusp form Watson Identity})
\begin{align*}
\sum_{n=1}^{\infty}\frac{\tau(n)}{\left(n+x^{2}+y^{2}\right)^{s}}\,_{2}F_{1}\left(\frac{s}{2},\,\frac{s+1}{2};\,12;\,\frac{4ny^{2}}{\left(n+x^{2}+y^{2}\right)^{2}}\right)\\
=\frac{2\times11!\times(2\pi)^{s-11}}{\Gamma(s)\,y^{11}x^{s-12}}\,\sum_{n=1}^{\infty}\,\tau(n)\,n^{\frac{s-23}{2}}\,J_{11}(4\pi\sqrt{n}\,y)\,K_{s-12}\left(4\pi\sqrt{n}\,x\right), & \,\,\,x,y>0,\,\,\,\text{Re}(s)>\frac{13}{2},
\end{align*}
which appears to be new (see [\cite{berndt_general_bessel}, eq. (13.1)]
for a companion formula).
   
\end{example}

\begin{example}\label{cusp forms exponential twist}
Let $f(\tau)$ be a holomorphic cusp form of weight $k$ for the full
modular group and $a_{f}(n)$ its Fourier coefficients. Also, assume
that $p,q$ are integers such that $(p,q)=1$ and consider the Dirichlet
series
\[
L_{f}(s,p/q):=\sum_{n=1}^{\infty}\frac{a_{f}(n)\,e^{\frac{2\pi ip}{q}n}}{n^{s}},\,\,\,\,\,\text{Re}(s)>\frac{k+1}{2}.
\]
We know from Lemma \ref{functional equation cusp} that $L_{f}(s,p/q)$ is an entire function
and satisfies the functional equation:
\[
\left(\frac{2\pi}{q}\right)^{-s}\Gamma(s)\,L_{f}\left(s,\frac{p}{q}\right)=(-1)^{k/2}\,\left(\frac{2\pi}{q}\right)^{-(k-s)}\Gamma(k-s)\,L_{f}\left(k-s,\,-\overline{p}/q\right),
\]
where $\overline{p}$ is such that $p\,\overline{p}\equiv1 \mod q$.   

From this,  we can create the analogue of Jacobi's $\psi-$function (\ref{generalized psi for Hecke})
as follows 
\begin{equation}
\psi_{f,p/q}(x,z):=(k-1)!\,\left(\sqrt{\frac{\pi x}{2q}}\,z\right)^{1-k}\,\sum_{n=1}^{\infty}a_{f}(n)\,e^{\frac{2\pi ip}{q}n}\,n^{\frac{1-k}{2}}\,e^{-\frac{2\pi n}{q}\,x}\,J_{k-1}\left(\sqrt{\frac{2\pi n}{q}x}\,z\right),\label{analogue Jacobi cusp form twisted}
\end{equation}
defined for $\text{Re}(x)>0$ and $z\in\mathbb{C}$.

\bigskip{}

Clearly, the transformation formula for $\psi_{f,p/q}(x,z)$ (\ref{transformation formula psi function general}) is explicitly given by 
\begin{equation}
\sum_{n=1}^{\infty}a_{f}(n)\,e^{\frac{2\pi ip}{q}n}\,n^{\frac{1-k}{2}}\,e^{-\frac{2\pi n}{q}\,x}\,J_{k-1}\left(\sqrt{\frac{2\pi n}{q}x}\,z\right)
=\frac{(-1)^{k/2}\,e^{-z^{2}/4}}{x}\,\sum_{n=1}^{\infty}a_{f}(n)\,e^{-\frac{2\pi i\overline{p}}{q}n}\,n^{\frac{1-k}{2}}\,e^{-\frac{2\pi n}{qx}}\,I_{k-1}\left(\sqrt{\frac{2\pi n}{qx}}\,z\right),\label{direct identity cusp form twisted}
\end{equation}
and it admits the integral representation, 
\begin{equation}
x^{k/2}\psi_{f,p/q}(x,z)=\frac{e^{-z^{2}/4}(-1)^{k/2}}{x^{k/2}}\,\psi_{f,-\frac{\overline{p}}{q}}\left(\frac{1}{x},iz\right)=\frac{1}{2\pi}\,\intop_{-\infty}^{\infty}\eta_{f}\left(\frac{k}{2}+it,\frac{p}{q}\right) \,_{1}F_{1}\left(\frac{k}{2}+it;\,k;\,-\frac{z^{2}}{4}\right)\,x^{-it}dt,\label{integral representation for cusp form twisted}
\end{equation}
where
\begin{equation}
\eta_{f}\left(s,\frac{p}{q}\right):=\left(\frac{2\pi}{q}\right)^{-s}\Gamma(s)\,L_{f}\left(s,\frac{p}{q}\right).\label{definition xi cusp forms twisted}
\end{equation}
\end{example}

\medskip{}

\begin{center}\section{Zeros of combinations attached to $\zeta_{\alpha}(s)$}\end{center}

\bigskip{}

\bigskip{}

    \subsection{Dirichlet series attached to powers of the Theta functions}

\bigskip{}

At the core of Hardy's proof of his Theorem is the fact that Jacobi's
$\theta-$function satisfies
\begin{equation}
\lim_{\omega\rightarrow\frac{\pi}{4}^{-}}\frac{d^{n}}{d\omega^{n}}\,\theta\left(e^{2i\omega}\right)=0\,\,\,\,\forall n\in\mathbb{N}_{0}.\label{behavior theta expected Hardy proof}
\end{equation}

\bigskip{}
If we replace the traditional role of $\theta(x)$ in (\ref{behavior theta expected Hardy proof}) by the analogue
$\psi_{\alpha}(x,z)$, we will see below that the study of the limit
(\ref{behavior theta expected Hardy proof}) can be made through the evaluation of 
\[
\lim_{\delta\rightarrow0^{+}}\psi_{\alpha}(i+\delta,z).
\]

However, to inspect $\psi_{\alpha}(i+\delta,z)$ is the same as to
study the series
\begin{equation}
2^{\frac{\alpha}{2}-1}\,\Gamma\left(\frac{\alpha}{2}\right)\,\left(\sqrt{\pi(i+\delta)}\,z\right)^{1-\frac{\alpha}{2}}\,\sum_{n=1}^{\infty}(-1)^{n}\,r_{\alpha}(n)\,n^{\frac{1}{2}-\frac{\alpha}{4}}\,e^{-\pi n\,\delta}\,J_{\frac{\alpha}{2}-1}(\sqrt{\pi\,n\,(i+\delta)}\,z)\label{drafted expression i+delta}
\end{equation}
as $\delta\rightarrow0^{+}$. For $\alpha=1$, this study was done
in {[}\cite{DKMZ}, p. 314{]} and it represents an easier case because $r_{1}(n)=2$ if $n$ is a perfect square
and it is zero otherwise. Indeed, the expression (\ref{drafted expression i+delta})
reduces to
\[
\psi_{1}(i+\delta,z)=2\,\sum_{n=1}^{\infty}(-1)^{n}e^{-\pi n^{2}\delta}\,\cos\left(\sqrt{\pi(i+\delta)}\,n\,z\right),
\]
which can be evaluated simply as (once we separate the previous series
for $n$ even and $n$ odd)
\[
\psi_{1}(i+\delta,z)=2\,\psi_{1}\left(4\delta,\sqrt{\frac{i+\delta}{\delta}}\,z\right)-\psi_{1}\left(\delta,\sqrt{\frac{i+\delta}{\delta}}\,z\right).
\]

Thus, the study of the limit $\delta\rightarrow0^{+}$ is complete once we apply the transformation formula (\ref{Transformation formula psi alpha}) for $\psi_{1}(x,z)$!

\bigskip{}

By looking at (\ref{drafted expression i+delta}), one can see that it may be difficult to employ the previous trick
(i.e., to separate the sum for even and odd $n$) for general $\alpha>0$, and so we need to
have more information about the coefficients $(-1)^{n}r_{\alpha}(n)$
and the Dirichlet series attached to them. This study will be developed through a sequence of important lemmas. 

\bigskip{}

These lemmas concern the powers of certain variants of Jacobi's theta function. 
Like in the case of $\vartheta_{3}(\tau)$ (\ref{power theta}), we
can consider arbitrary powers of the remaining thetanulls, 
\begin{equation}
\vartheta_{2}(\tau)=2\,\sum_{n=0}^{\infty}e^{\pi i\tau\left(n+\frac{1}{2}\right)^{2}},\,\,\,\,\vartheta_{4}(\tau)=1+2\,\sum_{n=1}^{\infty}(-1)^{n}e^{\pi i\tau n^{2}},\,\,\,\,\text{Im}(\tau)>0,\label{theta nuuuullls}
\end{equation}
and study their expansions at $i\infty$ by taking 
\[
\theta_{2}(x):=\vartheta_{2}(ix)=2\,\sum_{n=0}^{\infty}e^{-\pi\left(n+\frac{1}{2}\right)^{2}\,x},\,\,\,\,\theta_{4}(x):=\vartheta_{4}(ix)=1+2\sum_{n=1}^{\infty}(-1)^{n}e^{-\pi n^{2}x}.
\]

It follows from Jacobi's formula (\ref{transformation formula Jacobi with parameter z}) that $\theta_{2}(x)$ and $\theta_{4}(x)$
are connected via the transformation formula
\begin{equation}
\theta_{2}\left(\frac{1}{x}\right)=\sqrt{x}\,\theta_{4}\left(x\right).\label{transformation formula theta nulls}
\end{equation}
\\

Arithmetical functions akin to $r_{\alpha}(n)$ arise from the study of the powers of $\theta_{2}(x)$ and $\theta_{4}(x)$. 
For example, modifying slightly the computations in (\ref{power theta})
and (\ref{construction power series ralpha}), we can define the Fourier
coefficients of $\theta_{2}^{\alpha}(x)$ as follows:
\begin{align}
\vartheta_{2}^{\alpha}(ix) & :=\theta_{2}^{\alpha}(x)=\left(2\,\sum_{n=0}^{\infty}e^{-\pi x\left(n+\frac{1}{2}\right)^{2}}\right)^{\alpha}=\left(2\,e^{-\frac{\pi x}{4}}+2\,\sum_{n=1}^{\infty}e^{-\pi x(n+\frac{1}{2})^{2}}\right)^{\alpha}\nonumber \\
 & =2^{\alpha}e^{-\frac{\pi\alpha x}{4}}\left(1+\,\sum_{n=1}^{\infty}e^{-\pi x\left(n^{2}+n\right)}\right)^{\alpha}=2^{\alpha}\,e^{-\frac{\pi\alpha x}{4}}\,\sum_{j=0}^{\infty}\left(\begin{array}{c}
\alpha\\
j
\end{array}\right)\,\left(\sum_{n=1}^{\infty}e^{-\pi x\left(n^{2}+n\right)}\right)^{j}\nonumber \\
 & :=\sum_{m=0}^{\infty}\tilde{r}_{\alpha}(m)\,e^{-\pi\left(m+\frac{\alpha}{4}\right)x}\label{definition ralpha tilde}
\end{align}
where we have taken the new variable of summation as $m:=n_{1}(n_{1}+1)+...+n_{j}(n_{j}+1)$.
Analogously to $r_{\alpha}(n)$, it is possible to show that the coefficients
$\tilde{r}_{\alpha}(n)$ grow polynomially\footnote{We do not necessarily need a bound of the form (\ref{estimate r alpha (n) useful}) to proceed with our
reasoning. It suffices to check that the argument by Lagarias and
Rains [\cite{lagarias_reins}, p. 18, Theorem 3.3.{]} works for $\vartheta_{2}(\tau)$. See Lemma \ref{Lagarias like Lemma} below.} with $n$. Note also that $\tilde{r}_{\alpha}(0)=2^{\alpha}$ by construction. 

\bigskip{}

Henceforth, it is meaningful to introduce a Dirichlet series attached
to the coefficients $\tilde{r}_{\alpha}(n)$, this is, for some finite $\tilde{\sigma}_{\alpha}$, we define
\begin{equation}
\tilde{\zeta}_{\alpha}(s):=\sum_{n=0}^{\infty}\frac{\tilde{r}_{\alpha}(n)}{\left(n+\frac{\alpha}{4}\right)^{s}},\,\,\,\,\,\,\text{Re}(s)>\tilde{\sigma}_{\alpha}.\label{tilDirichlet definition}
\end{equation}

We can also introduce a Dirichlet series attached to any positive
power of $\theta_{4}(x)$. Proceed once more as in (\ref{power theta}):
\begin{align*}
\vartheta_{4}^{\alpha}(ix):=\theta_{4}^{\alpha}(x) & =\left(1+2\,\sum_{n=1}^{\infty}(-1)^{n}e^{-\pi n^{2}x}\right)^{\alpha}\\
 & =1+\sum_{j=1}^{\infty}\left(\begin{array}{c}
\alpha\\
j
\end{array}\right)\,2^{j}\,\left(\sum_{n=1}^{\infty}(-1)^{n}e^{-\pi n^{2}x}\right)^{j}\\
 & =1+\sum_{j=1}^{\infty}\left(\begin{array}{c}
\alpha\\
j
\end{array}\right)\,2^{j}\,\sum_{n_{1},...,n_{j}=1}^{\infty}(-1)^{n_{1}+n_{2}+...+n_{j}}\,e^{-\pi(n_{1}^{2}+...+n_{j}^{2})\,x}\\
 & =1+\sum_{j=1}^{\infty}\left(\begin{array}{c}
\alpha\\
j
\end{array}\right)\,2^{j}\,\sum_{n_{1},...,n_{j}=1}^{\infty}(-1)^{n_{1}^{2}+n_{2}^{2}+...+n_{j}^{2}}\,e^{-\pi(n_{1}^{2}+...+n_{j}^{2})\,x}.
\end{align*}

If we now define a new variable of summation $m=n_{1}^{2}+...+n_{j}^{2}$,
then the finite sum over $j$ is exactly the same as we have obtained
in defining (\ref{construction power series ralpha}) but with an
extra factor containing $(-1)^{m}$. Following (\ref{construction power series ralpha}),
we see that
\begin{align}
\theta_{4}^{\alpha}(x) & =1+\sum_{j=1}^{\infty}\left(\begin{array}{c}
\alpha\\
j
\end{array}\right)\,2^{j}\,\sum_{n_{1},...,n_{j}=1}^{\infty}(-1)^{n_{1}^{2}+n_{2}^{2}+...+n_{j}^{2}}\,e^{-\pi(n_{1}^{2}+...+n_{j}^{2})\,x}\nonumber \\
 & =1+\sum_{m=1}^{\infty}(-1)^{m}\,r_{\alpha}(m)\,e^{-\pi mx}.\label{powers theta null 4}
\end{align}

The computations above show that $(-1)^{n}r_{\alpha}(n)$ are the
coefficients of the Fourier expansion of $\vartheta_{4}^{\alpha}(\tau)$
at the cusp $i\infty$.

\bigskip{}

The Dirichlet series attached to these coefficients is then: 
\begin{equation}
\zeta_{\alpha}^{\star}(s):=\sum_{n=1}^{\infty}\frac{(-1)^{n}\,r_{\alpha}(n)}{n^{s}},\,\,\,\,\text{Re}(s)>\sigma_{\alpha}:=\begin{cases}
\alpha/2 & \text{if }\alpha\geq4\\
1+\frac{\alpha}{4} & \text{if }0<\alpha<4.
\end{cases}\label{definition Dirichlet series star}
\end{equation}

\bigskip{}

In the following subsection, we connect the Dirichlet series (\ref{tilDirichlet definition}) and
(\ref{definition Dirichlet series star}) through a functional equation and prove a new summation formula for their coefficients. 
\\

\begin{center}
    \subsection{Summation formulas}
\end{center}

\bigskip{}

To establish the connection between the Dirichlet series $\zeta_{\alpha}^{\star}(s)$
and $\tilde{\zeta}_{\alpha}(s)$, we need to assure that the latter
is actually a Dirichlet series, i.e., that $\tilde{\sigma}_{\alpha}$
in (\ref{tilDirichlet definition}) is a finite number. The next lemma gives an estimate
for $\tilde{r}_{\alpha}(n)$ similar in spirit to the bounds found
in \cite{lagarias_reins}.
\\

\begin{lemma}\label{Lagarias like Lemma}

Let $\alpha>0$ and $\tilde{r}_{\alpha}(n)$ be defined by (\ref{definition ralpha tilde}). Then, for any $n\in\mathbb{N}$, $\tilde{r}_{\alpha}(n)$ satisfies the
estimate
\begin{equation}
|\tilde{r}_{\alpha}(n)|<540\,(2n)^{\alpha/2}.\label{estimate tilde ralpha}
\end{equation}

Therefore, the Dirichlet series defined by (\ref{tilDirichlet definition}) converges absolutely
in the right half-plane $\text{Re}(s)>\tilde{\sigma}_{\alpha}:=\frac{\alpha}{2}+1$.
    
\end{lemma}

\begin{proof}

Our proof is nothing but a simple adaptation of {[}\cite{lagarias_reins},
Theorem 3.3., p. 19, eq. (3.12){]}. Write $\varphi_{2}(q)=2\,\sum_{n=0}^{\infty}q^{(n+\frac{1}{2})^{2}}$
for $|q|<1$: invoking Cauchy's integral formula, the following bound
takes place
\begin{align*}
|\tilde{r}_{\alpha}(n)| & =\frac{1}{2\pi}\left|\intop_{-\pi}^{\pi}\varphi_{2}^{\alpha}\left(Re^{i\phi}\right)\,R^{-n}e^{-in\phi}\,d\phi\right|\leq R^{-n}\,\max_{\phi\in[-\pi,\pi]}\left|\varphi_{2}^{\alpha}\left(Re^{i\phi}\right)\right|\\
 & =R^{-n}\,\left(\max_{\phi\in[-\pi,\pi]}\left|\varphi_{2}\left(Re^{i\phi}\right)\right|\right)^{\alpha},\,\,\,\,\,0<R<1,
\end{align*}
since $\alpha>0$ by hypothesis. Clearly, for every $\phi\in[-\pi,\pi]$,
\[
\left|\varphi_{2}\left(Re^{i\phi}\right)\right|\leq\varphi_{2}\left(R\right),
\]
which immediately gives
\begin{equation}
|\tilde{r}_{\alpha}(n)|\leq R^{-n}\,\varphi_{2}^{\alpha}\left(R\right),\label{bound tilde ralpha}
\end{equation}
for any choice of $R\in(0,1)$. For each $n$, take $R:=R(n)=e^{-\pi A/n}$
for a fixed $A>1$. Then (\ref{bound tilde ralpha}) yields
\begin{align}
|\tilde{r}_{\alpha}(n)| & \leq e^{\pi A}\,\varphi_{2}^{\alpha}\left(e^{-\pi A/n}\right)=e^{\pi A}\,\theta_{2}^{\alpha}\left(\frac{A}{n}\right)=e^{\pi A}\,\left(\frac{n}{A}\right)^{\alpha/2}\,\theta_{4}^{\alpha}\left(\frac{n}{A}\right)\nonumber \\
 & =\frac{e^{\pi A}}{A^{\alpha/2}}\,\left\{ 1+2\sum_{k=1}^{\infty}(-1)^{k}e^{-\pi k^{2}\frac{n}{A}}\right\} ^{\alpha}n^{\alpha/2}\leq\frac{e^{\pi A}\,n^{\alpha/2}}{A^{\alpha/2}}\left\{ 1+\frac{2}{e^{\pi n/A}-1}\right\}^{\alpha} \nonumber \\
 & \leq\frac{e^{\pi A}}{A^{\alpha/2}}\,\left\{ 1+\frac{2}{e^{\pi/A}-1}\right\}^{\alpha} \,n^{\alpha/2}.\label{bunch of inequalities}
\end{align}
where we have used (\ref{transformation formula theta nulls}) together
with elementary bounds on the series defining $\theta_{4}(x)$. The
desired estimate (\ref{estimate tilde ralpha}) now follows once we
take $A=2$ in (\ref{bunch of inequalities}).    
\end{proof}

\bigskip{}

Our next lemma gives a functional equation for the zeta functions $\zeta_{\alpha}^{\star}(s)$
and $\tilde{\zeta}_{\alpha}(s)$. 

\begin{lemma}\label{functional equation tilde and star}
Let $r_{\alpha}(n)$ and $\tilde{r}_{\alpha}(n)$ be given by (\ref{construction power series ralpha})
and (\ref{definition ralpha tilde}) respectively.

\bigskip{}

Then the Dirichlet series defined by (\ref{definition Dirichlet series star})
can be analytically continued to an entire function satisfying Hecke's
functional equation

\begin{equation}
\pi^{-s}\Gamma(s)\,\zeta_{\alpha}^{\star}(s)=\pi^{-\left(\frac{\alpha}{2}-s\right)}\Gamma\left(\frac{\alpha}{2}-s\right)\,\tilde{\zeta}_{\alpha}\left(\frac{\alpha}{2}-s\right),\label{functional equation!}
\end{equation}
where $\tilde{\zeta}_{\alpha}(s)$ is given by (\ref{tilDirichlet definition}).
\end{lemma}

\begin{proof}
We have deduced above that $(-1)^{n}r_{\alpha}(n)$ are the coefficients
of the Fourier expansion of $\vartheta_{4}^{\alpha}(\tau)$ at the
cusp $i\infty$. It is clear from this that the Mellin transform holds
\[
\pi^{-s}\Gamma\left(s\right)\,\zeta_{\alpha}^{\star}(s)=\intop_{0}^{\infty}x^{s-1}\left\{ \theta_{4}^{\alpha}(x)-1\right\} \,dx,\,\,\,\,\,\text{Re}(s)>\sigma_{\alpha},
\]
where $\sigma_{\alpha}$ is explicitly given by (\ref{definition Dirichlet series star}).
By following Riemann's paper \cite{riemann, titchmarsh_zetafunction}, we study the
functional equation for $\zeta_{\alpha}^{\star}(s)$: indeed, for
$\text{Re}(s)>\sigma_{\alpha}$,
\begin{equation}
\pi^{-s}\Gamma(s)\,\zeta_{\alpha}^{\star}(s) =\intop_{0}^{\infty}x^{s-1}\left\{ \theta_{4}^{\alpha}(x)-1\right\} \,dx=\intop_{0}^{1}x^{s-1}\left\{ \theta_{4}^{\alpha}(x)-1\right\} \,dx +\intop_{1}^{\infty}x^{s-1}\left\{ \theta_{4}^{\alpha}(x)-1\right\} \,dx.\label{decomposing Integral}
\end{equation}

Using a particular case of Jacobi's formula (\ref{transformation formula Jacobi with parameter z}) given in (\ref{transformation formula theta nulls}), 
\[
\theta_{4}^{\alpha}(1/x)-1=x^{\alpha/2}\,\theta_{2}^{\alpha}(x)-1,
\]
we see that the first integral on the right of (\ref{decomposing Integral}) can be written as: 
\begin{equation}
\intop_{0}^{1}x^{s-1}\left\{ \theta_{4}^{\alpha}(x)-1\right\} \,dx =\intop_{1}^{\infty}x^{-s-1}\left\{ \theta_{4}^{\alpha}\left(1/x\right)-1\right\} \,dx=\intop_{1}^{\infty}x^{\frac{\alpha}{2}-s-1}\theta_{2}^{\alpha}(x)\,dx -\frac{1}{s}.\label{After Jacobi's formula}
\end{equation}

Combining (\ref{decomposing Integral}) and (\ref{After Jacobi's formula})
leads to the representation

\begin{equation}
\pi^{-s}\Gamma\left(s\right)\,\zeta_{\alpha}^{\star}(s)=\intop_{1}^{\infty}\left[x^{s-1}\left\{ \theta_{4}^{\alpha}(x)-1\right\} +x^{\frac{\alpha}{2}-s-1}\,\theta_{2}^{\alpha}(x)\right]\,dx-\frac{1}{s}.\label{analytic continuation zeta star}
\end{equation}

\bigskip{}

Since the integral on the right-hand side of (\ref{analytic continuation zeta star})
converges absolutely for every $s\in\mathbb{C}$, it represents an
entire function of $s\in\mathbb{C}$. Hence, the previous representation
gives the analytic continuation of $\zeta_{\alpha}^{\star}(s)$ as
an entire function. Furthermore, since $\Gamma(s)$ has a simple pole
at $s=0$, it is clear to see that $\zeta_{\alpha}^{\star}(0)=-1$.

\bigskip{}

We can do the same kind of computations with $\tilde{\zeta}_{\alpha}(s)$. We find that, for $\text{Re}(s)>\tilde{\sigma}_{\alpha}:=\frac{\alpha}{2}+1$ (see (\ref{tilDirichlet definition}) above)
\begin{align}
\pi^{-s}\Gamma(s)\,\tilde{\zeta}_{\alpha}(s) & =\intop_{0}^{\infty}x^{s-1}\theta_{2}^{\alpha}(x)\,dx=\intop_{0}^{1}x^{s-1}\theta_{2}^{\alpha}(x)\,dx+\intop_{1}^{\infty}x^{s-1}\theta_{2}^{\alpha}(x)\,dx\nonumber \\
 & =\intop_{1}^{\infty}x^{-s-1}\theta_{2}^{\alpha}(1/x)\,dx+\intop_{1}^{\infty}x^{s-1}\,\theta_{2}^{\alpha}(x)\,dx\nonumber \\
 & =\intop_{1}^{\infty}x^{-s-1}\left(x^{\alpha/2}\left(\theta_{4}^{\alpha}(x)-1\right)+x^{\alpha/2}\right)\,dx+\intop_{1}^{\infty}x^{s-1}\theta_{2}^{\alpha}(x)\,dx\nonumber \\
 & =\intop_{1}^{\infty}\left(x^{\frac{\alpha}{2}-s-1}\left(\theta_{4}^{\alpha}(x)-1\right)+x^{s-1}\theta_{2}^{\alpha}(x)\right)\,dx+\intop_{1}^{\infty}x^{-s-1+\frac{\alpha}{2}}dx\nonumber \\
 & =\intop_{1}^{\infty}\left(x^{\frac{\alpha}{2}-s-1}\left(\theta_{4}^{\alpha}(x)-1\right)+x^{s-1}\theta_{2}^{\alpha}(x)\right)\,dx+\frac{1}{s-\frac{\alpha}{2}}.\label{tilde zeta}
\end{align}

From this we can see that $\tilde{\zeta}_{\alpha}(s)$ can be continued
to the complex plane as a meromorphic function having a simple pole
at $s=\frac{\alpha}{2}$ with residue $\pi^{\alpha/2}/\Gamma(\alpha/2)$.
Furthermore, the right-hand sides of (\ref{analytic continuation zeta star})
and (\ref{tilde zeta}) can be turned into one another under the reflection
$s\longleftrightarrow\frac{\alpha}{2}-s$. This implies the functional equation (\ref{functional equation!})
and completes the proof.    
\end{proof}

\bigskip{}

Since the functional equation (\ref{functional equation!}) holds,
we might expect that the summation formula (\ref{final formula for 1f1 theorem})
must hold as well. The next lemma gives precisely this.

\begin{lemma}\label{summation formula (-1) ralpha}
Let $r_{\alpha}(n)$ be the coefficients of the series expansion of
$\theta^{\alpha}(x)-1$ and $\tilde{r}_{\alpha}(n)$ be defined by
(\ref{definition ralpha tilde}). Then for $\text{Re}(x)>0$ and any
$y\in\mathbb{C}$, the following identity holds
\begin{align}
\sum_{n=1}^{\infty}(-1)^{n}\,r_{\alpha}(n)\,n^{\frac{1}{2}-\frac{\alpha}{4}}\,e^{-\pi nx}\,J_{\frac{\alpha}{2}-1}(\sqrt{\pi\,n}\,y) & =-\frac{y^{\frac{\alpha}{2}-1}\pi^{\frac{\alpha}{4}-\frac{1}{2}}}{2^{\frac{\alpha}{2}-1}\Gamma(\alpha/2)}+\nonumber \\
+\,\frac{e^{-\frac{y^{2}}{4x}}}{x}\,\sum_{n=0}^{\infty}\tilde{r}_{\alpha}(n)\,\left(n+\frac{\alpha}{4}\right)^{\frac{1}{2}-\frac{\alpha}{4}}\,e^{-\frac{\pi}{x}\,\left(n+\frac{\alpha}{4}\right)} & I_{\frac{\alpha}{2}-1}\left(\frac{\sqrt{\pi(n+\frac{\alpha}{4})}\,y}{x}\right).\label{final formula for 1f1 theorem-1}
\end{align}
   
\end{lemma}

\begin{proof}
By the previous result, we have that the pair of Dirichlet series 
\[
\phi(s)=\pi^{-s}\sum_{n=1}^{\infty}\frac{(-1)^{n}r_{\alpha}(n)}{n^{s}},\,\,\,\,\psi(s)=\pi^{-s}\sum_{n=0}^{\infty}\frac{\tilde{r}_{\alpha}(n)}{\left(n+\frac{\alpha}{4}\right)^{s}},
\]
satisfies Hecke's functional equation (\ref{Hecke Dirichlet series Functional}). 
Thus, by an application of our summation formulas (\ref{final formula for 1f1 theorem})
or (\ref{generalized Theta reflection}) to this pair, we just need to take the simple substitutions $r=\frac{\alpha}{2}$, $\rho=0$
(because $\phi(s)$ is entire by the previous lemma) and $\phi(0)=-1$,
obtaining (\ref{final formula for 1f1 theorem-1}).    
\end{proof}

\begin{center}
\subsection{The behavior of $\psi_{\alpha}(x,z)$} \label{vanishing theta Theorem 1.1.}    
\end{center}

Using the previous Lemma \ref{summation formula (-1) ralpha} we are now able to show that $\psi_{\alpha}(x,z)$ has the same kind of behavior as the classical Jacobi's theta function when $x$ approaches the imaginary axis. We now establish one of the most important results of our paper. 

\bigskip{}

\begin{lemma}\label{Vanishing Theta lemma}
Let $\alpha>0$ and $r_{\alpha}(n)$ be defined as the coefficients
of the Fourier expansion of $\theta^{\alpha}(x)$ (\ref{second definition varthet coefficients}). For $\text{Re}(x)>0$
and $z\in\mathbb{C}$, let $\psi_{\alpha}(x,z)$ denote the analogue
of Jacobi's $\psi-$function (\ref{Definition Jacobi final version}),
\begin{equation}
\psi_{\alpha}(x,z)=2^{\frac{\alpha}{2}-1}\,\Gamma\left(\frac{\alpha}{2}\right)\,\left(\sqrt{\pi x}\,z\right)^{1-\frac{\alpha}{2}}\,\sum_{n=1}^{\infty}r_{\alpha}(n)\,n^{\frac{1}{2}-\frac{\alpha}{4}}\,e^{-\pi n\,x}\,J_{\frac{\alpha}{2}-1}(\sqrt{\pi\,n\,x}\,z).\label{Definition Jacobi final version vanishing}
\end{equation}

Then, for any $z$ satisfying the condition:
\begin{equation}
z\in\mathscr{D}_{\alpha}:=\left\{ z\in\mathbb{C}\,:\,|\text{Re}(z)|<\sqrt{\frac{\pi\alpha}{2}},\,|\text{Im}(z)|<\sqrt{\frac{\pi\alpha}{2}}\right\} ,\label{condition of belonging alpha}
\end{equation}
and every $m\in\mathbb{N}_{0}$, one has that:
\begin{equation}
\frac{d^{m}}{d\omega^{m}}\left(1+\psi_{\alpha}\left(e^{2i\omega},z\right)\right)\rightarrow0,\,\,\,\,\,\text{as \,\,\,\ensuremath{\omega\rightarrow\frac{\pi}{4}^{-}}}.\label{exponentially fast decay}
\end{equation}
\end{lemma}

\begin{proof}
Note that $e^{2i\omega}\rightarrow i$ as $\omega\rightarrow\frac{\pi}{4}^{-}$
along a circular path where $\text{Re}(e^{2i\omega}),\,\,\text{Im}(e^{2i\omega})>0.$
Thus, we can write $e^{2i\omega}$ as $i+\delta$, where $\delta\rightarrow0$
along any path in the region $|\arg(\delta)|<\frac{\pi}{2}$. With
this change of variable and Fa\`a di Bruno's formula,
we can write
\begin{equation}
\frac{d^{m}}{d\omega^{m}}\psi_{\alpha}\left(e^{2i\omega},z\right)=(2i)^{m}\sum_{b_{1},...,b_{m}}\frac{m!\,}{b_{1}!\cdot...\cdot b_{m}!}\,\prod_{j=1}^{m}\left(\frac{i+\delta}{j!}\right)^{b_{j}}\frac{d^{b_{1}+...+b_{m}}}{d\delta^{b_{1}+...+b_{m}}}\,\psi_{\alpha}\left(i+\delta,z\right),\label{Faa di Bruno}
\end{equation}
where the sum is over all nonnegative integers $b_{1},...,b_{m}$
such that $b_{1}+2b_{2}+...+mb_{m}=m$.

\bigskip{}

Therefore, to prove (\ref{exponentially fast decay}), we shall evaluate
\begin{align*}
\psi_{\alpha}(i+\delta,\,z) & =2^{\frac{\alpha}{2}-1}\,\Gamma\left(\frac{\alpha}{2}\right)\,\left(\sqrt{\pi(i+\delta)}\,z\right)^{1-\frac{\alpha}{2}}\,\sum_{n=1}^{\infty}r_{\alpha}(n)\,n^{\frac{1}{2}-\frac{\alpha}{4}}\,e^{-\pi n\,(i+\delta)}\,J_{\frac{\alpha}{2}-1}(\sqrt{\pi\,n\,(i+\delta)}\,z)\\
 & =2^{\frac{\alpha}{2}-1}\,\Gamma\left(\frac{\alpha}{2}\right)\,\left(\sqrt{\pi(i+\delta)}\,z\right)^{1-\frac{\alpha}{2}}\,\sum_{n=1}^{\infty}(-1)^{n}\,r_{\alpha}(n)\,n^{\frac{1}{2}-\frac{\alpha}{4}}\,e^{-\pi n\,\delta}\,J_{\frac{\alpha}{2}-1}(\sqrt{\pi\,n\,(i+\delta)}\,z).
\end{align*}

Using the previous formula (\ref{final formula for 1f1 theorem-1})
with $x=\delta$ and $y=\sqrt{i+\delta}\,z$, we obtain
\begin{equation}
1+\psi_{\alpha}(i+\delta,z)=2^{\frac{\alpha}{2}-1}\,\Gamma\left(\frac{\alpha}{2}\right)\,\left(\sqrt{\pi(i+\delta)}\,z\right)^{1-\frac{\alpha}{2}}\,\frac{e^{-\frac{(i+\delta)z^{2}}{4\delta}}}{\delta}\,\sum_{n=0}^{\infty}\tilde{r}_{\alpha}(n)\,\left(n+\frac{\alpha}{4}\right)^{\frac{1}{2}-\frac{\alpha}{4}}\,e^{-\frac{\pi}{\delta}\,\left(n+\frac{\alpha}{4}\right)}I_{\frac{\alpha}{2}-1}\left(\frac{\sqrt{\pi(n+\frac{\alpha}{4})\,(i+\delta)}\,z}{\delta}\right).\label{after application summation ralpha}
\end{equation}

\bigskip{}

Note that (\ref{exponentially fast decay}) is established once we
prove that any derivative (with respect to $\delta$) of the right-hand
side of (\ref{after application summation ralpha}) tends to zero
as $\delta\rightarrow0$ along any path in the region $|\arg(\delta)|<\frac{\pi}{2}$.
We will prove first that (\ref{exponentially fast decay}) holds for
$m=0$ and the remaining cases will follow. In order to check this,
we use the fact that $\psi_{\alpha}(x,z)$ is analytic as a function
of both $x$ in $\text{Re}(x)>0$ and $z\in\mathbb{C}$ \footnote{The analyticity of $\psi_{\alpha}(x,z)$ is, of course, a direct consequence
of the more general result given in Corollary \ref{integral representation theta}.}, so it suffices to bound each term of the series as $\delta\rightarrow0^{+}$.

\bigskip{}

Like in Corollary \ref{integral representation theta}, we divide the proof in two cases: $\alpha>1$
or $0<\alpha\leq1$. For the first case, we use the integral representation
for the modified Bessel function {[}\cite{NIST}, p.252, eq. (10.32.2){]},
connected to the Poisson integral (\ref{Poisson Bessel})
\begin{equation}
\left(\frac{z}{2}\right)^{-\nu}I_{\nu}(z)=\frac{1}{\sqrt{\pi}\Gamma\left(\nu+\frac{1}{2}\right)}\,\intop_{-1}^{1}\left(1-t^{2}\right)^{\nu-\frac{1}{2}}\,e^{zt}dt,\,\,\,\,\,\text{Re}(\nu)>-\frac{1}{2},\,\,\,z\in\mathbb{C},\label{analogue Poisson integral}
\end{equation}
from which one can immediately obtain the bound (with real $\nu>-\frac{1}{2}$
and $z\in\mathbb{C}$)
\begin{equation}
\left|\left(\frac{z}{2}\right)^{-\nu}I_{\nu}(z)\right|\leq\frac{e^{|\text{Re}(z)|}}{\Gamma(\nu+1)}.\label{Bound Bessel}
\end{equation}

\bigskip{}

Using (\ref{Bound Bessel}), Lemma \ref{Lagarias like Lemma} and the fact that $\tilde{r}_{\alpha}(0)=2^{\alpha}$, we can bound $|1+\psi_{\alpha}(i+\delta,z)|$
simply as follows:
\begin{align*}
|1+\psi_{\alpha}(i+\delta,z)| & \leq\delta^{-\frac{\alpha}{2}}\,e^{-\frac{\text{Re}\left(\left(i+\delta\right)\,z^{2}\right)}{4\delta}}\,\sum_{n=0}^{\infty}|\tilde{r}_{\alpha}(n)|\,e^{-\frac{\pi}{\delta}\,\left(n+\frac{\alpha}{4}\right)}\,\exp\left(\frac{\sqrt{\pi(n+\frac{\alpha}{4})}}{\delta}\,\left|\text{Re}\left(\sqrt{i+\delta}\,z\right)\right|\right)\\
& < 2^{\alpha} \delta^{-\frac{\alpha}{2}}\,e^{-\frac{\text{Re}\left(\left(i+\delta\right)\,z^{2}\right)}{4\delta}}\,\,e^{-\frac{\pi \alpha}{4 \delta}}\,\exp\left(\frac{\sqrt{\pi \alpha}}{2 \delta}\,\left|\text{Re}\left(\sqrt{i+\delta}\,z\right)\right|\right)+\\
 & + 540\,\delta^{-\frac{\alpha}{2}}\,e^{-\frac{\text{Re}\left(\left(i+\delta\right)\,z^{2}\right)}{4\delta}}\,\sum_{n=1}^{\infty}(2n)^{\alpha/2}\,e^{-\frac{\pi}{\delta}\,\left(n+\frac{\alpha}{4}\right)}\,\exp\left(\frac{\sqrt{\pi(n+\frac{\alpha}{4})}}{\delta}\,\left|\text{Re}\left(\sqrt{i+\delta}\,z\right)\right|\right).
\end{align*}

Thus, to prove
(\ref{exponentially fast decay}) it is sufficient to show that, for
every fixed $n\in\mathbb{N}_{0}$ and $z\in\mathscr{D}_{\alpha}$ (see condition (\ref{condition of belonging alpha})),
\begin{equation}
\lim_{\delta\rightarrow0^{+}}\,\delta^{-\frac{\alpha}{2}}\,\exp\left(-\frac{\pi}{\delta}\,\left(n+\frac{\alpha}{4}\right)+\frac{\sqrt{\pi(n+\frac{\alpha}{4})}}{\delta}\,\left|\text{Re}\left(\sqrt{i+\delta}\,z\right)\right|-\frac{\text{Re}\left(\left(i+\delta\right)\,z^{2}\right)}{4\delta}\right)=0.\label{limit to prove at the end}
\end{equation}

Using the fact that $\text{Re}\left(\sqrt{i+\delta}\,z\right)\rightarrow\text{Re}(\sqrt{i}\,z)=\frac{1}{\sqrt{2}}\left(\text{Re}(z)-\text{Im}(z)\right)$
as $\delta\rightarrow0^{+}$ and $\text{Re}\left((i+\delta)\,z^{2}\right)\rightarrow-2\,\text{Re}(z)\,\text{Im}(z)$
also in this limit, we get
\begin{align}
\lim_{\delta\rightarrow0^{+}}\,\delta^{-\frac{\alpha}{2}}\,e^{-\frac{\text{Re}\left(\left(i+\delta\right)\,z^{2}\right)}{4\delta}}\,e^{-\frac{\pi}{\delta}\,\left(n+\frac{\alpha}{4}\right)}\,\exp\left(\frac{\sqrt{\pi(n+\frac{\alpha}{4})}}{\delta}\,\left|\text{Re}\left(\sqrt{i+\delta}\,z\right)\right|\right)\nonumber \\
=\lim_{\delta\rightarrow0^{+}}\,\delta^{-\frac{\alpha}{2}}\,\exp\left[-\frac{\pi}{\delta}\left(\left(n+\frac{\alpha}{4}\right)-\frac{\left|\text{Re}(z)-\text{Im}(z)\right|}{\sqrt{2\pi}}\sqrt{n+\frac{\alpha}{4}}-\frac{\text{Re}(z)\,\text{Im}(z)}{2\pi}\right)\right].\label{before considering polynomials}
\end{align}

\bigskip{}

Note that the term in the exponential is a quadratic polynomial of
variable $X=\sqrt{n+\frac{\alpha}{4}}$, with the polynomial being
explicitly given by
\begin{equation}
P(X)=X^{2}-\frac{|\text{Re}(z)-\text{Im}(z)|}{\sqrt{2\pi}}\,X-\frac{\text{Re}(z)\,\text{Im}(z)}{2\pi}.\label{Polynomial exponential}
\end{equation}

We know that, for a given $X_{0}\in\mathbb{R}_{>0}$, $P(X_{0})>0$
if $X_{0}>\sup\left\{ y\in\mathbb{R}\,:\,P(y)=0\right\} $. It is
an easy task to compute the zeros of the polynomial $P(X)$, so the
previous condition on $X_{0}$ reduces to:
\begin{equation}
P(X_{0})>0\,\,\,\text{if }\,\,\,\frac{|\text{Re}(z)-\text{Im}(z)|}{2\sqrt{2\pi}}+\frac{|\text{Re}(z)+\text{Im}(z)|}{2\sqrt{2\pi}}<X_{0}.\label{elementary condition polynomials}
\end{equation}

Since the polynomial in the exponential is defined for $X=\sqrt{n+\frac{\alpha}{4}}$,
$n\in\mathbb{N}_{0}$, a sufficient condition to assure that $P\left(\sqrt{n+\frac{\alpha}{4}}\right)>0$
for every $n\in\mathbb{N}_{0}$ is that $z$ must satisfy
\begin{equation}
|\text{Re}(z)-\text{Im}(z)|+|\text{Re}(z)+\text{Im}(z)|<\sqrt{2\pi\alpha},\label{sufficient condition positivity}
\end{equation}
which actually is the case if $z\in\mathscr{D_{\alpha}=}\left\{ z\in\mathbb{C}\,:\,|\text{Re}(z)|<\sqrt{\frac{\pi\alpha}{2}},\,|\text{Im}(z)|<\sqrt{\frac{\pi\alpha}{2}}\right\} $.
Therefore, under this hypothesis one sees that the limit in (\ref{before considering polynomials})
goes to zero as $\delta\rightarrow0^{+}$. Hence,
\[
1+\psi_{\alpha}\left(e^{2i\omega},z\right)\rightarrow0\,\,\,\,\,\,\text{as \,\, \ensuremath{\omega\rightarrow\frac{\pi}{4}^{-}} and\,\,\, \ensuremath{z}\ensuremath{\,\,\in\mathscr{D}_{\alpha}}.}
\]

Finally, since $\exp\left(-A/\delta\right),$ $A>0$, tends to zero
faster than any power $\delta^{N}$ as $\delta\rightarrow0^{+}$,
any derivative (with respect to $\delta$) of $\psi_{\alpha}(i+\delta,\,z)$
will go to zero in this limit, proving (\ref{exponentially fast decay})
for every $m\geq1$. This concludes the proof of the lemma for the
case where $\alpha>1$.

\bigskip{}

Assume now that $0<\alpha\leq1$: then the same argument must follow
but we cannot invoke the integral representation (\ref{analogue Poisson integral})
for $I_{\frac{\alpha}{2}-1}(z)$. Just like in (\ref{limit to prove at the end})
we need to evaluate the limit

\begin{equation}
\lim_{\delta\rightarrow0^{+}}\left|2^{\frac{\alpha}{2}-1}\,\Gamma\left(\frac{\alpha}{2}\right)\,\left(\sqrt{\pi(i+\delta)}\,z\right)^{1-\frac{\alpha}{2}}\frac{e^{-\frac{(i+\delta)z^{2}}{4\delta}}}{\delta}\,e^{-\frac{\pi}{\delta}\,\left(n+\frac{\alpha}{4}\right)}I_{\frac{\alpha}{2}-1}\left(\frac{\sqrt{\pi(n+\frac{\alpha}{4})\,(i+\delta)}\,z}{\delta}\right)\right|\label{limit to evaluate}
\end{equation}
for every fixed $n\in\mathbb{N}_{0}$. This can be done by appealing
to the well-known Hankel expansion of the modified Bessel function
(see {[}\cite{NIST}, p. 255, eq. (10.40.5){]} and (\ref{asymptotic I}) above), 
\begin{equation}
I_{\nu}(z)\sim\frac{e^{z}}{\sqrt{2\pi z}}\,\frac{\cos(\pi\nu)}{\pi}\,\sum_{n=0}^{\infty}\frac{\Gamma\left(\frac{1}{2}+\nu+n\right)\,\Gamma\left(\frac{1}{2}-\nu+n\right)}{(2z)^{n}n!}+\frac{e^{-z+(\nu+\frac{1}{2})\pi i}}{\sqrt{2\pi z}}\,\frac{\cos(\pi\nu)}{\pi}\,\sum_{n=0}^{\infty}\frac{(-1)^{n}\Gamma\left(\frac{1}{2}+\nu+n\right)\,\Gamma\left(\frac{1}{2}-\nu+n\right)}{(2z)^{n}n!},\label{asymptotic modified-1}
\end{equation}
which is valid for $-\frac{\pi}{2}<\arg(z)<\frac{3\pi}{2}$, $|z|\rightarrow\infty$.
Alternatively, in the range $-\frac{3}{2}\pi<\arg(z)<\frac{\pi}{2}$, $|z|\rightarrow\infty$, 
we have the expansion 
\begin{equation}
I_{\nu}(z)\sim\frac{e^{z}}{\sqrt{2\pi z}}\,\frac{\cos(\pi\nu)}{\pi}\,\sum_{n=0}^{\infty}\,\frac{\Gamma\left(\frac{1}{2}+\nu+n\right)\Gamma\left(\frac{1}{2}-\nu+n\right)}{(2z)^{n}\,n!}+\frac{e^{-z-(\nu+\frac{1}{2})\pi i}}{\sqrt{2\pi z}}\,\frac{\cos(\pi\nu)}{\pi}\,\sum_{n=0}^{\infty}\frac{(-1)^{n}\Gamma\left(\frac{1}{2}+\nu+n\right)\,\Gamma\left(\frac{1}{2}-\nu+n\right)}{(2z)^{n}n!}.\label{asymptotics modified second range}
\end{equation}

In any case, from (\ref{asymptotic modified-1}) and (\ref{asymptotics modified second range})
and real $\nu$, one can obtain the simple bound
\begin{equation}
|I_{\nu}(z)|\leq2\,\left(\frac{e^{\text{Re}(z)}}{\sqrt{2\pi |z|}}+\frac{e^{-\text{Re}(z)}}{\sqrt{2\pi|z|}}\right)\leq\sqrt{\frac{8}{\pi|z|}}\,\exp\left(|\text{Re}(z)|\right)\label{Inequality Modified Bessel any z}
\end{equation}
which becomes valid for every $|z|>M$, where $M$ is sufficiently
large.

\bigskip{}

We now apply (\ref{Inequality Modified Bessel any z}) to (\ref{limit to evaluate})
and we obtain, for $\delta>0$ sufficiently small, some fixed $n\in\mathbb{N}_{0}$
and every $z\in\mathscr{D}_{\alpha}$,
\begin{align*}
\left|2^{\frac{\alpha}{2}-1}\,\Gamma\left(\frac{\alpha}{2}\right)\,\left(\sqrt{\pi(i+\delta)}\,z\right)^{1-\frac{\alpha}{2}}\frac{e^{-\frac{(i+\delta)z^{2}}{4\delta}}}{\delta}\,e^{-\frac{\pi}{\delta}\,\left(n+\frac{\alpha}{4}\right)}I_{\frac{\alpha}{2}-1}\left(\frac{\sqrt{\pi(n+\frac{\alpha}{4})\,(i+\delta)}\,z}{\delta}\right)\right|\\
<\frac{\pi^{-\frac{\alpha+1}{4}}\,2^{\frac{\alpha}{2}+1}\,\Gamma\left(\frac{\alpha}{2}\right)\,|z|^{\frac{1-\alpha}{2}}}{\left(n+\frac{\alpha}{4}\right)^{1/4}\,\sqrt{\delta}}\,\exp\left(-\frac{\pi}{\delta}\,\left(n+\frac{\alpha}{4}\right)+\frac{\sqrt{\pi(n+\frac{\alpha}{4})}}{\delta}\,\left|\text{Re}\left(\sqrt{i+\delta}\,z\right)\right|-\frac{\text{Re}((i+\delta)z^{2})}{4\delta}\right).
\end{align*}

Note that there is no problem if $z$ is arbitrarily close to the
origin because we are supposing that $0<\alpha\leq1$. Thus, to prove
that (\ref{limit to evaluate}) tends to zero, we just follow the
same logic as before, because the term in the exponential is exactly
the same as in the previous case (\ref{before considering polynomials})
and it involves the quadratic polynomial (\ref{Polynomial exponential}).
Thus, the condition $z\in\mathscr{D}_{\alpha}$ needs to be kept in
this range of $\alpha$ and this completes the proof of (\ref{exponentially fast decay}). 
\end{proof}

\begin{center}\subsection{Some additional lemmas} \label{additional lemmas 1.1} \end{center} 

With Lemma \ref{Vanishing Theta lemma} completely established, we are ready to prove Theorem \ref{zeta alpha hypergeometric zeros}. Before doing this we provide two additional lemmas, one of which being an easy consequence of Lemma \ref{Vanishing Theta lemma}. 
\\

\begin{lemma}\label{vanishing with analytic}
Let $h:\,\mathbb{C}\longmapsto\mathbb{C}$ be an analytic function.
Then, for every $z\in\mathscr{D}_{\alpha}$ (\ref{condition of belonging alpha}) and arbitrary $m\in\mathbb{N}_{0}$,
we have
\begin{equation}
\lim_{\omega\rightarrow\frac{\pi}{4}^{-}}\,\frac{d^{m}}{d\omega^{m}}\left\{ h(\omega)\,\left(e^{-z^{2}/8}+e^{z^{2}/8}\,\psi_{\alpha}\left(e^{2i\omega},z\right)\right)\right\} =-2\,h^{(m)}\left(\frac{\pi}{4}\right)\,\sinh\left(\frac{z^{2}}{8}\right).\label{analogue of Landau lemma}
\end{equation}  
\end{lemma}

\begin{proof}
Applying the Leibniz formula for the derivative, we obtain
\begin{equation}
\frac{d^{m}}{d\omega^{m}}\left\{ h(\omega)\,\left(e^{-z^{2}/8}+e^{z^{2}/8}\,\psi_{\alpha}\left(e^{2i\omega},z\right)\right)\right\} =\sum_{k=0}^{m}\left(\begin{array}{c}
m\\
k
\end{array}\right)\,\frac{d^{m-k}}{d\omega^{m-k}}h(\omega)\,\frac{d^{k}}{d\omega^{k}}\left(e^{-z^{2}/8}+e^{z^{2}/8}\,\psi_{\alpha}\left(e^{2i\omega},z\right)\right).\label{application Leibniz}
\end{equation}

Thus, letting $\omega\rightarrow\frac{\pi}{4}^{-}$, we see from lemma \ref{Vanishing Theta lemma} that any derivative of order $k\geq1$ of $\psi_{\alpha}\left(e^{2i\omega},z\right)$
tends to zero as $\omega\rightarrow\frac{\pi}{4}^{-}$. Thus, the
only surviving term is when $k=0$, which immediately gives (\ref{analogue of Landau lemma}).  
\end{proof}

\bigskip{}

Since our proof will follow very closely the argument in \cite{DKMZ}, we will also require Kronecker's lemma \cite{Hardy_Wright}. 

\begin{lemma}\label{Kronecker Lemma}
If $x\notin\mathbb{Q}$, then the sequence of fractional parts $(\{n\,x\})_{n\in\mathbb{N}}$
is dense in the interval $(0,1)$.
\end{lemma}

\begin{center}\subsection{Proof of Theorem \ref{zeta alpha hypergeometric zeros}} \label{proof of Theorem 1.1} \end{center} 

The proof that follows uses the previous lemmas and integral representations, but the main ideas are due to Hardy and to a variation of his argument given by Dixit et al. in the papers \cite{DRRZ, DKMZ}. 
\\

We start by using the integral representation given in Example \ref{example 2.1}
(\ref{equation in the setting of zeta alpha}), replacing there $x$
by $e^{2i\omega}$, $-\frac{\pi}{4}<\omega<\frac{\pi}{4}$. We obtain
the formula
\[
\frac{1}{2\pi}\,\intop_{-\infty}^{\infty}\eta_{\alpha}\left(\frac{\alpha}{4}+it\right)\,_{1}F_{1}\left(\frac{\alpha}{4}+it;\,\frac{\alpha}{2};\,-\frac{z^{2}}{4}\right)\,e^{2\omega t}\,dt=e^{\frac{i\omega\alpha}{2}}\,\psi_{\alpha}\left(e^{2i\omega},z\right)-e^{-\frac{i\omega\,\alpha}{2}}e^{-z^{2}/4}.
\]

We can use now Kummer's formula (\ref{Kummer confluent transformation})
on the integrand to obtain 
\begin{equation}
\frac{e^{-z^{2}/8}}{2\pi}\,\intop_{-\infty}^{\infty}\eta_{\alpha}\left(\frac{\alpha}{4}+it\right)\,_{1}F_{1}\left(\frac{\alpha}{4}-it;\,\frac{\alpha}{2};\,\frac{z^{2}}{4}\right)\,e^{2\omega t}\,dt=e^{\frac{i\omega\alpha}{2}}\,e^{z^{2}/8}\,\psi_{\alpha}\left(e^{2i\omega},z\right)-e^{-\frac{i\omega\,\alpha}{2}}e^{-z^{2}/8}.\label{after kummer in the whole proof}
\end{equation}

Now if we add and subtract the term $e^{-z^{2}/8}\,e^{i\omega\alpha/2}$
on the right-hand side of the previous equality, we rewrite (\ref{after kummer in the whole proof}) in the form
\begin{equation}
\frac{e^{-z^{2}/8}}{2\pi}\,\intop_{-\infty}^{\infty}\eta_{\alpha}\left(\frac{\alpha}{4}+it\right)\,_{1}F_{1}\left(\frac{\alpha}{4}-it;\,\frac{\alpha}{2};\,\frac{z^{2}}{4}\right)\,e^{2\omega t}\,dt
=-2\,e^{-z^{2}/8}\,\cos\left(\frac{\omega\alpha}{2}\right)+e^{\frac{i\omega\alpha}{2}}\left(e^{-z^{2}/8}+e^{z^{2}/8}\,\psi_{\alpha}\left(e^{2i\omega},z\right)\right).\label{symmetric identity}
\end{equation}

\bigskip{}

After having an identity like (\ref{symmetric identity}), the structure
of our proof follows \cite{DKMZ} almost line by line. Since we want to study
the zeros of an infinite combination of $\eta_{\alpha}(s+i\lambda_{j})\,_{1}F_{1}(s+i\lambda_{j};\,\frac{\alpha}{2};\,\frac{z^{2}}{4})$,
we change the variable to get
\begin{align}
\frac{e^{-z^{2}/8}}{2\pi}\,\intop_{-\infty}^{\infty}\eta_{\alpha}\left(\frac{\alpha}{4}+i(t+\lambda_{j})\right)\,_{1}F_{1}\left(\frac{\alpha}{4}-i(t+\lambda_{j});\,\frac{\alpha}{2};\,\frac{z^{2}}{4}\right)\,e^{2\omega t}\,dt=\nonumber \\
=\frac{e^{-z^{2}/8}e^{-2\omega\lambda_{j}}}{2\pi}\,\intop_{-\infty}^{\infty}\eta_{\alpha}\left(\frac{\alpha}{4}+it\right)\,_{1}F_{1}\left(\frac{\alpha}{4}-it;\,\frac{\alpha}{2};\,\frac{z^{2}}{4}\right)\,e^{2\omega t}\,dt=\nonumber \\
=e^{-2\omega\lambda_{j}}\left\{ -2\,e^{-z^{2}/8}\,\cos\left(\frac{\omega\alpha}{2}\right)+e^{\frac{i\omega\alpha}{2}}\left(e^{-z^{2}/8}+e^{z^{2}/8}\,\psi_{\alpha}\left(e^{2i\omega},z\right)\right)\right\}=\nonumber \\
=-\,e^{-z^{2}/8}\,e^{-2\omega\lambda_{j}+i\,\frac{\omega\alpha}{2}}-\,e^{-z^{2}/8}\,e^{-2\omega\lambda_{j}-i\,\frac{\omega\alpha}{2}}+e^{\frac{i\omega\alpha}{2}-2\omega\lambda_{j}}\left(e^{-z^{2}/8}+e^{z^{2}/8}\,\psi_{\alpha}\left(e^{2i\omega},z\right)\right).\label{pre differentiation}
\end{align}

\bigskip{}

Proceeding now as in Hardy's own proof \cite{hardy_note}, let us take the
$p-$th derivative of both sides of (\ref{pre differentiation}) with
respect to $\omega$: since the integral on the left of (\ref{pre differentiation})
converges absolutely and uniformly with respect of $\omega\in(-\pi/4,\,\pi/4)$,
one can apply the Leibniz rule and deduce the equality:
\begin{align}
\frac{2^{p}\,e^{-z^{2}/8}}{2\pi}\,\intop_{-\infty}^{\infty}t^{p}\eta_{\alpha}\left(\frac{\alpha}{4}+i\,(t+\lambda_{j})\right)\,_{1}F_{1}\left(\frac{\alpha}{4}-i\,(t+\lambda_{j});\,\frac{\alpha}{2};\,\frac{z^{2}}{4}\right)\,e^{2\omega t}\,dt=\nonumber \\
=-\,e^{-z^{2}/8}\left(-2\lambda_{j}+i\frac{\alpha}{2}\right)^{p}\,e^{-2\omega\lambda_{j}+i\,\frac{\omega\alpha}{2}}-\,e^{-z^{2}/8}\left(-2\lambda_{j}-i\frac{\alpha}{2}\right)^{p}\,e^{-2\omega\lambda_{j}-i\,\frac{\omega\alpha}{2}}+\nonumber \\
+\frac{d^{p}}{d\omega^{p}}\left\{ e^{\frac{i\omega\alpha}{2}-2\omega\lambda_{j}}\left(e^{-z^{2}/8}+e^{z^{2}/8}\,\psi_{\alpha}\left(e^{2i\omega},z\right)\right)\right\} .\label{post diferentiation}
\end{align}

We now consider the change of variables,
\begin{equation}
\frac{i\alpha}{2}-2\lambda_{j}:=r_{j}\,e^{i\theta_{j}},\,\,\,0<\theta_{j}<\frac{\pi}{2},\label{polar coordinates sequence}
\end{equation}
and, using these new coordinates, write (\ref{post diferentiation})
in a more suitable form
\begin{align}
\intop_{-\infty}^{\infty}t^{p}\eta_{\alpha}\left(\frac{\alpha}{4}+i\,(t+\lambda_{j})\right)\,_{1}F_{1}\left(\frac{\alpha}{4}-i\,(t+\lambda_{j});\,\frac{\alpha}{2};\,\frac{z^{2}}{4}\right)\,e^{2\omega t}\,dt=\nonumber \\
=-4\pi\,\left(\frac{r_{j}}{2}\right)^{p}\,e^{-2\omega\lambda_{j}}\,\cos\left(p\,\theta_{j}+\frac{\omega\alpha}{2}\right)+\frac{2\pi}{2^{p}}\,e^{z^{2}/8}\,\frac{d^{p}}{d\omega^{p}}\left\{ e^{\frac{i\omega\alpha}{2}-2\omega\lambda_{j}}\left(e^{-z^{2}/8}+e^{z^{2}/8}\,\psi_{\alpha}\left(e^{2i\omega},z\right)\right)\right\} .\label{real at last}
\end{align}

To apply Hardy's method, we need to have a real integrand on the left-hand
side of (\ref{real at last}) in order to count the changes of its
sign. This is done if we sum another integral
\[
\intop_{-\infty}^{\infty}t^{p}\eta_{\alpha}\left(\frac{\alpha}{4}+i\,(t+\lambda_{j})\right)\,_{1}F_{1}\left(\frac{\alpha}{4}+i\,(t+\lambda_{j});\,\frac{\alpha}{2};\,\frac{\overline{z}^{2}}{4}\right)\,e^{2\omega t}\,dt,
\]
whose integrand is the complex conjugate of the integrand in (\ref{real at last}).

\bigskip{}

Taking the real part in both sides of (\ref{real at last}), multiplying
by $c_{j}\in\ell^{1}$ and summing over $j$, we can formally derive:
\begin{align}
\intop_{-\infty}^{\infty}t^{p}\,F_{z,\alpha}\left(\frac{\alpha}{4}+it\right)\,e^{2\omega t}\,dt & =-\frac{8\pi}{2^{p}}\,\sum_{j=1}^{\infty}c_{j}\,r_{j}^{p}e^{-2\omega\lambda_{j}}\,\left[\cos\left(p\,\theta_{j}+\frac{\omega\alpha}{2}\right)\right]+\nonumber \\
+\frac{4\pi}{2^{p}}\,\text{Re} & \left(e^{z^{2}/8}\,\frac{d^{p}}{d\omega^{p}}\left\{ \sum_{j=1}^{\infty}c_{j}e^{\frac{i\omega\alpha}{2}-2\omega\lambda_{j}}\left(e^{-z^{2}/8}+e^{z^{2}/8}\,\psi_{\alpha}\left(e^{2i\omega},z\right)\right)\right\} \right),\label{Taking real parts in that way}
\end{align}
where $F_{z,\alpha}(s)$ is the function given by (\ref{function defining coooombinations}), 
\[
F_{z,\alpha}(s)=\sum_{j=1}^{\infty}c_{j}\,\eta_{\alpha}\left(s+i\lambda_{j}\right)\,\left\{ _{1}F_{1}\left(\frac{\alpha}{2}-s-i\lambda_{j};\,\frac{\alpha}{2};\,\frac{z^{2}}{4}\right)+\,_{1}F_{1}\left(\frac{\alpha}{2}-\overline{s}+i\lambda_{j};\,\frac{\alpha}{2};\,\frac{\overline{z}^{2}}{4}\right)\right\} ,
\]
whose zeros we will now study.

\bigskip{}

But before this, we justify (\ref{Taking real parts in that way})
by following a similar reasoning to that given in {[}\cite{DKMZ}, p.
316{]}. By Stirling's formula and convex estimates for $\zeta_{\alpha}(s)$
(obtained from the general Phragm\'en-Lindel\"of principle (\ref{Lindelof Phragmen for Hecke})),
we know that
\[
\left|\eta_{\alpha}\left(\frac{\alpha}{4}+it\right)\right|\ll_{\alpha}\,|t|^{A(\alpha)}\,e^{-\frac{\pi}{2}|t|},\,\,\,\,\,|t|\rightarrow\infty.
\]

This, together with the asymptotic estimate for Kummer's
function (\ref{asymptotic whittaker sense}),
\[
\left|\,_{1}F_{1}\left(\frac{\alpha}{4}-it;\,\frac{\alpha}{2};\,\frac{z^{2}}{4}\right)\right|\ll_{\alpha,z}\,|t|^{\frac{1}{4}-\frac{\alpha}{4}}\exp\left(|z|\sqrt{|t|}\right),\,\,\,\,\,|t|\rightarrow\infty,
\]
yields
\begin{equation}
\sum_{j=1}^{\infty}\left|c_{j}\,\eta_{\alpha}\left(\frac{\alpha}{4}+i\,(t+\lambda_{j})\right)\text{Re}\left(\,_{1}F_{1}\left(\frac{\alpha}{4}-i\,(t+\lambda_{j});\,\frac{\alpha}{2};\,\frac{z^{2}}{4}\right)\right)\right|\ll_{\alpha,z}C_{\lambda}\,|t|^{B(\alpha)}\,e^{-\frac{\pi}{2}|t|+|z|\sqrt{|t|}}\sum_{j=1}^{\infty}|c_{j}|<\infty\label{estimate simple}
\end{equation}
where we have used the fact that $\left(\lambda_{j}\right)_{j\in\mathbb{N}}$
is a bounded sequence and $\left(c_{j}\right)_{j\in\mathbb{N}}\in\ell^{1}$.
The term $C_{\lambda}$ only stands for a positive constant which
depends on the bounds of the sequence $\left(\lambda_{j}\right)_{j\in\mathbb{N}}$.
This observation now allows the interchange of the orders of integration
and summation in the process leading to the left-hand side of (\ref{Taking real parts in that way}).

\bigskip{}

The very same reasoning can be used to show the procedure
\[
\sum_{j=1}^{\infty}c_{j}\,\frac{d^{p}}{d\omega^{p}}\left\{ e^{\frac{i\omega\alpha}{2}-2\omega\lambda_{j}}\left(e^{-z^{2}/8}+e^{z^{2}/8}\,\psi_{\alpha}\left(e^{2i\omega},z\right)\right)\right\} =\frac{d^{p}}{d\omega^{p}}\left\{ \sum_{j=1}^{\infty}c_{j}e^{\frac{i\omega\alpha}{2}-2\omega\lambda_{j}}\left(e^{-z^{2}/8}+e^{z^{2}/8}\,\psi_{\alpha}\left(e^{2i\omega},z\right)\right)\right\} ,
\]
used on the right-hand side of (\ref{Taking real parts in that way}).

\bigskip{}

The strategy now is to let $\omega\rightarrow\frac{\pi}{4}^{-}$ on
both sides of (\ref{Taking real parts in that way}). The previous
lemmas are indispensable to study the behavior of the right-hand side under this limit. Since $(c_{j})_{j\in \mathbb{N}}\in\ell^{1}$, it is clear that the
function
\[
h_{\alpha}(\omega):=\sum_{j=1}^{\infty}c_{j}e^{\frac{i\omega\alpha}{2}-2\omega\lambda_{j}}
\]
is analytic for every $\omega\in\mathbb{C}$. Hence, according to
Lemma \ref{vanishing with analytic}, 
\begin{align}
\lim_{\omega\rightarrow\frac{\pi}{4}^{-}}\,\frac{d^{p}}{d\omega^{p}}\left\{ \sum_{j=1}^{\infty}c_{j}e^{\frac{i\omega\alpha}{2}-2\omega\lambda_{j}}\left(e^{-z^{2}/8}+e^{z^{2}/8}\,\psi_{\alpha}\left(e^{2i\omega},z\right)\right)\right\}  & =\lim_{\omega\rightarrow\frac{\pi}{4}^{-}}\,\frac{d^{p}}{d\omega^{p}}\left\{ h_{\alpha}(\omega)\,\left(e^{-z^{2}/8}+e^{z^{2}/8}\,\psi_{\alpha}\left(e^{2i\omega},z\right)\right)\right\} \nonumber \\
=-2\,\sinh\left(\frac{z^{2}}{8}\right)\,\sum_{j=1}^{\infty}c_{j}\left(\frac{i\alpha}{2}-2\lambda_{j}\right)^{p}e^{\frac{i\pi\alpha}{8}-\frac{\pi}{2}\lambda_{j}}\, & =-2\,\sinh\left(\frac{z^{2}}{8}\right)\,\sum_{j=1}^{\infty}c_{j}\,r_{j}^{p}\,e^{-\frac{\pi}{2}\lambda_{j}}\,e^{i\left(\frac{\pi\alpha}{8}+p\,\theta_{j}\right)}, \label{in the middle of nowhere}
\end{align}
where in the last step we have just considered the coordinates (\ref{polar coordinates sequence}).

\bigskip{}

\bigskip{}

Returning to the identity (\ref{Taking real parts in that way}),
we get from (\ref{in the middle of nowhere}), 
\begin{align}
\lim_{\omega\rightarrow\frac{\pi}{4}^{-}}\,\intop_{-\infty}^{\infty}t^{p}F_{z,\alpha}\left(\frac{\alpha}{4}+it\right)&\,e^{2\omega t}\,dt =-\frac{8\pi}{2^{p}}\,\sum_{j=1}^{\infty}c_{j}r_{j}^{p}\,e^{-\frac{\pi}{2}\lambda_{j}}\,\left[\cos\left(p\,\theta_{j}+\frac{\pi}{8}\alpha\right)+\text{Re}\left(e^{z^{2}/8}\,\sinh\left(\frac{z^{2}}{8}\right)\,e^{i\left(\frac{\pi\alpha}{8}+p\,\theta_{j}\right)}\right)\right]\nonumber \\
=-\frac{8\pi}{2^{p}} \,\sum_{j=1}^{\infty}c_{j}r_{j}^{p}e^{-\frac{\pi}{2}\lambda_{j}} &\left[\cos\left(p\,\theta_{j}+\frac{\pi}{8}\alpha\right)\left(1+\text{Re}\left(e^{z^{2}/8}\,\sinh\left(\frac{z^{2}}{8}\right)\right)\right)-\sin\left(p\theta_{j}+\frac{\pi}{8}\alpha\right)\,\text{Im}\left(e^{z^{2}/8}\,\sinh\left(\frac{z^{2}}{8}\right)\right)\right].\label{first template}
\end{align}

Now, if we once more follow     \cite{DKMZ} and set
\begin{equation}
u_{z}:=1+\text{Re}\left(e^{z^{2}/8}\,\sinh\left(\frac{z^{2}}{8}\right)\right),\,\,\,\,\,v_{z}:=\text{Im}\left(e^{z^{2}/8}\,\sinh\left(\frac{z^{2}}{8}\right)\right),\label{definition new coordinates uz and vz}
\end{equation}
we are able to write the right-hand side of (\ref{first template})
in the approriate form
\begin{equation}
-\frac{8\pi}{2^{p}}\,w_{z}\,\sum_{j=1}^{\infty}c_{j}r_{j}^{p}\,e^{-\frac{\pi}{2}\lambda_{j}}\,\cos\left(p\theta_{j}+\frac{\pi\alpha}{8}+\beta_{z}\right),\,\,\text{for}\,\,\,\,w_{z}=\sqrt{u_{z}^{2}+v_{z}^{2}},\,\,\,\beta_{z}:=\arctan\left(\frac{v_{z}}{u_{z}}\right).\label{right template}
\end{equation}

\bigskip{}

Since all the parameters in the previous expression are fixed either
by the Dirichlet series ($\alpha$) or by the sequence taken in the
statement $(\theta_{j})_{j\in\mathbb{N}}$, the only free parameter
is the integer $p$, which is the number of times we have
differentiated both sides of (\ref{pre differentiation}). Therefore,
the previous expression can be thought of as a sequence of real numbers,
say $\left(s_{p}\right)_{p\in\mathbb{N}}$.

\bigskip{}

Under only some minor modifications of the argument given in {[}\cite{DKMZ},
pp. 317-321{]}, we now show that the sequence $(s_{p})_{p\in\mathbb{N}}$
defined by (\ref{right template}) changes its sign for infinitely many
values of $p$.

\bigskip{}

Since $\left(\lambda_{j}\right)_{j\in\mathbb{N}}$ attains its bounds
and is made of distinct elements, we know that there is some positive
integer $M$ such that
\begin{equation}
|\lambda_{M}|=\max_{j\in\mathbb{N}}\,\left\{ |\lambda_{j}|\right\} ,\,\,\,\,\,\,\lambda_{j}\neq\lambda_{M}\,\,\,\,\text{for }j\neq M\implies r_{j}<r_{M}\,\,\,\,\text{for }j\neq M,\label{maximality of lambda}
\end{equation}
where we have used the change of coordinates (\ref{polar coordinates sequence}).
Observe now that we can write the series in (\ref{right template})
as
\begin{equation}
-\frac{8\pi}{2^{p}}\,w_{z}\,c_{M}r_{M}^{p}\,e^{-\frac{\pi}{2}\lambda_{M}}\,\cos\left(p\theta_{M}+\frac{\pi\alpha}{8}+\beta_{z}\right)\,\left\{ 1+E(X,z)+H(X,z)\right\} ,\label{things almost at the end}
\end{equation}
where
\[
E(X,z):=\sum_{j\neq M,\,j\leq X}\frac{c_{j}}{c_{M}}\,e^{-\frac{\pi}{2}\left(\lambda_{j}-\lambda_{M}\right)}\left(\frac{r_{j}}{r_{M}}\right)^{p}\,\frac{\cos\left(p\theta_{j}+\frac{\pi\alpha}{8}+\beta_{z}\right)}{\cos\left(p\theta_{M}+\frac{\pi\alpha}{8}+\beta_{z}\right)}
\]
and
\[
H(X,z):=\sum_{j\neq M,\,j>X}\frac{c_{j}}{c_{M}}\,e^{-\frac{\pi}{2}\left(\lambda_{j}-\lambda_{M}\right)}\left(\frac{r_{j}}{r_{M}}\right)^{p}\,\frac{\cos\left(p\theta_{j}+\frac{\pi\alpha}{8}+\beta_{z}\right)}{\cos\left(p\theta_{M}+\frac{\pi\alpha}{8}+\beta_{z}\right)},
\]
where $X$ is some large parameter that will be chosen later.

\bigskip{}

We will now see that, as a function of $p$, the sign of $\cos\left(p\theta_{M}+\frac{\pi\alpha}{8}+\beta_{z}\right)$
will prevail in (\ref{things almost at the end}), making the sign
of the integral on the left of (\ref{first template}) to change
infinitely often.

\bigskip{}

We divide the proof of this fact in two items: in the first item,
we show it for $\frac{\theta_{M}}{2\pi}\notin\mathbb{Q}$ and in the
second we argue that the case $\frac{\theta_{M}}{2\pi}\in\mathbb{Q}$
ultimately reduces to the former.
\bigskip{}

\begin{enumerate}
\item Assume that $\frac{\theta_{M}}{2\pi}\notin\mathbb{Q}$: by Kronecker's
lemma (see Lemma \ref{Kronecker Lemma}), $\left(\left\{ n\,\frac{\theta_{M}}{2\pi}\right\} \right)_{n\in\mathbb{N}}$
is a dense subset of $(0,1)$. Hence, 
\[
\left(\cos\left(n\,\theta_{M}+\frac{\pi\alpha}{8}+\beta_{z}\right)\right)_{n\in\mathbb{N}}=\left(\cos\left(2\pi\,\left\{ \frac{n\,\theta_{M}}{2\pi}\right\} +\frac{\pi\alpha}{8}+\beta_{z}\right)\right)_{n\in\mathbb{N}}
\]
is a dense subset of $\left[-1,1\right]$ due to continuity and surjectivity
of the function $f(X)=\cos(X+a)$. Therefore, we can find two sequences of integers
$\left(q_{n}\right)_{n\in\mathbb{N}}$ and $\left(r_{n}\right)_{n\in\mathbb{N}}$
such that 
\begin{equation}
\cos\left(2\pi\,\left\{ \frac{q_{n}\,\theta_{M}}{2\pi}\right\} +\frac{\pi\alpha}{8}+\beta_{z}\right)\rightarrow\frac{1}{2},\,\,\,\,\,\cos\left(2\pi\,\left\{ \frac{r_{n}\,\theta_{M}}{2\pi}\right\} +\frac{\pi\alpha}{8}+\beta_{z}\right)\rightarrow-\frac{1}{2},\,\,\,\text{as }n\rightarrow\infty.\label{construction sequence}
\end{equation}
Thus, for a sufficiently large $N_{0}$ and $p\in\left(q_{n}\right)_{n\geq N_{0}}\cup\left(r_{n}\right)_{n\geq N_{0}}$,
we have the inequality 
\[
\left|\cos\left(p\theta_{M}+\frac{\pi\alpha}{8}+\beta_{z}\right)\right|\geq\frac{1}{3}.
\]
Since $\left(\lambda_{k}\right)_{k\in\mathbb{N}}$ is bounded and
attains its bounds, under this choice of $N_{0}$ we have the bound
for $H(X,z)$, 
\[
|H(X,z)|\leq\frac{3}{|c_{M}|}\,e^{\frac{\pi}{2}\,\max_{k,\ell}|\lambda_{k}-\lambda_{\ell}|}\sum_{j\neq M,\,j>X}|c_{j}|<\frac{1}{3},
\]
once we take $X\geq X_{0}$ large enough. With the same choice of
$X$, we bound $E(X,z)$ as follows
\[
\left|E(X,z)\right|\leq\frac{3\mu_{X}^{p}}{|c_{M}|}\,e^{\frac{\pi}{2}\max_{k,\ell}|\lambda_{k}-\lambda_{\ell}|}\sum_{j\neq M,\,j\leq X}|c_{j}|,
\]
where 
\[
\mu_{X}=\max_{j\leq X}\left\{ \frac{r_{j}}{r_{M}}\right\} .
\]
By (\ref{maximality of lambda}), we know that $\mu_{X}<1$ and so,
for $N_{0}$ sufficiently large and $p\geq\max\left\{ q_{N_{0}},r_{N_{0}}\right\} $,
\[
|E(X,z)|<\frac{1}{3}.
\]
Therefore, for our choice of $N_{0}$ and $X$, $1+E(X,z)+H(X,z)>\frac{1}{3}$,
and so, for $p\in\left(q_{n}\right)_{n\geq N_{0}}\cup\left(r_{n}\right)_{n\geq N_{0}}$
and from (\ref{first template}), (\ref{things almost at the end}), the limit
\[
\lim_{\omega\rightarrow\frac{\pi}{4}^{-}}\,\intop_{-\infty}^{\infty}t^{p}F_{z,\alpha}\left(\frac{\alpha}{4}+it\right)\,e^{2\omega t}\,dt
\]
must have the same sign (as a sequence depending on $p\in\left(q_{n}\right)_{n\in\mathbb{N}}\cup\left(r_{n}\right)_{n\in\mathbb{N}}$)
as the sequence
\[
\left(s_{p}^{\prime}\right)_{p\in(q_{n})\cup(r_{n})_{n\in\mathbb{N}}}:=-\frac{8\pi}{2^{p}}\,w_{z}\,c_{M}r_{M}^{p}\,e^{-\frac{\pi}{2}\lambda_{M}}\,\cos\left(p\theta_{M}+\frac{\pi\alpha}{8}+\beta_{z}\right),
\]
 whose sign changes infinitely often by the construction (\ref{construction sequence}).
\item We now deal with the case where $\frac{\theta_{M}}{2\pi}\in\mathbb{Q}$ making the same remark as in [\cite{DKMZ}, pp. 320-321].
Note that, for a fixed small $\delta>0$, Theorem \ref{zeta alpha hypergeometric zeros} is equivalent
to the statement that $F_{z,\alpha}(s+i\delta)$ has infinitely many
zeros at the line $\text{Re}(s)=\frac{\alpha}{4}$. However, throughout
the course of the previous argument we just need to replace the elements
of the sequence $\left(\lambda_{j}\right)_{j\in\mathbb{N}}$ by $\lambda_{j}^{\star}:=\lambda_{j}+\delta$.
But now in this new analysis we need to rewrite (\ref{polar coordinates sequence})
as $\frac{i\alpha}{2}-2\lambda_{j}^{\star}=r_{j}^{\star}\,e^{i\theta_{j}^{\star}}$.
Since we have the freedom of choosing $\delta$, we may take it in
such a way that $\frac{\theta_{M}^{\star}}{2\pi}\notin\mathbb{Q}$,
due to density of the irrational numbers in the real line. Therefore,
if we prove Theorem \ref{zeta alpha hypergeometric zeros} under the hypothesis $\frac{\theta_{M}}{2\pi}\notin\mathbb{Q}$,
we can easily see that the same strategy will work for $\frac{\theta_{M}}{2\pi}\in\mathbb{Q}$.
\end{enumerate}
\bigskip{}

We are ready to finish the proof of Theorem \ref{zeta alpha hypergeometric zeros}. By contradiction, let
us assume that $F_{z,\alpha}(s)$ has only a finite number of zeros
located on the line $\text{Re}(s)=\frac{\alpha}{4}$.

\bigskip{}

Under this hypothesis, there exists some $T_{0}>0$ such that one
of the following situations takes place:
\begin{enumerate}
\item $F_{z,\alpha}\left(\frac{\alpha}{4}+it\right)>0$ for $|t|>T_{0}$;
\item $F_{z,\alpha}\left(\frac{\alpha}{4}+it\right)<0$ for $|t|>T_{0}$;
\item $F_{z,\alpha}\left(\frac{\alpha}{4}+it\right)>0$ for $t>T_{0}$ and
$F_{z,\alpha}\left(\frac{\alpha}{4}+it\right)<0$ for $t<-T_{0}$;
\item $F_{z,\alpha}\left(\frac{\alpha}{4}+it\right)<0$ for $t>T_{0}$ and
$F_{z,\alpha}\left(\frac{\alpha}{4}+it\right)>0$ for $t<-T_{0}$.
\end{enumerate}
But the second contradiction hypothesis easily reduces to the first
if we carry out the proof with $\left(c_{j}\right)_{j\in\mathbb{N}}$
replaced by $\left(-c_{j}\right)_{j\in\mathbb{N}}$. Analogously,
the fourth hypothesis reduces to the third.  Henceforth, for the
sake of completing the proof, we just need to deal with cases 1. and
3. above. 

\bigskip{}

By hypotheses 1. and 3., we know that $F_{z,\alpha}\left(\frac{\alpha}{4}+it\right)>0$
for $t>T_{0}$: this implies the inequality,
\begin{equation}
\lim_{\omega\rightarrow\frac{\pi}{4}^{-}}\,\intop_{T_{0}}^{T}t^{p}F_{z,\alpha}\left(\frac{\alpha}{4}+it\right)\,e^{2\omega t}\,dt\leq\lim_{\omega\rightarrow\frac{\pi}{4}^{-}}\,\intop_{T_{0}}^{\infty}t^{p}F_{z,\alpha}\left(\frac{\alpha}{4}+it\right)\,e^{2\omega t}\,dt:=L_{z,p}(T_{0}),\label{implication limit}
\end{equation}
where $L_{z,p}(T_{0})$ exists and is finite, according to (\ref{first template}).
Therefore, for every $T\geq T_{0}$, the following inequality holds
\begin{equation}
\intop_{T_{0}}^{T}t^{p}\,F_{z,\alpha}\left(\frac{\alpha}{4}+it\right)\,e^{\frac{\pi}{2}t}\,dt\leq L_{z,p}\left(T_{0}\right).\label{inequality}
\end{equation}
 But since $T\geq T_{0}$ is arbitrary, (\ref{inequality}) implies
\begin{equation}
\intop_{T_{0}}^{\infty}t^{p}\,F_{z,\alpha}\left(\frac{\alpha}{4}+it\right)\,e^{\frac{\pi}{2}t}\,dt\leq L_{z,p}\left(T_{0}\right).\label{inequality 2}
\end{equation}

\bigskip{}

Using (\ref{inequality 2}) and one of the contradiction hypotheses
1. or 3., we now prove
\begin{equation}
\intop_{-\infty}^{\infty}t^{p}\,\left|F_{z,\alpha}\left(\frac{\alpha}{4}+it\right)\right|\,e^{\frac{\pi}{2}t}\,dt<\infty.\label{absolute convergence dominated}
\end{equation}

\bigskip{}

In fact, since the hypotheses 1. and 3. imply that $F_{z,\alpha}\left(\frac{\alpha}{4}+it\right)=\left|F_{z,\alpha}\left(\frac{\alpha}{4}+it\right)\right|$
for $t>T_{0}$ and $F_{z,\alpha}\left(\frac{\alpha}{4}+it\right)=\pm\left|F_{z,\alpha}\left(\frac{\alpha}{4}+it\right)\right|$
for\footnote{here, the $+$ sign is under hypothesis 1. while the $-$ sign is under hypothesis 3.} $t<-T_{0}$,
\begin{align}
\intop_{-\infty}^{\infty}t^{p}\,F_{z,\alpha}\left(\frac{\alpha}{4}+it\right)\,e^{\frac{\pi}{2}t}dt & =\intop_{0}^{\infty}\left\{ F_{z,\alpha}\left(\frac{\alpha}{4}+it\right)\,e^{\frac{\pi}{2}t}+(-1)^{p}\,F_{z,\alpha}\left(\frac{\alpha}{4}-it\right)\,e^{-\frac{\pi}{2}t}\right\} \,t^{p}dt\nonumber \\
 & =\intop_{0}^{T_{0}}\left\{ F_{z,\alpha}\left(\frac{\alpha}{4}+it\right)\,e^{\frac{\pi}{2}t}+(-1)^{p}\,F_{z,\alpha}\left(\frac{\alpha}{4}-it\right)\,e^{-\frac{\pi}{2}t}\right\} \,t^{p}dt\nonumber \\
 & +\intop_{T_{0}}^{\infty}\left\{ \left|F_{z,\alpha}\left(\frac{\alpha}{4}+it\right)\right|\,e^{\frac{\pi}{2}t}\pm(-1)^{p}\,\left|F_{z,\alpha}\left(\frac{\alpha}{4}-it\right)\right|\,e^{-\frac{\pi}{2}t}\right\} \,t^{p}dt.\label{before dominated}
\end{align}

If $F_{z,\alpha}\left(\frac{\alpha}{4}+it\right)>0$ for $t<-T_{0}$
(condition 1. above) we take $p$ as an even integer and if $F_{z,\alpha}\left(\frac{\alpha}{4}+it\right)<0$
for $t<-T_{0}$ (condition 3. above) we assume $p$ to be odd. In
any case, the second integral in (\ref{before dominated}) reduces
to
\begin{equation}
\intop_{T_{0}}^{\infty}\left\{ \left|F_{z,\alpha}\left(\frac{\alpha}{4}+it\right)\right|\,e^{\frac{\pi}{2}t}+\left|F_{z,\alpha}\left(\frac{\alpha}{4}-it\right)\right|\,e^{-\frac{\pi}{2}t}\right\} \,t^{p}dt=\intop_{-\infty}^{\infty}|t|^{p}\left|F_{z,\alpha}\left(\frac{\alpha}{4}+it\right)\right|\,e^{\frac{\pi}{2}t}dt-\intop_{-T_{0}}^{T_{0}}|t|^{p}\,\left|F_{z,\alpha}\left(\frac{\alpha}{4}+it\right)\right|\,e^{\frac{\pi}{2}t}dt,\label{almost at dominated}
\end{equation}
and so, combining (\ref{before dominated}) and (\ref{almost at dominated}),
\begin{align}
\intop_{-\infty}^{\infty}|t|^{p}\left|F_{z,\alpha}\left(\frac{\alpha}{4}+it\right)\right|\,e^{\frac{\pi}{2}t}dt & =\intop_{-\infty}^{\infty}t^{p}\,F_{z,\alpha}\left(\frac{\alpha}{4}+it\right)\,e^{\frac{\pi}{2}t}dt+\intop_{-T_{0}}^{T_{0}}|t|^{p}\,\left|F_{z,\alpha}\left(\frac{\alpha}{4}+it\right)\right|\,e^{\frac{\pi}{2}t}dt\nonumber \\
 & -\intop_{0}^{T_{0}}\left\{ F_{z,\alpha}\left(\frac{\alpha}{4}+it\right)\,e^{\frac{\pi}{2}t}+(-1)^{p}\,F_{z,\alpha}\left(\frac{\alpha}{4}-it\right)\,e^{-\frac{\pi}{2}t}\right\} \,t^{p}dt<\infty,\label{finishing dominated}
\end{align}
establishing (\ref{absolute convergence dominated}).

\bigskip{}

We are now free to use the dominated convergence Theorem:
by (\ref{first template}), (\ref{things almost at the end}) and
(\ref{absolute convergence dominated}),
\begin{align}
\intop_{0}^{\infty}\left\{ t^{p}\,F_{z,\alpha}\left(\frac{\alpha}{4}+it\right)\,e^{\frac{\pi}{2}t}+(-1)^{p}\,t^{p}\,F_{z,\alpha}\left(\frac{\alpha}{4}-it\right)\,e^{-\frac{\pi}{2}t}\right\} \,t^{p}\,dt\nonumber \\
=-\frac{8\pi}{2^{p}}\,w_{z}\,c_{M}r_{M}^{p}\,e^{\frac{\pi}{2}\lambda_{M}}\,\cos\left(p\theta_{M}+\frac{\pi\alpha}{8}+\beta_{z}\right)\,\left\{ 1+E(X,z)+H(X,z)\right\} .\label{one of final ineq}
\end{align}

But we have already seen that there are infinitely many integers $p\in\left(q_{n}\right)_{n\geq N_{0}}\cup\left(r_{n}\right)_{n\geq N_{0}}$
for which the right-hand side of (\ref{one of final ineq}) is negative,
this is, there are infinitely many $p$ such that
\begin{align}
\intop_{T_{0}}^{\infty}\left\{ F_{z,\alpha}\left(\frac{\alpha}{4}+it\right)\,e^{\frac{\pi}{2}t}+(-1)^{p}\,F_{z,\alpha}\left(\frac{\alpha}{4}-it\right)\,e^{-\frac{\pi}{2}t}\right\} t^{p}dt
<-\intop_{0}^{T_{0}}&\left\{ F_{z,\alpha}\left(\frac{\alpha}{4}+it\right)\,e^{\frac{\pi}{2}t}+(-1)^{p}\,F_{z,\alpha}\left(\frac{\alpha}{4}-it\right)\,e^{-\frac{\pi}{2}t}\right\} t^{p}dt <\nonumber \\
<T_{0}^{p}\,\intop_{0}^{T_{0}}\left|F_{z,\alpha}\left(\frac{\alpha}{4}+it\right)\,e^{\frac{\pi}{2}t}+(-1)^{p}\,F_{z,\alpha}\left(\frac{\alpha}{4}-it\right)\,e^{-\frac{\pi}{2}t}\right|&dt\leq K(T_{0})\,T_{0}^{p},\label{upper bound}
\end{align}
where $K$ only depends on $T_{0}$ and not on $p$. By our hypotheses
(1. or 3.) over $F_{z,\alpha}\left(\frac{\alpha}{4}+it\right)$ (and
for the choice of $p$ depending whether 1. or 3. takes place), we
know that there exists some $\epsilon:=\epsilon(T_{0})>0$ such that
\begin{align*}
F_{z,\alpha}\left(\frac{\alpha}{4}+it\right)\,e^{\frac{\pi}{2}t}+(-1)^{p}\,F_{z,\alpha}\left(\frac{\alpha}{4}-it\right)\,e^{-\frac{\pi}{2}t} & =\left|F_{z,\alpha}\left(\frac{\alpha}{4}+it\right)\right|\,e^{\frac{\pi}{2}t}+\left|F_{z,\alpha}\left(\frac{\alpha}{4}-it\right)\right|\,e^{-\frac{\pi}{2}t}\\
 & \geq\epsilon(T_{0}),\,\,\,\,\forall t\in\left[2T_{0},\,2T_{0}+1\right].
\end{align*}

This proves the lower bound:
\begin{align}
\intop_{T_{0}}^{\infty}\left\{ F_{z,\alpha}\left(\frac{\alpha}{4}+it\right)\,e^{\frac{\pi}{2}t}+(-1)^{p}\,F_{z,\alpha}\left(\frac{\alpha}{4}-it\right)\,e^{-\frac{\pi}{2}t}\right\} t^{p}dt & >\nonumber \\
>\intop_{2T_{0}}^{2T_{0}+1}\left\{ \left|F_{z,\alpha}\left(\frac{\alpha}{4}+it\right)\right|\,e^{\frac{\pi}{2}t}+\left|F_{z,\alpha}\left(\frac{\alpha}{4}-it\right)\right|\,e^{-\frac{\pi}{2}t}\right\} t^{p}dt & \geq\epsilon\left(T_{0}\right)\,\left(2T_{0}\right)^{p}.\label{lower bound}
\end{align}

\bigskip{}

Comparing (\ref{upper bound}) with (\ref{lower bound}), we conclude
that the inequality
\[
2^{p}\leq\frac{K(T_{0})}{\epsilon(T_{0})}
\]
must hold for infinitely many values of $p\in\left(p_{n}\right)_{n\geq N_{0}}\cup\left(q_{n}\right)_{n\geq N_{0}}$.
This is absurd because we can take $p$ as large as desired. $\blacksquare$

\begin{center}\section{Zeros of combinations attached to $\zeta(s,Q)$} \label{section Epstein binary} \end{center} 

In this section we apply all the previous work to prove a version of Theorem \ref{zeta alpha hypergeometric zeros} where $\zeta_{\alpha}(s)$ is replaced by the Epstein zeta function attached to
a binary, integral and positive definite quadratic form. To adapt
the work done in the proof of Theorem \ref{zeta alpha hypergeometric zeros}, as well as to apply the Summation
formula (\ref{final formula for 1f1 theorem}), we need some lemmas concerning twists of the Epstein zeta function by an additive character $e^{2\pi i n p/q}$, $p/q \in \mathbb{Q}$. 

The first lemma given in the next subsection is essentially given in {[}\cite{callahan_smith}, Theorems 2 and
3{]} and was used by Jutila to derive exponential sums attached to
Epstein zeta functions in \cite{jutila_quadratic}.

\bigskip{}

However, to match the notation introduced in our paper and to be easier
to apply the next result to our summation formula, we will give a brief proof of it. 

\newpage

\begin{center}
\subsection{Exponential sums attached to Quadratic forms} \label{section exponential}    
\end{center}

\bigskip{}

\begin{lemma}\label{Lemma Exponential sum Callahan}

Let $(p,\,q)=1$ and $Q(m,\,n)=Am^2+Bm\,n+Cn^2$ be a binary, positive definite and integral
quadratic form with discriminant $\Delta:=4AC-B^2$. Consider the periodic Epstein zeta function,
\begin{equation}
\zeta\left(s,\,Q,\,\frac{p}{q}\right):=\sum_{m,n\neq0}\frac{\exp\left(-\frac{2\pi ip}{q}\left(Am^{2}+Bm\,n+Cn^{2}\right)\right)}{\left(Am^{2}+Bm\,n+Cn^{2}\right)^{s}},\,\,\,\,\,\text{Re}(s)>1.\label{Epstein rational exponents}
\end{equation}

Then $\zeta\left(s,\,Q,\,\frac{p}{q}\right)$ can be analytically
continued (at most) as a meromorphic function with a simple pole located
at $s=1$, whose residue is
\begin{equation}
\text{Res}_{s=1}\zeta\left(s,\,Q,\,\frac{p}{q}\right)=\frac{2\pi}{q^{2}\sqrt{\Delta}}\,\sum_{k_{1},k_{2}=0}^{q-1}e^{-\frac{2\pi ip}{q}Q(k_{1},k_{2})}\label{residue periodic Epstein}
\end{equation}
(if the previous value is zero, $\zeta\left(s,\,Q,\,\frac{p}{q}\right)$
is entire).

Moreover, it satisfies the functional equation
\begin{equation}
\left(\frac{2\pi}{q\,\sqrt{\Delta}}\right)^{-s}\Gamma(s)\,\zeta\left(s,\,Q,\,\frac{p}{q}\right)=\left(\frac{2\pi}{q\,\sqrt{\Delta}}\right)^{-(1-s)}\Gamma\left(1-s\right)\,\tilde{\zeta}\left(1-s,\,Q,\,\frac{p}{q}\right),\label{Functional equation periodic}
\end{equation}
where $\tilde{\zeta}(s;\,Q;\,p/q)$ represents the Dirichlet series
\begin{equation}
\tilde{\zeta}\left(s,\,Q,\,\frac{p}{q}\right):=\sum_{m,n\neq0}\frac{b_{Q}(m,n,p/q)}{Q(m,n)^{s}},\,\,\,\,\,\text{Re}(s)>1,\label{Dual Dirichlet series Epstein}
\end{equation}
with 
\begin{equation}
b_{Q}(m,n,p/q):=\frac{1}{q}\sum_{k_{1},k_{2}=0}^{q-1}e^{-\frac{2\pi ip}{q}Q(k_{1},k_{2})}\,e^{\frac{2\pi i}{q}(mk_{1}+nk_{2})}.\label{Definition Arithmetical Function new}
\end{equation}  
\end{lemma}

\begin{proof}
For $\text{Re}(s)>1$, we start by writing (\ref{Epstein rational exponents})
as
\begin{align}
\zeta\left(s,\,Q,\,\frac{p}{q}\right) & :=\sum_{m,n\neq0}\frac{e^{-2\pi i\,\frac{p}{q}\,Q(m,n)}}{Q(m,n)^{s}}=\sum_{k_{1},k_{2}=0}^{q-1}e^{-2\pi i\,\frac{p}{q}Q\left(k_{1},\,k_{2}\right)}\,\sum_{\ell_{1},\ell_{2}\neq0}\frac{1}{Q\left(\ell_{1}q+k_{1},\,\ell_{2}q+k_{2}\right)^{s}}\nonumber \\
 & =q^{-2s}\sum_{k_{1},k_{2}=0}^{q-1}e^{-2\pi i\,\frac{p}{q}Q\left(k_{1},\,k_{2}\right)}\,\sum_{\ell_{1},\ell_{2}\neq0}\,\frac{1}{Q\left(\ell_{1}+\frac{k_{1}}{q},\,\ell_{2}+\frac{k_{2}}{q}\right)^{s}}.\label{first expression}
\end{align}
Note that the infinite series with respect to $\ell_{1},\ell_{2}$
is well defined because the denominator cannot be zero. For $\text{Re}(s)>1$, it represents a particular example of the Epstein
zeta function (\ref{deifnition epstein intor})
\begin{equation}
\zeta\left(s,\,\mathbf{g},\,\mathbf{h},\,Q\right):=\sum_{\mathbf{m}\in\mathbb{Z}^{n},\,\mathbf{m}+\mathbf{g}\neq\mathbf{0}}\frac{\exp\left(2\pi i\,\mathbf{m}\cdot\mathbf{h}\right)}{Q\left(\mathbf{m}+\mathbf{g}\right)^{s}},\,\,\,\,\,\text{Re}(s)>\frac{n}{2}\label{Epstein Paul Epstein}
\end{equation}
when $Q$ is binary and $\mathbf{h}=(0,0)$. We know from Lemma
\ref{Epstein functional} above that, for $\mathbf{h}\notin\mathbb{Z}^{n}$, $\zeta\left(s,\,\mathbf{g},\,\mathbf{h},\,Q\right)$
can be analytically continued as an entire function. If, however,
$\mathbf{h}\in\mathbb{Z}^{n}$, $\zeta\left(s,\,\mathbf{g},\,\mathbf{h},\,Q\right)$
has an analytic continuation into the complex plane as a meromorphic
function possessing a simple pole located at $s=n/2$ with residue
\begin{equation}
\text{Res}_{s=\frac{n}{2}}\zeta\left(s,\,\mathbf{g},\,\mathbf{h},\,Q\right)=\frac{(2\pi)^{n/2}}{\sqrt{D(Q)}}\,\Gamma\left(\frac{n}{2}\right).\label{residue Epstein classical}
\end{equation}

For $n=2$, we know that $D(Q)$ reduces to the discriminant $\Delta:=4AC-B^{2}$
of the binary quadratic form. Hence, for $\text{Re}(s)>1$, we can
reduce (\ref{first expression}) to
\begin{equation}
\zeta\left(s,\,Q,\,\frac{p}{q}\right)=q^{-2s}\sum_{k_{1},k_{2}=0}^{q-1}e^{-2\pi i\,\frac{p}{q}Q\left(k_{1},\,k_{2}\right)}\,\zeta\left(s,\,\frac{\mathbf{k}}{q},\,\mathbf{0},\,Q\right),\label{expression suitable for residue}
\end{equation}
which proves that $\zeta\left(s,\,Q,\,\frac{p}{q}\right)$ has an
analytic continuation as a meromorphic function with a simple
pole located at $s=1$ with residue
\[
\text{Res}_{s=1}\zeta\left(s,\,Q,\,\frac{p}{q}\right)=\frac{2\pi}{q^{2}\sqrt{\Delta}}\,\sum_{k_{1},k_{2}=0}^{q-1}e^{-\frac{2\pi ip}{q}Q(k_{1},k_{2})}.
\]

This proves the first part of the lemma. To prove the second part,
we use the functional equation for the Epstein zeta function (\ref{formula at intro})
\begin{equation}
\left(\frac{2\pi}{D(Q)^{1/n}}\right)^{-s}\Gamma(s)\,\zeta\left(s,\,\mathbf{g},\,\mathbf{h},\,Q\right)=\exp\left(-2\pi i\,\mathbf{g}\cdot\mathbf{h}\right)\,\left(\frac{2\pi}{D(Q^{\dagger})^{1/n}}\right)^{s-\frac{n}{2}}\Gamma\left(\frac{n}{2}-s\right)\,\zeta\left(\frac{n}{2}-s,\mathbf{h},\,-\mathbf{g},\,Q^{\dagger}\right),\label{functional equation general Epstein with exponential!}
\end{equation}
where $Q^{\dagger}$ is the adjoint quadratic form (\ref{adjoint quadratic defini}).
Since $\mathbf{h}=(0,0)$ and $Q$ is binary, $\zeta\left(s,\mathbf{0},-\frac{\mathbf{k}}{q},\,Q^{\dagger}\right)=\zeta\left(s,\mathbf{0},\frac{\mathbf{k}}{q},\,Q\right)$,
so by (\ref{first expression}) and the previous functional
equation (\ref{functional equation general Epstein with exponential!}) with $n=2$ and $D(Q)=\Delta$, we obtain 
\begin{equation}
\left(\frac{2\pi}{q\sqrt{\Delta}}\right)^{-s}\Gamma(s)\,\zeta\left(s,\,Q,\,\frac{p}{q}\right)=\left(\frac{2\pi}{q\sqrt{\Delta}}\right)^{-(1-s)}\Gamma\left(1-s\right)\,\frac{1}{q}\,\sum_{k_{1},k_{2}=0}^{q-1}e^{-2\pi i\,\frac{p}{q}Q\left(k_{1},\,k_{2}\right)}\,\zeta\left(1-s,\,\mathbf{0},\,\frac{\mathbf{k}}{q},\,Q\right).\label{Functional equation}
\end{equation}

Now, for $\text{Re}(s)>1$, it is simple to see that
\begin{align*}
\frac{1}{q}\,\sum_{k_{1},k_{2}=0}^{q-1}e^{-2\pi i\,\frac{p}{q}Q\left(k_{1},\,k_{2}\right)}\,\zeta\left(s,\,\frac{\mathbf{k}}{q},\,\mathbf{0},\,Q\right) & =\frac{1}{q}\sum_{k_{1},k_{2}=0}^{q-1}e^{-2\pi i\,\frac{p}{q}Q\left(k_{1},\,k_{2}\right)}\,\sum_{m,n\neq0}\frac{e^{\frac{2\pi i}{q}\left(m\,k_{1}+n\,k_{2}\right)}}{Q\left(m,n\right)^{s}}=\\
=\sum_{m,n\neq0}\frac{1}{Q(m,n)^{s}}\,\frac{1}{q}\,\sum_{k_{1},k_{2}=0}^{q-1}\exp & \left(-2\pi i\,\frac{p}{q}Q\left(k_{1},\,k_{2}\right)+\frac{2\pi i}{q}\left(m\,k_{1}+n\,k_{2}\right)\right)=\sum_{m,n\neq0}\frac{b_{Q}(m,n,p/q)}{Q(m,n)^{s}},
\end{align*}
proving (\ref{Functional equation periodic}).    
\end{proof}

\begin{remark}\label{writing as single Dirichlet series Epstein}
Note that both $\zeta\left(s,\,Q,\,\frac{p}{q}\right)$ and $\tilde{\zeta}\left(s,\,Q,\,\frac{p}{q}\right)$
can be written as Dirichlet series over one variable of summation.
It is not hard to see that, if $r_{Q}(n)$ denotes the number of representations
of $n$ by $Q$, then
\begin{equation}
\zeta\left(s,\,Q,\,\frac{p}{q}\right)=\sum_{n=1}^{\infty}\frac{r_{Q}(n)\,\exp\left(-\frac{2\pi i\,p}{q}n\right)}{n^{s}},\,\,\,\,\,\,\text{Re}(s)>1\label{representation single Dirichlet periodic}
\end{equation}
and
\begin{equation}
\tilde{\zeta}\left(s,\,Q,\,\frac{p}{q}\right):=\sum_{n=1}^{\infty}\frac{\tilde{b}_{Q}(n,p/q)}{n^{s}},\,\,\,\,\,\,\text{Re}(s)>1,\label{Representation single Dirichlet dual}
\end{equation}
where
\begin{equation}
\tilde{b}_{Q}(n,p/q):=\frac{1}{q}\,\sum_{Q(\alpha,\beta)=n}\,\sum_{k_{1},k_{2}=0}^{q-1}\,\exp\left(-\frac{2\pi ip}{q}Q(k_{1},k_{2})+\frac{2\pi i}{q}(\alpha k_{1}+\beta k_{2})\right).\label{definition new arithmetical function}
\end{equation}

In the definition above we are summing over pairs of integers $(\alpha,\beta)$ such that $Q(\alpha,\beta)=n$. Note that $|\tilde{b}_{Q}(n,p/q)|\leq r_{Q}(n)$, so the absolute
convergence of the Dirichlet series on the right-hand side of (\ref{Representation single Dirichlet dual})
is assured for $\text{Re}(s)>1$.
\end{remark}

From the previous lemma, it is natural to apply the summation formula
given in Theorem \ref{summation formula with 1F1} not only to $\zeta(s,Q)$ (as we have done in
Example \ref{example_Epstein}) but also to $\zeta(s,Q,p/q)$. The next result gives
one particular case of our summation formula which will be crucial
in the proof of Theorem \ref{Epstein result}.

\begin{lemma} \label{exponential sum Epstein}
Let $Q(m,n)=Am^{2}+Bm\,n+Cn^{2}$ be an integral and positive definite
quadratic form and $\Delta:=4AC-B^{2}$ be its discriminant. Then, for every $\text{Re}(x)>0$,
$y\in\mathbb{C}$ and $(p,q)=1$, the following summation formula
takes place
\begin{align}
1+\sum_{n=1}^{\infty}r_{Q}(n)\,\text{exp}\left(-\frac{2\pi ip}{q}n\right)\,e^{-\frac{2\pi n\,x}{q\sqrt{\Delta}}}\,J_{0}\left(\frac{\sqrt{2\pi n}\,y}{\sqrt{q}\Delta^{1/4}}\right)\nonumber \\
=\frac{\,e^{-\frac{y^{2}}{4x}}}{qx}\,\sum_{k_{1},k_{2}=0}^{q-1}e^{-\frac{2\pi ip}{q}Q(k_{1},k_{2})}+\frac{e^{-\frac{y^{2}}{4x}}}{x}\,\sum_{n=1}^{\infty}\tilde{b}_{Q}(n,p/q)\,e^{-\frac{2\pi n}{q\sqrt{\Delta}\,x}}\,I_{0}\left(\frac{\sqrt{2\pi n}\,y}{\sqrt{q}\,\Delta^{1/4}\,x}\right).\label{summation formula that's the way}
\end{align}
where $\tilde{b}_{Q}(n,p/q)$ is explicitly given by (\ref{definition new arithmetical function}).   
\end{lemma}
\begin{proof}
Since the functional equation (\ref{Functional equation periodic})
takes place and the analytic continuation of $\zeta(s,\,Q,\,p/q)$
satisfies all the properties of the class $\mathcal{A}$, we can invoke
(\ref{final formula for 1f1 theorem}) for the pair of Dirichlet series: 
\[
\phi(s)=\left(\frac{2\pi}{q\sqrt{\Delta}}\right)^{-s}\,\sum_{n=1}^{\infty}\frac{\exp\left(-\frac{2\pi i\,p}{q}n\right)\,r_{Q}(n)}{n^{s}},\,\,\,\,\text{Re}(s)>1
\]
and 
\[
\psi(s)=\left(\frac{2\pi}{q\sqrt{\Delta}}\right)^{-s}\,\sum_{n=1}^{\infty}\frac{\tilde{b}_{Q}(n,p/q)}{n^{s}},\,\,\,\,\text{Re}(s)>1.
\]

Under this template, from (\ref{residue periodic Epstein}) we know
that the residue of $\phi(s)$ is $\rho=q^{-1}\,\sum_{k_{1},k_{2}=0}^{q-1}e^{-\frac{2\pi ip}{q}Q(k_{1},k_{2})}$.
To compute $\phi(0)$, it suffices to know $\zeta(0,Q,p/q)$. To this
end, we use (\ref{Functional equation}): multiplying both sides of
this equation by $s$ and letting $s\rightarrow0$, we obtain 
\begin{align*}
\zeta\left(0,\,Q,\,\frac{p}{q}\right) & =\lim_{s\rightarrow0}\,s\,\left(\frac{2\pi q}{\sqrt{\Delta}}\right)^{s-1}\Gamma\left(1-s\right)\,\frac{1}{q}\,\sum_{k_{1},k_{2}=0}^{q-1}e^{-2\pi i\,\frac{p}{q}Q\left(k_{1},\,k_{2}\right)}\,\zeta\left(1-s,\,\mathbf{0},\,\frac{\mathbf{k}}{q},\,Q\right)\\
=\lim_{s\rightarrow0}\,s\,\left(\frac{2\pi}{q\sqrt{\Delta}}\right)^{s-1}\Gamma\left(1-s\right) & \,\frac{1}{q}\,\left\{ \zeta\left(1-s,\,\mathbf{0},\,\mathbf{0},\,Q\right)+\sum_{k_{1},k_{2}\neq0}^{q-1}e^{-2\pi i\,\frac{p}{q}Q\left(k_{1},\,k_{2}\right)}\,\zeta\left(1-s,\,\mathbf{0},\,\frac{\mathbf{k}}{q},\,Q\right)\right\} .
\end{align*}

If $k_{1}\neq0$ or $k_{2}\neq0$, then $\mathbf{k}/q\notin\mathbb{Z}^{2}$, and 
so each Dirichlet series on the second sum $\zeta(1-s,\mathbf{0},\mathbf{k}/q,Q)$ is an entire function by Lemma \ref{Epstein functional}. 
Therefore, the only singularity comes from $\zeta(1-s,\mathbf{0},\mathbf{0},Q)$ at $s=0$, giving
\begin{equation}
\zeta\left(0,\,Q,\,\frac{p}{q}\right)=-\left(\frac{2\pi}{q\sqrt{\Delta}}\right)^{-1}\,\frac{1}{q}\,\lim_{s\rightarrow1}\,(s-1)\,\zeta\left(s,\,\mathbf{0},\,\mathbf{0},\,Q\right)=-1,\label{value at zero important}
\end{equation}
where we have used (\ref{residue Epstein classical}). An application
of the summation formula (\ref{final formula for 1f1 theorem}) yields
(\ref{summation formula that's the way}) immediately.     
\end{proof}

\begin{center} \subsection{The behavior of $\psi_{Q}(x,z)$}\end{center}

Using the previous lemma, we can now study the behavior of the analogue of Jacobi's $\psi-$function attached to $\zeta(s,Q)$. The following lemma essentially says the same as Lemma \ref{Vanishing Theta lemma} in the present situation. 

\begin{lemma}\label{vanishing theta Quadratic}

Let $Q(m,n)=Am^{2}+Bm\,n+Cn^{2}$ be a binary, integral and positive
definite quadratic form and let $\Delta:=4AC-B^{2}$. Assume also
that $Q$ is reduced, this is, $\text{gcd}(A,\,B,\,C)=1$ and that
$\sqrt{\Delta}\equiv2\mod4$. Consider the generalized Jacobi's
$\psi-$function attached to $Q$ (given in Example \ref{example_Epstein}),
\begin{equation}
\psi_{Q}(x,z)=\sum_{n=1}^{\infty}r_{Q}(n)\,e^{-\frac{2\pi nx}{\sqrt{\Delta}}}\,J_{0}\left(\sqrt{\frac{2\pi nx}{\Delta^{1/2}}}\,z\right),\,\,\,\,\,\,\text{Re}(x)>0,\,\,\,z\in\mathbb{C}.\label{jacobi psi quadratic}
\end{equation}

Then, for every $z$ satisfying the condition:
\begin{equation}
z\in\mathscr{D}_{Q}:=\left\{ z\in\mathbb{C}\,:\,|\text{Re}(z)|<\frac{2\sqrt{\pi}}{\Delta^{3/4}},\,\,|\text{Im}(z)|<\frac{2\sqrt{\pi}}{\Delta^{3/4}}\right\} ,\label{condition quadratic form vanishing theta}
\end{equation}
every $m\in\mathbb{N}_{0}$ and every analytic function $h:\,\mathbb{C}\longmapsto\mathbb{C}$,
one has the asymptotic formula
\begin{equation}
\lim_{\omega\rightarrow\frac{\pi}{4}^{-}}\,\frac{d^{m}}{d\omega^{m}}\left\{ h(\omega)\,\left(e^{-z^{2}/8}+e^{z^{2}/8}\,\psi_{Q}\left(e^{2i\omega},z\right)\right)\right\} =-2\,h^{(m)}\left(\frac{\pi}{4}\right)\,\sinh\left(\frac{z^{2}}{8}\right).\label{quadratic form desired conclusion}
\end{equation}

\end{lemma}
\begin{proof}
We proceed as in the proof of lemma \ref{Vanishing Theta lemma}. By Lemma \ref{vanishing with analytic} and Corollary \ref{integral representation theta}, it suffices to compute $\psi_{Q}(i+\delta,\,z)$ as
$\delta\rightarrow0^{+}$. Note that
\[
1+\psi_{Q}(i+\delta,z)=1+\sum_{n=1}^{\infty}r_{Q}(n)\,e^{-\frac{2\pi ni}{\sqrt{\Delta}}}\,e^{-\frac{2\pi n\delta}{\sqrt{\Delta}}}\,J_{0}\left(\sqrt{\frac{2\pi n(i+\delta)}{\Delta^{1/2}}}\,z\right).
\]

Applying the previous summation formula (\ref{summation formula that's the way})
with $p=1,\,q=\sqrt{\Delta}$, $x=\sqrt{\Delta}\,\delta$ and $y=\sqrt{i+\delta}\,z\,\Delta^{1/4}$,
we obtain
\begin{equation}
1\,+\,\psi_{Q}(i+\delta,z)=\frac{\,e^{-\frac{(i+\delta)z^{2}}{4\delta}}}{\Delta\,\delta}\,\sum_{k_{1},k_{2}=0}^{\sqrt{\Delta}-1}e^{-\frac{2\pi i}{\sqrt{\Delta}}Q(k_{1},k_{2})}+\frac{e^{-\frac{(i+\delta)z^{2}}{4\delta}}}{\sqrt{\Delta}\,\delta}\,\sum_{n=1}^{\infty}\tilde{b}_{Q}\left(n,1/\sqrt{\Delta}\right)\,e^{-\frac{2\pi n}{\Delta^{3/2}\delta}}\,I_{0}\left(\frac{\sqrt{2\pi n(i+\delta)}}{\Delta^{3/4}\,}\,\frac{z}{\delta}\right)\label{first step towards proof}
\end{equation}

\bigskip{}

To have the conclusion (\ref{quadratic form desired conclusion}),
we need to show that the right-hand side of (\ref{first step towards proof})
tends to zero exponentially fast as $\delta\rightarrow0^{+}$. To
do this, one needs to get rid of the residual term involving $\delta^{-1}$.

\bigskip{}

We now show that, under the conditions above, the residual term (\ref{residue periodic Epstein})
is zero: by hypothesis, we know that $\sqrt{\Delta}$ is an integer
and that $\sqrt{\Delta}^{2}+B^{2}=4AC$. Since the sum of two odd
numbers is not divisible by $4$, this automatically shows that $B$
and $\sqrt{\Delta}$ must be even.

\bigskip{}

Since $(A,\,B,\,C)=1$ and $B$ is even, at least $A$ or $C$ is
odd and without any loss of generality suppose that $A$ is odd. We
now evaluate the Gauss sum (\ref{residue periodic Epstein}), restricting
ourselves to the sum with respect to $k_{1}$. We will see that this
sum vanishes if $\sqrt{\Delta}\equiv2\mod4$. Indeed
\begin{align*}
\sum_{k_{1},k_{2}=0}^{\sqrt{\Delta}-1}e^{-\,\frac{2\pi i}{\sqrt{\Delta}}Q\left(k_{1},\,k_{2}\right)} & =\sum_{k_{2}=0}^{\sqrt{\Delta}}\left\{ \sum_{k_{1}=0}^{\sqrt{\Delta}/2-1}e^{-\,\frac{2\pi i}{\sqrt{\Delta}}Q\left(k_{1},\,k_{2}\right)}+\sum_{k_{1}=\sqrt{\Delta}/2}^{\sqrt{\Delta}-1}\,e^{-\,\frac{2\pi i}{\sqrt{\Delta}}Q\left(k_{1},\,k_{2}\right)}\right\} \\
 & =\sum_{k_{2}=0}^{\sqrt{\Delta}}\left\{ \sum_{k_{1}=0}^{\sqrt{\Delta}/2-1}e^{-\,\frac{2\pi i}{\sqrt{\Delta}}Q\left(k_{1},\,k_{2}\right)}+\sum_{k_{1}=0}^{\sqrt{\Delta}/2-1}\,e^{-\,\frac{2\pi i}{\sqrt{\Delta}}Q\left(k_{1}+\frac{\sqrt{\Delta}}{2},\,k_{2}\right)}\right\} .
\end{align*}

However, one easily checks that 
\begin{equation}
Q\left(k_{1}+\frac{\sqrt{\Delta}}{2},\,k_{2}\right)=Q\left(k_{1},k_{2}\right)+\sqrt{\Delta}\left(\frac{A}{4}\sqrt{\Delta}+A\,k_{1}+\frac{B}{2}k_{2}\right),\label{quadratic form things}
\end{equation}
and so,
\begin{align*}
\exp\left(-\frac{2\pi i}{\sqrt{\Delta}}Q\left(k_{1}+\frac{\sqrt{\Delta}}{2},\,k_{2}\right)\right) & =\exp\left(-\frac{2\pi i}{\sqrt{\Delta}}Q(k_{1},k_{2})-2\pi i\left(\frac{A}{4}\sqrt{\Delta}+A\,k_{1}+\frac{B}{2}k_{2}\right)\right)\\
=\exp\left(-\frac{2\pi i}{\sqrt{\Delta}}Q(k_{1},k_{2})\right)\cdot & \,\exp\left(-\frac{\pi i}{2}A\sqrt{\Delta}\right)=-\exp\left(-\frac{2\pi i}{\sqrt{\Delta}}Q(k_{1},k_{2})\right),
\end{align*}
because $A$ is odd and $\sqrt{\Delta}\equiv2\mod4$. Therefore,
\[
\sum_{k_{1},k_{2}=0}^{\sqrt{\Delta}-1}e^{-\,\frac{2\pi i}{\sqrt{\Delta}}Q\left(k_{1},\,k_{2}\right)}=0,
\]
so that our summation formula is reduced to
\begin{equation}
1+\tilde{\psi}_{Q}(i+\delta,z)=\frac{e^{-\frac{(i+\delta)z^{2}}{4\delta}}}{\sqrt{\Delta}\,\delta}\,\sum_{n=1}^{\infty}\tilde{b}_{Q}\left(n,1/\sqrt{\Delta}\right)\,e^{-\frac{2\pi n}{\Delta^{3/2}\delta}}\,I_{0}\left(\frac{\sqrt{2\pi n(i+\delta)}}{\Delta^{3/4}\,}\,\frac{z}{\delta}\right).\label{summation quadratic}
\end{equation}

\bigskip{}

Now it is a matter of following the proof of Lemma \ref{Vanishing Theta lemma}: we need to show that
\[
\lim_{\delta\rightarrow0^{+}}\left|\frac{e^{-\frac{(i+\delta)z^{2}}{4\delta}}}{\delta}\,e^{-\frac{2\pi n}{\Delta^{3/2}\delta}}I_{0}\left(\frac{\sqrt{2\pi n(i+\delta)}}{\Delta^{3/4}\,}\,\frac{z}{\delta}\right)\right|=0
\]
for any fixed $n\in\mathbb{N}$. To do it, we just need to invoke
the simple bound (\ref{Bound Bessel}) and to recall that $\text{Re}\left(\sqrt{i+\delta}\,z\right)\rightarrow\text{Re}(\sqrt{i}\,z)=\frac{1}{\sqrt{2}}\left(\text{Re}(z)-\text{Im}(z)\right)$
and $\text{Re}\left((i+\delta)\,z^{2}\right)\rightarrow-2\,\text{Re}(z)\,\text{Im}(z)$
as $\delta\rightarrow0^{+}$, so that the term in the previous equation is bounded by 
\[
\lim_{\delta\rightarrow0^{+}}\,\frac{1}{\delta}\,\exp\left(-\frac{2\pi}{\Delta^{3/2}\delta}\left(n-\frac{\Delta^{3/4}}{2\sqrt{\pi}}\left|\text{Re}(z)-\text{Im}(z)\right|\,\sqrt{n}-\frac{\text{Re}(z)\,\text{Im}(z)\,\Delta^{3/2}}{4\pi}\right)\right).
\]

As before, the expression in the exponential is a quadratic polynomial with
variable $X=\sqrt{n}$ of the form
\[
P(X)=X^{2}-\frac{\Delta^{3/4}}{2\sqrt{\pi}}\left|\text{Re}(z)-\text{Im}(z)\right|\,X-\frac{\text{Re}(z)\,\text{Im}(z)\,\Delta^{3/2}}{4\pi}.
\]

After computing the zero of $P(X)$ with highest value, we know that
a sufficient condition for the positivity of $P(X)$ at a point $X=X_{0}$
is that
\begin{equation}
\frac{\Delta^{3/4}}{4\sqrt{\pi}}\left[|\text{Re}(z)-\text{Im}(z)|+|\text{Re}(z)+\text{Im}(z)|\right]<X_{0}.\label{condition discrimiiiinant}
\end{equation}

Thus, we want (\ref{condition discrimiiiinant}) to hold for
every $X_{0}=\sqrt{n}$ and every $n\in\mathbb{N}$. This is satisfied
if: 
\begin{equation}
|\text{Re}(z)-\text{Im}(z)|+|\text{Re}(z)+\text{Im}(z)|<\frac{4\sqrt{\pi}}{\Delta^{3/4}},\label{condition starting at n 1}
\end{equation}
which is the case whenever $z\in\mathscr{D}_{Q}=\left\{ z\in\mathbb{C}\,:\,|\text{Re}(z)|<\frac{2\sqrt{\pi}}{\Delta^{3/4}},\,\,|\text{Im}(z)|<\frac{2\sqrt{\pi}}{\Delta^{3/4}}\right\} $.
Now it is just a matter of following the conclusion of Lemmas \ref{Vanishing Theta lemma} and \ref{vanishing with analytic}
to deduce (\ref{quadratic form desired conclusion}).

\end{proof}

\begin{remark} \label{discriminant condition domain}
Note that the condition (\ref{condition starting at n 1}) is indicated
to assure that (\ref{condition discrimiiiinant}) holds for every
$X_{0}=\sqrt{n}$, $n\in\mathbb{N}$. But this condition can be relaxed
and replaced by
\begin{equation}
|\text{Re}(z)-\text{Im}(z)|+|\text{Re}(z)+\text{Im}(z)|<\frac{4\sqrt{\pi\,m_{Q}}}{\Delta^{3/4}},\label{better condition quadratic form}
\end{equation}
where $m_{Q}$ is the least integer for which $\tilde{b}_{Q}\left(m_{Q},1/\sqrt{\Delta}\right)\neq0$.
For example, when $Q(m,n)=Q_{0}(m,n):=m^{2}+n^{2}$, $\Delta=4$,
the condition (\ref{condition quadratic form vanishing theta}) does
not reduce to (\ref{condition rectangle z zeta alpha}) for $\alpha=2$.
This is because
\begin{align*}
\tilde{b}_{Q_{0}}\left(1,\frac{1}{2}\right) & :=\frac{1}{2}\,\sum_{\alpha^{2}+\beta^{2}=1}\,\sum_{k_{1},k_{2}=0}^{1}\,\exp\left(-\pi i\,\left(k_{1}^{2}+k_{2}^{2}\right)+\pi i\,(\alpha k_{1}+\beta k_{2})\right)\\
 & =0,
\end{align*}
while $\tilde{b}_{Q_{0}}(2,1/2)\neq0$. Thus, $m_{Q_{0}}=2$, and
then (\ref{better condition quadratic form}) now reduces to (\ref{condition rectangle z zeta alpha}).
It is a nice exercise of notation and arrangement of variables to
check that (\ref{summation quadratic}) reduces to (\ref{final formula for 1f1 theorem-1})
when $Q(m,n)=m^{2}+n^{2}$.  
\end{remark}

\begin{center}\subsection{Proof of Theorem \ref{Epstein result}} \label{section 4} \end{center} 

Now, the proof of our Theorem \ref{Epstein result} is almost the same (modulo different
computations) as the proof of Theorem \ref{zeta alpha hypergeometric zeros}. Therefore, we will brief
and just indicate the main steps. Start with integral representation
(\ref{Integral representation Epstein binary}) given in Example \ref{example_Epstein}, 
replacing there $x=e^{2i\omega}$, with $-\frac{\pi}{4}<\omega<\frac{\pi}{4}$,
\[
\frac{1}{2\pi}\,\intop_{-\infty}^{\infty}\eta_{Q}\left(\frac{1}{2}+it\right)\,_{1}F_{1}\left(\frac{1}{2}+it;\,1;\,-\frac{z^{2}}{4}\right)e^{2\omega t}\,dt=e^{i\omega}\,\psi_{Q}\left(e^{2i\omega},z\right)-e^{-i\omega}\,e^{-z^{2}/4}.
\]

From Kummer's formula (\ref{Kummer confluent transformation}) and adding and subtracting the term $e^{-z^{2}/8}\,e^{i\omega}$,
we obtain
\begin{equation}
\frac{e^{-z^{2}/8}}{2\pi}\,\intop_{-\infty}^{\infty}\eta_{Q}\left(\frac{1}{2}+it\right)\,_{1}F_{1}\left(\frac{1}{2}-it;\,1;\,-\frac{z^{2}}{4}\right)e^{2\omega t}\,dt=-2\,e^{-z^{2}/8}\,\cos\left(\omega\right)+e^{i\omega}\left(e^{-z^{2}/8}+e^{z^{2}/8}\,\psi_{Q}\left(e^{2i\omega},z\right)\right),\label{arranging symmetric form in quadratic}
\end{equation}
which is analogous to (\ref{symmetric identity}) for $\alpha=2$.
Therefore, we can use the same computations as in the previous section:
changing the variable $t$ to $t+\lambda_{j}$, differentiating $p$
times on both sides and appealing to the notation (\ref{polar coordinates sequence}), we deduce
\begin{align*}
\intop_{-\infty}^{\infty}t^{p}\eta_{Q}\left(\frac{1}{4}+i\,(t+\lambda_{j})\right)\,_{1}F_{1}\left(\frac{1}{2}-i\,(t+\lambda_{j});\,1;\,\frac{z^{2}}{4}\right)\,e^{2\omega t}\,dt\\
=-4\pi\,\left(\frac{r_{j}}{2}\right)^{p}\,e^{-2\omega\lambda_{j}}\,\cos\left(p\,\theta_{j}+\omega\right)+\frac{2\pi}{2^{p}}\,e^{z^{2}/8}\,\frac{d^{p}}{d\omega^{p}}\left\{ e^{i\omega-2\omega\lambda_{j}}\left(e^{-z^{2}/8}+e^{z^{2}/8}\,\psi_{Q}\left(e^{2i\omega},z\right)\right)\right\} .
\end{align*}

Continuing to argue in the same manner, we can fully justify the equality
(c.f. (\ref{Taking real parts in that way}) above)
\begin{align}
\intop_{-\infty}^{\infty}t^{p}\,F_{z,Q}\left(\frac{1}{2}+it\right)\,e^{2\omega t}\,dt & =-\frac{8\pi}{2^{p}}\,\sum_{j=1}^{\infty}c_{j}\,r_{j}^{p}e^{-2\omega\lambda_{j}}\,\left[\cos\left(p\,\theta_{j}+\omega\right)\right]+\nonumber \\
+\frac{4\pi}{2^{p}}\,\text{Re} & \left(e^{z^{2}/8}\,\frac{d^{p}}{d\omega^{p}}\left\{ \sum_{j=1}^{\infty}c_{j}e^{i\omega-2\omega\lambda_{j}}\left(e^{-z^{2}/8}+e^{z^{2}/8}\,\psi_{Q}\left(e^{2i\omega},z\right)\right)\right\} \right),\label{euality quququadratic}
\end{align}
where 
\[
F_{z,Q}(s)=\sum_{j=1}^{\infty}c_{j}\,\left(\frac{2\pi}{\sqrt{\Delta}}\right)^{-(s+i\lambda_{j})}\Gamma\left(s+i\lambda_{j}\right)\,\zeta(s+i\lambda_{j},\,Q)\,\left\{ _{1}F_{1}\left(1-s-i\lambda_{j};\,1;\,\frac{z^{2}}{4}\right)+\,_{1}F_{1}\left(1-\overline{s}+i\lambda_{j};\,1;\,\frac{\overline{z}^{2}}{4}\right)\right\}.
\]

\bigskip{}
By letting $\omega\rightarrow\frac{\pi}{4}^{-}$, we see from Lemma
\ref{vanishing theta Quadratic} above (namely, relation (\ref{quadratic form desired conclusion}))
that the right-hand side of (\ref{euality quququadratic}) can be
simplified to:
\[
-\frac{8\pi}{2^{p}}\,\sum_{j=1}^{\infty}c_{j}r_{j}^{p}e^{-\frac{\pi}{2}\lambda_{j}}\,\left[\cos\left(p\,\theta_{j}+\frac{\pi}{4}\right)\left(1+\text{Re}\left(e^{z^{2}/8}\,\sinh\left(\frac{z^{2}}{8}\right)\right)\right)-\sin\left(p\theta_{j}+\frac{\pi}{4}\right)\,\text{Im}\left(e^{z^{2}/8}\,\sinh\left(\frac{z^{2}}{8}\right)\right)\right],
\]
which is exactly the same as the right-hand side of (\ref{first template})
with $\alpha=2$. Therefore, from this point onward, the proof follows exactly the
same principles as the proof of Theorem \ref{zeta alpha hypergeometric zeros}. $\blacksquare$

\bigskip{}

\bigskip{}

\section{Zeros of combinations attached to $L_{k}(s,\chi)$}

Similarly to the previous sections, in order to study the zeros of $L_{k}(s,\chi)$
we need a Lemma which gives the analytic
continuation of a periodic Dirichlet series containing information about the behavior of $\psi_{k,\chi}(x,z)$ (\ref{analogues character Jacobi psi!}). 

Before proving such an important lemma, we first need an auxiliary result, which will be the next lemma. 

Although the proof of our next result is quite standard and it
is a known result for $\delta=0$ (see Lemma \ref{Epstein functional} above)
we prove it for the case where $\delta=1$ because this might be instructive
to some readers. Although straightforward, we could not track any reference
containing explicitly its statement. 

\begin{center}
 \subsection{Exponential sums attached to Dirichlet characters} \label{exp dirichlet characters}   
\end{center}

\begin{lemma}\label{functional equation shifted Epstein}
Let $\delta\in\{0,1\}$ and $0<a_{i}<1$ for every $i\in\{1,...,k\}$.
Consider the Dirichlet series
\begin{equation}
\varphi_{k,\delta}(s;a_{1},...,a_{k}):=\sum_{n_{1},...,n_{k}\in\mathbb{Z}}\,\frac{\left(n_{1}+a_{1}\right)^{\delta}\cdot...\cdot\left(n_{k}+a_{k}\right)^{\delta}}{\left(\left(n_{1}+a_{1}\right)^{2}+...+\left(n_{k}+a_{k}\right)^{2}\right)^{s}},\,\,\,\,\text{Re}(s)>\frac{k}{2}\left(1+\delta\right).\label{Definition Shifted Epstein}
\end{equation}

Then $\varphi_{k,\delta}(s;\,a_{1},...,a_{k})$ has the following
properties:
\begin{enumerate}
\item If $\delta=0$, it can be continued as a meromorphic function with
a simple pole located at $s=\frac{k}{2}$ with residue $\text{Res}_{s=k/2}\,\,\varphi_{k,0}(s;\,a_{1},...,a_{k})=\frac{\pi^{k/2}}{\Gamma(k/2)}$.
\item If $\delta=1$, it can be continued as an entire function.
\end{enumerate}
Moreover, it satisfies the functional equation:
\begin{equation}
\pi^{-s}\Gamma(s)\,\varphi_{k,\delta}(s;\,a_{1},...,a_{k})=(-i)^{\delta k}\,\pi^{-\left(k\left(\frac{1}{2}+\delta\right)-s\right)}\Gamma\left(k\left(\frac{1}{2}+\delta\right)-s\right)\,\tilde{\varphi}_{k,\delta}\left(k\left(\frac{1}{2}+\delta\right)-s;\,a_{1},...,a_{k}\right),\label{functional equation auxiliar function Dirichlet}
\end{equation}
where, for $\text{Re}(s)>\frac{k}{2}(1+\delta)$, $\tilde{\varphi}_{k,\delta}(s;\,a_{1},...,a_{k})$
can be written as the Dirichlet series
\begin{equation}
\tilde{\varphi}_{k,\delta}\left(s;\,a_{1},...,a_{k}\right)=\sum_{n_{1},...,n_{k}\neq0}\frac{n_{1}^{\delta}\cdot...\cdot n_{k}^{\delta}\,e^{2\pi i\left(n_{1},...,n_{k}\right)\cdot(a_{1},...,a_{k})}}{\left(n_{1}^{2}+...+n_{k}^{2}\right)^{s}}.\label{Dual Dirichlet series}
\end{equation}

Furthermore, $\tilde{\varphi}_{k,\delta}\left(s;\,a_{1},...,a_{k}\right)$
can be analytically continued as an entire function no matter the
value of $\delta$.    
\end{lemma}

\begin{proof}
Note that the condition $0<a_{i}<1$ implies that the multiple series (\ref{Definition Shifted Epstein})
is well-defined for any $\left(n_{1},...,n_{k}\right)\in\mathbb{Z}^{k}$.
Note also that for $\delta=0$ this result is already known because
the Dirichlet series is reduced (\ref{Definition Shifted Epstein})
to the Epstein zeta function (\ref{deifnition epstein intor}), this is
\[
\varphi_{k,0}(s;a_{1},...,a_{k})=\zeta\left(s,\mathbf{a},\,\mathbf{0},\,\mathbf{I}_{k}\right),
\]
where $\mathbf{a}=(a_{1},...,a_{k})$ and $\mathbf{I}_{k}$ denotes
the diagonal quadratic form $n_{1}^{2}+...+n_{k}^{2}$. The pole of
$\varphi_{k,0}(s;a_{1},...,a_{k})$ comes from the pole of the Epstein
zeta function at $s=k/2$. Furthermore, we have
\[
\tilde{\varphi}_{k,0}\left(s;\,a_{1},...,a_{k}\right)=\zeta\left(s,\,\mathbf{0},\,\mathbf{a},\,\mathbf{I}_{k}\right).
\]

Since $\left(a_{1},...,a_{k}\right)\notin\mathbb{Z}^{k}$, by the
analytic continuation of Epstein's zeta function we see that $\tilde{\varphi}_{k,0}\left(s;\,a_{1},...,a_{k}\right)$
must be entire. The functional equation (\ref{functional equation auxiliar function Dirichlet})
is an immediate corollary of (\ref{formula at intro}).

\bigskip{}

We now focus on the case where $\delta=1$. We start by writing (\ref{Definition Shifted Epstein})
as a Mellin transform
\begin{equation}
\pi^{-s}\Gamma(s)\,\varphi_{k,1}(s;a_{1},...,a_{k})=\intop_{0}^{\infty}x^{s-1}\prod_{j=1}^{k}\sum_{n_{j}\in\mathbb{Z}}\left(n_{j}+a_{j}\right)\,e^{-\pi\left(n_{j}+a_{j}\right)^{2}x}\,dx,\,\,\,\,\,\text{Re}(s)>k, \label{Mellin integral}
\end{equation}
so that the proof of (\ref{functional equation auxiliar function Dirichlet}) relies on applying Poisson's summation formula to
each factor. This is standard and we obtain for each $j\in\{1,...,k\}$,
\begin{equation}
\sum_{n_{j}\in\mathbb{Z}}(n_{j}+a_{j})\,e^{-\pi(n_{j}+a_{j})^{2}x}=-\frac{i}{x^{3/2}}\,\sum_{n_{j}\in\mathbb{Z}}n_{j}\,\exp\left(-\frac{\pi}{x}n_{j}^{2}+2\pi i\,n_{j}\,a_{j}\right).\label{Poisson summation formula shifted}
\end{equation}

Hence, breaking the integral on (\ref{Mellin integral}) into $(0,1)$
and $(1,\infty)$ and applying (\ref{Poisson summation formula shifted})
on the first integral, we get the representation:
\begin{align*}
\pi^{-s}\Gamma(s)\,\varphi_{k,1}(s;a_{1},...,a_{k}) & =\intop_{1}^{\infty}x^{s-1}\sum_{n_{1},...,n_{k}\in\mathbb{Z}}\left(n_{1}+a_{1}\right)\cdot...\cdot\left(n_{k}+a_{k}\right)\,\exp\left(-\pi\left(\left(n_{1}+a_{1}\right)^{2}+...+\left(n_{k}+a_{k}\right)^{2}\right)x\right)\,dx\\
+\intop_{1}^{\infty}x^{\frac{3}{2}k-s-1}\,(-i)^{k} & \sum_{n_{1},...,n_{k}\in \mathbb{Z}}n_{1}\cdot...\cdot n_{k}\cdot\exp\left(-\pi x\left(n_{1}^{2}+...+n_{k}^{2}\right)+2\pi i\,\left(n_{1},...,n_{k}\right)\cdot\left(a_{1},...,a_{k}\right)\right)\,dx,
\end{align*}
which is invariant once we replace $s$ by $\frac{3k}{2}-s$ and $\varphi_{k,1}(s;a_{1},...,a_{k})$
by $\tilde{\varphi}_{k,1}\left(s;\,a_{1},...,a_{k}\right)$.

\bigskip{}

Moreover, the uniform and absolute convergence of the integrals given in the equality
above assure that $\varphi_{k,1}(s;a_{1},...,a_{k})$ and $\tilde{\varphi}_{k,1}\left(s;\,a_{1},...,a_{k}\right)$
are entire functions of $s$. 

\end{proof}

The next lemma has a role similar to Lemmas \ref{functional equation tilde and star} and \ref{Lemma Exponential sum Callahan} of the previous sections. 

\begin{lemma}\label{lemma exponential sum character}
Let $\chi$ be a Dirichlet character modulo $q$ and $p$ be an integer
such that $(p,q)=1$. If $\chi$ is even, define $\delta=0$. Otherwise,
if $\chi$ is odd, set $\delta=1$.

\bigskip{}

Consider the Dirichlet series
\begin{equation}
L_{k}(s,\chi,p/q):=\sum_{n_{1},...,n_{k}\neq0}\frac{n_{1}^{\delta}\,\chi(n_{1})\,...\,n_{k}^{\delta}\,\chi(n_{k})}{\left(n_{1}^{2}+...+n_{k}^{2}\right)^{s}}\,e^{-\frac{i\pi p}{q}\left(n_{1}^{2}+...+n_{k}^{2}\right)},\,\,\,\,\text{Re}(s)>\frac{k}{2}(1+\delta).\label{periodic k squares Characters}
\end{equation}

\bigskip{}

Then $L_{k}\left(s,\chi,p/q\right)$ has the following properties:
\begin{enumerate}
\item If $\chi$ is even, it can be analytically continued as a meromorphic
function with at most one simple pole located at $s=k/2$ with residue
given by 
\begin{equation}
\text{Res}_{s=k/2}L_{k}(s,\chi,p/q)=\frac{\pi^{k/2}}{\Gamma(k/2)(2q)^{k}}\left(\sum_{r=1}^{2q-1}\chi(r)\,e^{-\frac{\pi ip}{q}r^{2}}\right)^{k}.\label{residue explicit representation}
\end{equation}
In particular, if $p,q$ are odd integers, $L_{k}(s,\chi,p/q)$
is entire.
\item If $\chi$ is odd, it can be analytically continued as an entire function.
\end{enumerate}
Moreover, it
satisfies the functional equation
\begin{equation}
\left(\frac{\pi}{2q}\right)^{-s}\Gamma(s)\,L_{k}(s,\chi,p/q)=(-i)^{\delta k}\,\left(\frac{\pi}{2q}\right)^{-\left(k\left(\frac{1}{2}+\delta\right)-s\right)}\,\Gamma\left(k\left(\frac{1}{2}+\delta\right)-s\right)\,\tilde{L}_{k}\left(k\left(\frac{1}{2}+\delta\right)-s,\chi,p/q\right),\label{Functional equation character Epstein}
\end{equation}
where $\tilde{L}_{k}\left(s,\chi,p/q\right)$ is representable by the
series
\begin{equation}
\tilde{L}_{k}\left(s,\chi,p/q\right)=\sum_{n_{1},...,n_{k}\neq0}\frac{b_{\chi}(n_{1},...,n_{k};\,p/q)\,\cdot\,n_{1}^{\delta}\cdot...\cdot n_{k}^{\delta}}{\left(n_{1}^{2}+...+n_{k}^{2}\right)^{s}},\,\,\,\,\text{Re}(s)>\frac{k}{2}(1+\delta),\label{auxliary Dirichlet series character}
\end{equation}
with
\begin{equation}
b_{\chi}(n_{1},...,n_{k};\,p/q):=\left(2q\right)^{-\frac{k}{2}}\,\sum_{r_{1},...,r_{k}=0}^{2q-1}\chi(r_{1})\cdot...\cdot\chi(r_{k})\,\exp\left(-\frac{i\pi p}{q}\left(r_{1}^{2}+...+r_{k}^{2}\right)+2\pi i\,\left(n_{1},...,n_{k}\right)\cdot\left(\frac{r_{1}}{2q},...,\frac{r_{k}}{2q}\right)\right).\label{definition bchi}
\end{equation}

Furthermore,  (\ref{auxliary Dirichlet series character})
can be continued as an entire function no matter the
parity of $\chi$.
    
\end{lemma}

\begin{proof}
Since the arithmetical function $\chi(n)\,e^{-\frac{i\pi p}{q}n^{2}}$
has period $2q$, we may decompose the series (\ref{periodic k squares Characters})
into residue classes modulo $2q$, which gives:
\begin{align}
L_{k}(s,\chi,p/q) & :=\sum_{n_{1},...,n_{k}\neq0}\frac{n_{1}^{\delta}\,\chi(n_{1})\,...\,n_{k}^{\delta}\,\chi(n_{k})}{\left(n_{1}^{2}+...+n_{k}^{2}\right)^{s}}\,e^{-\frac{i\pi p}{q}\left(n_{1}^{2}+...+n_{k}^{2}\right)}=\nonumber \\
=(2q)^{k\delta-2s}\,\sum_{r_{1},...,r_{k}=1}^{2q-1} & \chi(r_{1})\cdot...\cdot\chi(r_{k})\,e^{-\frac{i\pi p}{q}\left(r_{1}^{2}+...+r_{k}^{2}\right)}\sum_{m_{1},...,m_{k}\in\mathbb{Z}}\,\frac{\left(m_{1}+\frac{r_{1}}{2q}\right)^{\delta}\cdot...\cdot\left(m_{k}+\frac{r_{k}}{2q}\right)^{\delta}}{\left(\left(m_{1}+\frac{r_{1}}{2q}\right)^{2}+...+\left(m_{k}+\frac{r_{k}}{2q}\right)^{2}\right)^{s}}\nonumber \\
=(2q)^{k\delta-2s} & \,\sum_{r_{1},...,r_{k}=1}^{2q-1}\,\chi(r_{1})\cdot...\cdot\chi(r_{k})\,e^{-\frac{i\pi p}{q}\left(r_{1}^{2}+...+r_{k}^{2}\right)}\,\varphi_{k,\delta}\left(s;\frac{r_{1}}{2q},...,\frac{r_{k}}{2q}\right),\,\,\,\,\text{Re}(s)>\frac{k}{2}(1+\delta),\label{writing combinations of phi}
\end{align}
where $\varphi_{k,\delta}(s;\,a_{1},...,a_{k})$ is given by (\ref{Definition Shifted Epstein}).
We now apply the previous lemma: if $\delta=1$, i.e., if $\chi$
is an odd Dirichlet character, we have that $L_{k}(s,\chi,p/q)$ is
entire because $\varphi_{k,1}(s;\,a_{1},...,a_{k})$ is entire.

\bigskip{}

On the other hand, if $\chi$ is even, then $L_{k}(s,\chi,p/q)$ must
have a simple pole at $s=k/2$ with residue given explicitly given
by (\ref{residue explicit representation}), due to item 1. of Lemma \ref{functional equation shifted Epstein}. 

\bigskip{}

In particular, if $\chi$ is even and $p,q$ are odd integers,
\begin{align*}
\sum_{r_{1},...,r_{k}=1}^{2q-1}\chi(r_{1})\cdot...\cdot\chi(r_{k})\,e^{-\frac{i\pi p}{q}\left(r_{1}^{2}+...+r_{k}^{2}\right)} & =\left(\sum_{r=1}^{2q-1}\chi(r)\,e^{-\frac{\pi ip}{q}r^{2}}\right)^{k}\\
=\left(\sum_{r=1}^{q-1}\chi(r)\,e^{-\frac{\pi ip}{q}r^{2}}+\sum_{r=q+1}^{2q-1}\chi(r)\,e^{-\frac{\pi ip}{q}r^{2}}\right)^{k} & =\left(\sum_{r=1}^{q-1}\chi(r)\,e^{-\frac{\pi ip}{q}r^{2}}-\sum_{r=1}^{q-1}\chi(r)\,e^{-\frac{\pi ip}{q}r^{2}}\right)^{k}=0,
\end{align*}
which shows that, under these conditions, $L_{k}(s,\chi,p/q)$ must
be entire.

\bigskip{}

The functional equation (\ref{Functional equation character Epstein})
now follows from (\ref{Dual Dirichlet series}): invoking (\ref{writing combinations of phi})
and (\ref{Dual Dirichlet series}), we see that
\begin{align*}
\left(\frac{\pi}{2q}\right)^{-s}\Gamma(s)\,L_{k}(s,\chi,p/q) & =\frac{(-i)^{\delta k}}{(2q)^{k/2}}\left(\frac{\pi}{2q}\right)^{-\left(k\left(\frac{1}{2}+\delta\right)-s\right)}\Gamma\left(k\left(\frac{1}{2}+\delta\right)-s\right)\times\\
\times\sum_{r_{1},...,r_{k}=1}^{2q-1}\,\chi(r_{1}) & \cdot...\cdot\chi(r_{k})\,\times e^{-\frac{i\pi p}{q}\left(r_{1}^{2}+...+r_{k}^{2}\right)}\,\tilde{\varphi}_{k,\delta}\left(k\left(\frac{1}{2}+\delta\right)-s;\,\frac{r_{1}}{2q},...,\frac{r_{k}}{2q}\right).
\end{align*}

But for $\text{Re}(s)>\frac{k}{2}(1+\delta)$,
\begin{align*}
\frac{1}{(2q)^{k/2}}\,\sum_{r_{1},...,r_{k}=1}^{2q-1}\,\chi(r_{1})\cdot...\cdot\chi(r_{k})&\,e^{-\frac{i\pi p}{q}\left(r_{1}^{2}+...+r_{k}^{2}\right)} \,\tilde{\varphi}_{k,\delta}\left(s;\,\frac{r_{1}}{2q},...,\frac{r_{k}}{2q}\right)=\\
=\sum_{n_{1},...,n_{k}\neq0}\frac{n_{1}^{\delta}\cdot...\cdot n_{k}^{\delta}\,}{\left(n_{1}^{2}+...+n_{k}^{2}\right)^{s}}\,\left(2q\right)^{-\frac{k}{2}}\,\sum_{r_{1},...,r_{k}=0}^{2q-1}\chi(r_{1})\cdot...\cdot\chi(r_{k})\, & \exp\left(-\frac{i\pi p}{q}\left(r_{1}^{2}+...+r_{k}^{2}\right)+2\pi i\,\left(n_{1},...,n_{k}\right)\cdot\left(\frac{r_{1}}{2q},...,\frac{r_{k}}{2q}\right)\right)\\
=\sum_{n_{1},...,n_{k}\neq0}\frac{b_{\chi}(n_{1},...,n_{k};\,p/q)\,n_{1}^{\delta}\cdot...\cdot n_{k}^{\delta}}{\left(n_{1}^{2}+...+n_{k}^{2}\right)^{s}}=\tilde{L}_{k}\left(s,\chi,p/q\right) & .
\end{align*}

\bigskip{}

To conclude, we just need to show that $\tilde{L}_{k}(s,\chi,p/q)$
is an entire function: this comes from the previous expression of
$\tilde{L}_{k}(s,\chi,p/q)$ as the sum

\[
\frac{1}{(2q)^{k/2}}\,\sum_{r_{1},...,r_{k}=1}^{2q-1}\,\chi(r_{1})\cdot...\cdot\chi(r_{k})\,e^{-\frac{i\pi p}{q}\left(r_{1}^{2}+...+r_{k}^{2}\right)}\,\tilde{\varphi}_{k,\delta}\left(s;\,\frac{r_{1}}{2q},...,\frac{r_{k}}{2q}\right).
\]

Since $0<r_{j}/(2q)<1$, we know that $\tilde{\varphi}_{k,\delta}\left(s;\,\frac{r_{1}}{2q},...,\frac{r_{k}}{2q}\right)$
is always entire (no matter if $\delta$ is 0 or 1) and so $\tilde{L}_{k}(s,\chi,p/q)$ must be entire as well.
    
\end{proof}

\begin{remark}
Note that we can write the multiple series defining (\ref{periodic k squares Characters})
and (\ref{auxliary Dirichlet series character}) as Dirichlet series
involving a single variable of summation. In fact, we may write (\ref{periodic k squares Characters})
as the Dirichlet series
\begin{equation}
L_{k}(s,\chi,p/q)=\sum_{n=1}^{\infty}\frac{r_{k,\chi}(n)\,e^{-\frac{i\pi p}{q}n}}{n^{s}},\,\,\,\,\,\text{Re}(s)>\frac{k}{2}\left(1+\delta\right),\label{writing as single dirichlet}
\end{equation}
where $r_{k,\chi}(n)$ is explicitly given by (\ref{general rk(n)}).
We can also write (\ref{auxliary Dirichlet series character}) as
\begin{equation}
\tilde{L}_{k}\left(s,\chi,p/q\right)=\sum_{n=1}^{\infty}\frac{\tilde{b}_{k,\chi}(n,\,p/q)\,}{n^{s}},\,\,\,\,\text{Re}(s)>\frac{k}{2}(1+\delta),\label{auxliary Dirichlet series character-1}
\end{equation}
with $\tilde{b}_{k,\chi}(n,p/q)$ being
\begin{equation}
\left(2q\right)^{-\frac{k}{2}}\sum_{n_{1}^{2}+...+n_{k}^{2}=n}\,n_{1}^{\delta}\cdot...\cdot n_{k}^{\delta}\,\sum_{r_{1},...,r_{k}=0}^{2q-1}\chi(r_{1})\cdot...\cdot\chi(r_{k})\,\exp\left(-\frac{i\pi p}{q}\left(r_{1}^{2}+...+r_{k}^{2}\right)+2\pi i\,\left(n_{1},...,n_{k}\right)\cdot\left(\frac{r_{1}}{2q},...,\frac{r_{k}}{2q}\right)\right).\label{definition bchi-1}
\end{equation}
   
\end{remark}

\begin{remark}
Note that we did not use the primitivity of the character $\chi$ in the proof of Lemma \ref{lemma exponential sum character}. 

In the next subsection we will see that, under this additional
hypothesis and by setting $p=0$ in (\ref{Functional equation character Epstein}),
we can recover the functional equation (\ref{sum of definition k squares L function})
given at the beginning of the paper.    
\end{remark}

\begin{center}\subsection{Proof of Lemma \ref{functional equation L_k}} \label{proof of lemma 1.3.}\end{center}
Since $\chi$ is a primitive character, we know that the Gauss sum
splits in the form

\begin{equation}
G(n,\overline{\chi}):=\sum_{j=1}^{q-1}\overline{\chi}(j)\,e^{2\pi inj/q}=\chi(n)\,G(\overline{\chi}),\label{primitivity condition}
\end{equation}
where $G(\chi):=G(1,\chi)$.

By the representation as Dirichlet series (\ref{periodic k squares Characters})
and (\ref{sum of definition k squares L function}), we clearly have
that $L_{k}\left(s,\chi\right)=L_{k}(s,\chi,0)$. Since (\ref{residue explicit representation})
is zero when $p=0$, we conclude that $L_{k}\left(s,\chi\right)$
is entire. To prove the functional equation (\ref{Functional equation Lk (s, chi)}),
we look at (\ref{Functional equation character Epstein}) and see
\begin{equation}
\left(\frac{\pi}{2q}\right)^{-s}\Gamma(s)\,L_{k}(s,\chi)=(-i)^{\delta k}\,\left(\frac{\pi}{2q}\right)^{-\left(k\left(\frac{1}{2}+\delta\right)-s\right)}\,\Gamma\left(k\left(\frac{1}{2}+\delta\right)-s\right)\,\tilde{L}_{k}\left(k\left(\frac{1}{2}+\delta\right)-s,\chi,0\right).\label{particular case fffunctional equation Lk at zero}
\end{equation}

Thus, it remains to simplify the expression for $\tilde{L}_{k}(s,\chi,p/q)$,
(\ref{auxliary Dirichlet series character}), when $p=0$: by (\ref{definition bchi}), we know that 
\begin{equation}
\tilde{L}_{k}\left(s,\chi,0\right)=(2q)^{-\frac{k}{2}}\,\sum_{n_{1},...,n_{k}\neq0}\sum_{r_{1},...,r_{k}=0}^{2q-1}\chi(r_{1})\cdot...\cdot\chi(r_{k})\,\exp\left(2\pi i\,\left(n_{1},...,n_{k}\right)\cdot\left(\frac{r_{1}}{2q},...,\frac{r_{k}}{2q}\right)\right)\,\frac{\,n_{1}^{\delta}\cdot...\cdot n_{k}^{\delta}}{\left(n_{1}^{2}+...+n_{k}^{2}\right)^{s}},\label{sum over lalala}
\end{equation}
for $\text{Re}(s)>\frac{k}{2}(1+\delta)$. Now, for every pair
$(n_{j},r_{j})$, the resulting Gauss sum is:
\begin{align}
\sum_{r_{j}=0}^{2q-1}\chi(r_{j})\,\exp\left(2\pi i\,n_{j}\,\frac{r_{j}}{2q}\right) & =\sum_{r_{j}=1}^{q-1}\chi(r_{j})\,\exp\left(2\pi i\,n_{j}\,\frac{r_{j}}{2q}\right)+\sum_{j=1}^{q-1}\chi(r_{j})\,\exp\left(2\pi i\,n_{j}\,\frac{r_{j}+q}{2q}\right)\nonumber \\
=\left(1+(-1)^{n_{j}}\right)\,\sum_{r_{j}=1}^{q-1}\chi(r_{j})\, & \exp\left(2\pi i\,n_{j}\,\frac{r_{j}}{2q}\right)=\begin{cases}
2\,\sum_{r_{j}=1}^{q-1}\chi(r_{j})\,\exp\left(2\pi i\,m_{j}\,\frac{r_{j}}{q}\right)=2\,G\left(m_{j},\chi\right) & n_{j}=2m_{j}\\
0 & n_{j}\,\,\,\text{odd}.
\end{cases}\label{evaluation explicit Gauss summm}
\end{align}

Therefore, the sum in (\ref{sum over lalala}) over $(n_{1},...,n_{k})$
does not vanish only if each $n_{j}$ is an even integer: writing $n_{j}=2m_{j}$,
we derive from (\ref{sum over lalala}) and (\ref{evaluation explicit Gauss summm}), 
\begin{align}
\tilde{L}_{k}\left(s,\chi,0\right) & =(2q)^{-\frac{k}{2}}2^{-2s+k(\delta+1)}\,\sum_{m_{1},...,m_{k}\neq0}\frac{G\left(m_{1},\chi\right)m_{1}^{\delta}\cdot...\cdot G\left(m_{k},\chi\right)m_{k}^{\delta}}{\left(m_{1}^{2}+...+m_{k}^{2}\right)^{s}}\nonumber \\
 & =\frac{G^{k}(\chi)}{q^{k/2}}\,2^{-2s+k\delta+\frac{k}{2}}\,\sum_{m_{1},...,m_{k}\neq0}\frac{\overline{\chi}(m_{1})m_{1}^{\delta}\cdot...\cdot\overline{\chi}(m_{k})m_{k}^{\delta}}{\left(m_{1}^{2}+...+m_{k}^{2}\right)^{s}}\nonumber \\
 & =\frac{G^{k}(\chi)}{q^{k/2}}\,2^{-2s+k\delta+\frac{k}{2}}\,L_{k}\left(s,\,\overline{\chi}\right),\label{final simpliiiification}
\end{align}
where in the second step we have used the fact that $\chi$ is primitive
(\ref{primitivity condition}). Joining (\ref{final simpliiiification})
and (\ref{particular case fffunctional equation Lk at zero}) gives
(\ref{Functional equation Lk (s, chi)}). $\blacksquare$

\bigskip{}

 \subsection{The behavior of $\psi_{\chi,k}(x,z)$}

We finally prove that the analogue of Jacobi's $\psi-$function given at (\ref{analogues character Jacobi psi!}) has the same properties as the previous ones. We start with a summation formula involving $r_{k,\chi}(n)$. 

\begin{lemma}\label{Summation formula characters Epstein}
 Let $\chi$ be a primitive Dirichlet character modulo $q$ and let
$r_{k,\chi}(n)$ and $\tilde{b}_{k,\chi}(n,\,p/q)$ be the arithmetical
functions given in (\ref{general rk(n)}) and (\ref{definition bchi-1}).

\bigskip{}

Assume further that $q\neq0 \mod4$ when $\chi$ is even.

\bigskip{}

Then, for every $\text{Re}(x)>0$ and $y\in\mathbb{C}$, the following
transformation formula takes place
\begin{align}
\sum_{n=1}^{\infty}r_{k,\chi}(n)\,e^{-\frac{i\pi}{q}n}\,n^{\frac{1}{2}-\frac{k}{2}\left(\frac{1}{2}+\delta\right)}\,e^{-\frac{\pi}{2q}n\,x}\,J_{k\left(\frac{1}{2}+\delta\right)-1}\left(\sqrt{\frac{\pi n}{2q}}\,y\right)\nonumber \\
=\frac{e^{-\frac{y^{2}}{4x}}}{x}\,\sum_{n=1}^{\infty}(-i)^{\delta k}\tilde{b}_{k,\chi}\left(n,\,1/q\right)\,n^{\frac{1}{2}-\frac{k}{2}\left(\frac{1}{2}+\delta\right)}\,e^{-\frac{\pi n}{2q\,x}}\,I_{k\left(\frac{1}{2}+\delta\right)-1}\left(\sqrt{\frac{\pi n}{2q}}\,\frac{y}{x}\right).\label{summation formuuula characccttteeers}
\end{align}   
\end{lemma}

\begin{proof}
 Apply the summation formula (\ref{final formula for 1f1 theorem})
or (\ref{generalized Theta reflection}) to the pair of Dirichlet
series (which satisfy Hecke's functional equation and belong to the
class $\mathcal{A}$ by Lemma \ref{lemma exponential sum character}): \begin{equation}
\phi(s)=\left(\frac{\pi}{2q}\right)^{-s}\,L_{k}\left(s,\chi,1/q\right):=\left(\frac{\pi}{2q}\right)^{-s}\,\sum_{n=1}^{\infty}\frac{r_{k,\chi}(n)\,e^{-\frac{i\pi}{q}n}}{n^{s}},\,\,\,\,\text{Re}(s)>\frac{k}{2}(1+\delta)\label{first pppair}
\end{equation}
and
\begin{equation}
\psi(s)=\left(\frac{\pi}{2q}\right)^{-s}\,(-i)^{\delta k}\,\tilde{L}_{k}\left(s,\chi,1/q\right)=\left(\frac{\pi}{2q}\right)^{-s}\,\sum_{n=1}^{\infty}\frac{(-i)^{\delta k}\,\tilde{b}_{k,\chi}(n,\,1/q)\,}{n^{s}},\,\,\,\,\text{Re}(s)>\frac{k}{2}(1+\delta).\label{seeecond peeer}
\end{equation}

This requires to take the simple substitutions $\lambda_{n}=\frac{\pi n}{2q}$,
$\mu_{n}=\frac{\pi n}{2q},$ $a(n)=r_{k,\chi}(n)\,e^{-\frac{i\pi p}{q}n}$,
$b(n)=(-i)^{\delta k}\tilde{b}_{k,\chi}(n,1/q)$ and $r=k\left(\frac{1}{2}+\delta\right)$
in (\ref{final formula for 1f1 theorem}). 

\bigskip{}

We first study the possible poles that $\phi(s)$ may have: if $\chi$
is odd, it follows lemma \ref{lemma exponential sum character} that $\phi(s)$ is entire.

\bigskip{}

Assume now that $\chi$ is even. Since there are no primitive characters
(mod $q$) such that $q\equiv2 \mod4$, our imposition that $q\neq0$
mod 4 when $\chi$ is even actually implies that $q$ must be odd.
By item 1. of Lemma \ref{lemma exponential sum character}, we know that $\phi(s)$ must be an entire function
under this assumption.

\bigskip{}

Thus, $\rho=0$ regardless of $\chi$ being even or odd. In order
to apply (\ref{final formula for 1f1 theorem}), we need to find also $\phi(0)$.
Since $\phi(s)\in\mathcal{A}$ and $\psi(s)$ given by (\ref{first pppair})
is an entire function, we know by Remark \ref{properties class A} that $\phi(0)$ must
be zero.

\bigskip{}

Finally, replacing $\rho=\phi(0)=0$ in (\ref{final formula for 1f1 theorem}),
we are able to derive (\ref{summation formuuula characccttteeers}).
\end{proof}

\bigskip{}

We now prove the analogue of Lemmas \ref{Vanishing Theta lemma} and \ref{vanishing theta Quadratic} of the previous sections.  

\begin{lemma} \label{vanishing theta character}
Let $\chi$ be a primitive Dirichlet character modulo $q$ and $r_{k,\chi}(n)$
be defined by (\ref{general rk(n)}). For $\text{Re}(x)>0$ and $z\in\mathbb{C}$,
consider the analogue of Jacobi's $\psi-$function (\ref{analogues character Jacobi psi!}):
\[
\psi_{\chi,k}(x,z):=2^{\frac{k}{2}+k\delta-1}\,\Gamma\left(\frac{k}{2}+k\delta\right)\left(\sqrt{\frac{\pi}{q}x}\,z\right)^{1-\frac{k}{2}-k\delta}\,\sum_{n=1}^{\infty}r_{k,\chi}(n)\,n^{\frac{1}{2}-\frac{k}{4}-\frac{k\delta}{2}}\,e^{-\frac{\pi n}{q}\,x}\,J_{\frac{k}{2}+k\delta-1}\left(\sqrt{\frac{\pi}{q}n\,x}\,z\right).
\]

\bigskip{}

Furthermore, let $h:\,\mathbb{C}\longmapsto\mathbb{C}$ be an analytic
function and assume that $z$ satisfies the condition:

\begin{equation}
z\in\mathscr{D}_{\chi}:=\left\{ z\in\mathbb{C}\,:\,|\text{Re}(z)|<\sqrt{\frac{\pi}{2q}},\,\,|\text{Im}(z)|<\sqrt{\frac{\pi}{2q}}\right\} .\label{domain characcccter here}
\end{equation}

\bigskip{}

Then:
\begin{enumerate}
\item if $\chi$ is even and $q\neq0 \mod4$ or
\item if $\chi$ is odd with arbitrary modulus,
\end{enumerate}
the asymptotic behavior takes place:
\begin{equation}
\lim_{\omega\rightarrow\frac{\pi}{4}^{-}}\,\frac{d^{m}}{d\omega^{m}}\left\{ h(\omega)\,\psi_{\chi,k}\left(e^{2i\omega},\,z\right)\right\} =0\,\,\,\,\forall m\in\mathbb{N}_{0}.\label{result for pssssi character}
\end{equation}
  
\end{lemma}

\begin{proof}
Following the proofs of Lemmas \ref{Vanishing Theta lemma} and \ref{vanishing with analytic} (in particular, see (\ref{Faa di Bruno})
and (\ref{application Leibniz})) (and changing the $\delta-$notation
there to $\epsilon$), we just need to consider the behavior of $\psi_{\chi,k}(i+\epsilon,z)$ as $\epsilon\rightarrow 0^{+}$. From (\ref{analogues character Jacobi psi!}), the expression of $\psi_{\chi,k}(i+\epsilon,z)$ is  
\begin{align}
&2^{\frac{k}{2}+k\delta-1}\,\Gamma\left(\frac{k}{2}+k\delta\right)\left(\sqrt{\frac{\pi}{q}(i+\epsilon)}\,z\right)^{1-\frac{k}{2}-k\delta}\,\sum_{n=1}^{\infty}r_{k,\chi}(n)\,n^{\frac{1}{2}-\frac{k}{4}-\frac{k\delta}{2}}\,e^{-\frac{\pi n}{q}\,(i+\epsilon)}\,J_{\frac{k}{2}+k\delta-1}\left(\sqrt{\frac{\pi}{q}n(i+\epsilon)}\,z\right)\nonumber\\
&=2^{\frac{k}{2}+k\delta-1}\,\Gamma\left(\frac{k}{2}+k\delta\right) \left(\sqrt{\frac{\pi}{q}(i+\epsilon)}\,z\right)^{1-\frac{k}{2}-k\delta}\,\sum_{n=1}^{\infty}r_{k,\chi}(n)\,n^{\frac{1}{2}-\frac{k}{4}-\frac{k\delta}{2}}\,e^{-\frac{\pi n}{q}i}\,e^{-\frac{\pi n}{q}\,\epsilon}\,J_{\frac{k}{2}+k\delta-1}\left(\sqrt{\frac{\pi}{q}n(i+\epsilon)}\,z\right)\label{psi chi i}.
\end{align}

We may apply the summation formula (\ref{summation formuuula characccttteeers}) to the right-hand side of (\ref{psi chi i}) with  $x$ in (\ref{summation formuuula characccttteeers}) being replaced by $2\epsilon$ and  $y$ by $\sqrt{2(i+\epsilon)}\,z$. This gives the transformation: 
\begin{align}
\psi_{\chi,k}(i+\epsilon,z)=&2^{\frac{k}{2}+k\delta-1}\,\Gamma\left(\frac{k}{2}+k\delta\right)\left(\sqrt{\frac{\pi}{q}(i+\epsilon)}\,z\right)^{1-\frac{k}{2}-k\delta}\,\frac{e^{-\frac{(i+\epsilon)z^{2}}{4\epsilon}}}{2\epsilon}\times \nonumber\\
\times & \,\sum_{n=1}^{\infty}(-i)^{\delta k}\tilde{b}_{k,\chi}\left(n,\,1/q\right)\,n^{\frac{1}{2}-\frac{k}{2}\left(\frac{1}{2}+\delta\right)}\,e^{-\frac{\pi n}{4q\epsilon}}\,I_{k\left(\frac{1}{2}+\delta\right)-1}\left(\sqrt{\frac{\pi n(i+\epsilon)}{q}}\,\frac{z}{2\epsilon}\right).\label{formula once more for characters k}
\end{align}

If we bound $I_{\nu}(z)$ in the same way as before (see 
(\ref{Bound Bessel})), a sufficient condition ensuring that (\ref{formula once more for characters k})
tends to zero as $\epsilon\rightarrow0^{+}$ is
\[
\lim_{\epsilon\rightarrow0^{+}}\epsilon^{-\frac{k}{2}-k\delta}\,\exp\left(-\frac{\pi n}{4q\epsilon}+\frac{1}{2\epsilon}\sqrt{\frac{\pi n}{q}}\,\frac{|\text{Re}(z)-\text{Im}(z)|}{\sqrt{2}}+\frac{\text{Re}(z)\,\text{Im}(z)}{2\epsilon}\right)=0.
\]

But by examining the quadratic polynomial (with variable $\sqrt{n}$),
\[
P(X)=X^{2}-X\,\sqrt{\frac{2q}{\pi}}\,|\text{Re}(z)-\text{Im}(z)|-\frac{2q\,\text{Re}(z)\,\text{Im}(z)}{\pi},
\]
one can see that $P(\sqrt{n})>0$ for every $n\in\mathbb{N}$ if:
\begin{equation}
|\text{Re}(z)-\text{Im}(z)|+|\text{Re}(z)+\text{Im}(z)|<\sqrt{\frac{2\pi}{q}},\label{condition first before statement character case}
\end{equation}
which is satisfied whenever $z$ satisfies the condition (\ref{domain characcccter here}). 
\end{proof}

\begin{remark}
In contrast with the statement of Theorem \ref{zeta alpha hypergeometric zeros}, the domain in (\ref{domain characcccter here}) does not depend
on $k$. Like in Remark \ref{discriminant condition domain}, enlarging this domain requires to
get more information about the arithmetical function $\tilde{b}_{k,\chi}\left(n,\,1/q\right)$.
As in (\ref{better condition quadratic form}), (\ref{condition first before statement character case})
can be replaced by:
\[
|\text{Re}(z)-\text{Im}(z)|+|\text{Re}(z)+\text{Im}(z)|<\sqrt{\frac{2\pi\,m_{\chi,k}}{q}},
\]
where $m_{\chi,k}$ is the least integer for which $\tilde{b}_{k,\chi}\left(n,\,1/q\right)\neq0$.    
\end{remark}

\begin{center}

 \subsection{Proof of Theorem \ref{Dirichlet L functions result}} \label{section 5}
   
\end{center}
Since the Dirichlet series $L_{k}(s,\chi)$ is entire, the argument in the proof of Theorem \ref{zeta alpha hypergeometric zeros} becomes easier, because we do not need to count the sign changes of the residual terms like in the expression (\ref{first template}). 

We shall see how the proof in this case goes: start with the integral representation (\ref{great identity for charackter k})
obtained in Example \ref{character case example} and replace there $x$ by $e^{2i\omega}$,
$-\frac{\pi}{4}<\omega<\frac{\pi}{4}$:
\[
\frac{1}{2\pi}\,\intop_{-\infty}^{\infty}\eta_{k}\left(\frac{k}{4}+\frac{k\delta}{2}+it,\chi\right)\,{}_{1}F_{1}\left(\frac{k}{4}+\frac{k\delta}{2}+it;\,\frac{k}{2}+k\delta;\,-\frac{z^{2}}{4}\right)e^{2\omega t}\,dt=e^{i\omega k\left(\frac{1}{2}+\delta\right)}\,\psi_{\chi,k}\left(e^{2i\omega},z\right).
\]

From Kummer's formula, we can rewrite the previous equality in the
symmetric form:
\begin{equation}
\frac{1}{2\pi}\,\intop_{-\infty}^{\infty}\eta_{k}\left(\frac{k}{4}+\frac{k\delta}{2}+it,\chi\right)\,{}_{1}F_{1}\left(\frac{k}{4}+\frac{k\delta}{2}-it;\,\frac{k}{2}+k\delta;\,\frac{z^{2}}{4}\right)e^{2\omega t}\,dt=e^{i\omega k\left(\frac{1}{2}+\delta\right)}\,e^{z^{2}/4}\,\psi_{\chi,k}\left(e^{2i\omega},z\right),\label{after kummmer in character case}
\end{equation}
which is somewhat remindful of the previous steps given in (\ref{after kummer in the whole proof})
and (\ref{arranging symmetric form in quadratic}). In the present
case, however, we are in a much easier situation because there are
no residual terms on the right-hand side of (\ref{after kummmer in character case}).

\bigskip{}

After changing the variable $t\rightarrow t+\lambda_{j}$ and differentiate
over $\omega$ a number $p$ of times, we are able to get:
\begin{align}
\intop_{-\infty}^{\infty}t^{p}\eta_{k}\left(\frac{k}{4}+\frac{k\delta}{2}+i\,(t+\lambda_{j}),\chi\right)\,{}_{1}F_{1}\left(\frac{k}{4}+\frac{k\delta}{2}-i(t+\lambda_{j});\,\frac{k}{2}+k\delta;\,\frac{z^{2}}{4}\right)e^{2\omega t}\,dt\nonumber \\
=\frac{2\pi}{2^{p}}\,e^{z^{2}/4}\,\frac{d^{p}}{d\omega^{p}}\left\{ e^{i\omega k\left(\frac{1}{2}+\delta\right)-2\omega\lambda_{j}}\,\psi_{\chi,k}\left(e^{2i\omega},z\right)\right\} .\label{after change variables}
\end{align}

\bigskip{}

\bigskip{}

Now, we make an important remark: in order to apply Hardy's method,
we need the integrand on the left of (\ref{after change variables})
to be a real function of $t$.

\bigskip{}

In the previous two cases, it was enough to take the real part of
the hypergeometric factor to assure this. In this case, however, we
note that the function
\[
\eta_{k}\left(\frac{k}{4}+\frac{k\delta}{2}+it,\chi\right)
\]
is not real-valued. This can be easily seen from the asymmetry in
the functional equation (\ref{Functional equation Lk (s, chi)}),
\begin{equation}
\eta_{k}\left(\frac{k}{4}+\frac{k\delta}{2}+it,\chi\right)=\left(\frac{(-i)^{\delta}G(\chi)}{\sqrt{q}}\right)^{k}\,\eta_{k}\left(\frac{k}{4}+\frac{k\delta}{2}-it,\overline{\chi}\right).\label{functional critical}
\end{equation}

However, since $|G(\chi)/\sqrt{q}|=1$, we can write $(-i)^{\delta}\,G(\chi)/\sqrt{q}:=e^{i\gamma}$
for some real $\gamma$ and then set:
\begin{equation}
\tilde{\eta}_{k}\left(\frac{k}{4}+\frac{k\delta}{2}+it,\chi\right):=e^{-i\gamma k/2}\,\eta_{k}\left(\frac{k}{4}+\frac{k\delta}{2}+it,\chi\right).\label{new Dirichlet series after gamma}
\end{equation}

\bigskip{}

It is now clear from (\ref{functional critical}) that $\tilde{\eta}_{k}\left(\frac{k}{4}+\frac{k\delta}{2}+it,\chi\right)$
is real-valued and its zeros are exactly the same as the ones of $\eta_{k}\left(\frac{k}{4}+\frac{k\delta}{2}+it,\chi\right)$.
Thus, we may multiply both sides of (\ref{after change variables})
by $e^{-i\gamma k/2}$ and have:
\begin{align*}
\intop_{-\infty}^{\infty}t^{p}\,\tilde{\eta}_{k}\left(\frac{k}{4}+\frac{k\delta}{2}+i\,(t+\lambda_{j}),\chi\right)\,{}_{1}F_{1}\left(\frac{k}{4}+\frac{k\delta}{2}-i(t+\lambda_{j});\,\frac{k}{2}+k\delta;\,\frac{z^{2}}{4}\right)e^{2\omega t}\,dt\\
=\frac{2\pi e^{-i\gamma k/2}}{2^{p}}\,e^{z^{2}/4}\,\frac{d^{p}}{d\omega^{p}}\left\{ e^{i\omega k\left(\frac{1}{2}+\delta\right)-2\omega\lambda_{j}}\,\psi_{\chi,k}\left(e^{2i\omega},z\right)\right\} .
\end{align*}

\bigskip{}

We take the real part of the hypergeometric factor, multiply by $c_{j}\in\ell^{1}$
and sum over $j$: with the same kind of justifications as in the
proof of Theorem \ref{zeta alpha hypergeometric zeros}, we can deduce
\begin{equation}
\intop_{-\infty}^{\infty}t^{p}\,\tilde{F}_{z,k,\chi}\left(\frac{k}{4}+\frac{k\delta}{2}+it\right)\,e^{2\omega t}\,dt
=\frac{4\pi}{2^{p}}\,\text{Re}\,\left[e^{-i\gamma k/2+z^{2}/4}\,\frac{d^{p}}{d\omega^{p}}\left\{ \sum_{j=1}^{\infty}c_{j}\,e^{i\omega k\left(\frac{1}{2}+\delta\right)-2\omega\lambda_{j}}\cdot\psi_{\chi,k}\left(e^{2i\omega},z\right)\right\} \right],\label{infinite summ justification character}
\end{equation}
where
\[
\tilde{F}_{z,k,\chi}(s):=e^{-i\gamma k/2}\,F_{z,k,\chi}(s),
\]
with $F_{z,k,\chi}(s)$ being
\[
\sum_{j=1}^{\infty}c_{j}\,\eta_{k}\left(s+i\lambda_{j},\chi\right)\,\left\{ _{1}F_{1}\left(k\left(\frac{1}{2}+\delta\right)-s-i\lambda_{j};\,k\left(\frac{1}{2}+\delta\right);\,\frac{z^{2}}{4}\right)+{}_{1}F_{1}\left(k\left(\frac{1}{2}+\delta\right)-\overline{s}+i\lambda_{j};\,k\left(\frac{1}{2}+\delta\right);\,\frac{\overline{z}^{2}}{4}\right)\right\}.
\]

If we show that the real-valued function $\tilde{F}_{z,k,\chi}(s)$ has infinitely many zeros,
we are done. This can be proved similarly as in Theorem \ref{zeta alpha hypergeometric zeros}: taking the limit $\omega\rightarrow\frac{\pi}{4}^{-}$
on both sides and using (\ref{result for pssssi character}), we deduce
that

\begin{equation}
\lim_{\omega\rightarrow\frac{\pi}{4}^{-}}\intop_{-\infty}^{\infty}t^{p}\,\tilde{F}_{z,k,\chi}\left(\frac{k}{4}+\frac{k\delta}{2}+it\right)\,e^{2\omega t}\,dt=0.\label{limit in entire case}
\end{equation}

\bigskip{}

If we now assume that the integrand has only a finite amount of zeros,
we can use our previous reasoning (see (\ref{before dominated}),
(\ref{almost at dominated}) and (\ref{finishing dominated}) above)
to show that
\[
\intop_{-\infty}^{\infty}t^{p}\,\left|\tilde{F}_{z,k,\chi}\left(\frac{k}{4}+\frac{k\delta}{2}+i\,(t+\lambda_{j})\right)\right|\,e^{\frac{\pi}{2}\,t}\,dt<\infty
\]
and so, by (\ref{limit in entire case}) and the dominated convergence
theorem
\[
\intop_{-\infty}^{\infty}t^{p}\,\tilde{F}_{z,k,\chi}\left(\frac{k}{4}+\frac{k\delta}{2}+it\right)\,e^{\frac{\pi}{2}t}\,dt=0.
\]

Mimicking (\ref{upper bound}) and (\ref{lower bound}), we
find that 
\begin{align}
\intop_{T_{0}}^{\infty}\left\{ \tilde{F}_{z,k,\chi}\left(\frac{k}{4}+\frac{k\delta}{2}+it\right)\,e^{\frac{\pi}{2}t}+(-1)^{p}\,\tilde{F}_{z,k,\chi}\left(\frac{k}{4}+\frac{k\delta}{2}-it\right)\,e^{-\frac{\pi}{2}t}\right\} \,t^{p}dt\nonumber \\
=-\intop_{0}^{T_{0}}\left\{ \tilde{F}_{z,k,\chi}\left(\frac{k}{4}+\frac{k\delta}{2}+it\right)\,e^{\frac{\pi}{2}t}+(-1)^{p}\,\tilde{F}_{z,k,\chi}\left(\frac{k}{4}+\frac{k\delta}{2}-it\right)\,e^{-\frac{\pi}{2}t}\right\} \,t^{p}dt\leq & K(T_{0})\,T_{0}^{p}\label{upper bound T0 character}
\end{align}
where $T_{0}$ is a (sufficiently large) parameter such that  $\tilde{F}_{z,k,\chi}\left(\frac{k}{4}+\frac{k\delta}{2}+it\right)\neq 0$
when $t>T_{0}$. Literally as in (\ref{lower bound}), we may deduce
the lower bound:

\begin{equation}
\intop_{T_{0}}^{\infty}\left\{ \tilde{F}_{z,k,\chi}\left(\frac{k}{4}+\frac{k\delta}{2}+it\right)\,e^{\frac{\pi}{2}t}+(-1)^{p}\,\tilde{F}_{z,k,\chi}\left(\frac{k}{4}+\frac{k\delta}{2}-it\right)\,e^{-\frac{\pi}{2}t}\right\} \,t^{p}dt\geq\epsilon(T_{0})\,(2T_{0})^{p}.\label{lower bound contradiction hypo}
\end{equation}

\bigskip{}

\bigskip{}
\bigskip{}

\bigskip{}

Comparing (\ref{upper bound T0 character}) and (\ref{lower bound contradiction hypo})
we conclude that
\[
2^{p}\leq\frac{K(T_{0})}{\epsilon(T_{0})}
\]
must hold for infinitely many values of $p$, which is absurd. $\blacksquare$

\bigskip{}

\bigskip{}

\begin{remark}
As it is clear from the proof, we only
need the condition that the sequence $\left( \lambda_{j}\right) _{j\in\mathbb{N}}$
is bounded, 
but we have not used the fact that it 
attains its bounds. Thus, we may relax this latter condition in the statement of Theorem \ref{Dirichlet L functions result}. Similar comments can be made for Theorem \ref{cusp form theorem}  
\end{remark}

\begin{center}\section{Zeros of combinations attached to $L_{f}(s,p/q)$} \label{section 7}\end{center}

In this final section we extend Wilton's result by proving Theorem
\ref{cusp form theorem}. We will be very brief because, unlike in previous sections,
we do not need to develop ``exponential summation formulas'' for
the Dirichlet series under study. The basic Dirichlet series in this section is: 
\begin{equation}
L_{f}(s,p/q):=\sum_{n=1}^{\infty}\frac{a_{f}(n)\,e^{\frac{2\pi ip}{q}n}}{n^{s}},\,\,\,\,\,\text{Re}(s)>\frac{k+1}{2},\,\,\,\,(p,q)=1,\label{Cusp form Dirichlet series twisted exponential}
\end{equation}
where $a_{f}(n)$ are the Fourier coefficients of the cusp form $f(\tau)$ and are assumed to be real. 
From the functional equation for $L_{f}(s,p/q)$ (see (\ref{Hecke functional equation periodic Cusp})
and (\ref{definition xi cusp forms twisted}) above):
\[
\eta_{f}\left(s,\frac{p}{q}\right)=(-1)^{k/2}\,\eta_{f}\left(k-s,\,-\frac{\overline{p}}{q}\right),
\]
one can check that $\eta_{f}(k/2+it,p/q)$ is only real (or purely
imaginary) when $p^{2}\equiv1\mod q$.

\bigskip{}

Therefore, from now on, we will assume that $p^{2}\equiv1 \mod\,q$.
Since $L_{f}(s,p/q)$ is entire, we expect that the proof of Theorem
\ref{cusp form theorem} will have exactly the same nature as the proofs of theorem \ref{Dirichlet L functions result} and so we shall omit it. 

\bigskip{}

The only purpose of this section is to show that the analogue of Jacobi's
$\psi-$function for $L_{f}(s,p/q)$, (\ref{analogue Jacobi cusp form twisted}),
vanishes exponentially fast when $x\rightarrow i$ through the half
circle $|x|=1,\,\,\text{Re}(x)>0$. This is the last result of our paper. After showing it, Theorem \ref{cusp form theorem} is automatically proved once we follow the lines of the previous sections. 

\bigskip{}

\begin{lemma}
Let $f(\tau)$ be a cusp form of weight $k$ for the full modular
group. Consider the Dirichlet series:
\[
L_{f}(s,p/q)=\sum_{n=1}^{\infty}\frac{a_{f}(n)\,e^{\frac{2\pi ip}{q}n}}{n^{s}},\,\,\,\,(p,q)=1,\,\,\,p^{2}\equiv1 \mod\,q
\]
and its analogue of Jacobi's $\psi-$function, 
\[
\psi_{f,p/q}(x,z):=(k-1)!\,\left(\sqrt{\frac{\pi x}{2q}}\,z\right)^{1-k}\,\sum_{n=1}^{\infty}a_{f}(n)\,e^{\frac{2\pi ip}{q}n}\,n^{\frac{1-k}{2}}\,e^{-\frac{2\pi n}{q}\,x}\,J_{k-1}\left(\sqrt{\frac{2\pi n}{q}x}\,z\right),
\]
defined for $\text{Re}(x)>0$ and $z\in\mathbb{C}$.

\bigskip{}

Suppose further that $z$ belongs to the region:
\begin{equation}
z\in\mathscr{D}_{q}:=\left\{ z\in\mathbb{C}\,:\,|\text{Re}(z)|<2\sqrt{\frac{\pi}{q}},\,\,|\text{Im}(z)|<2\sqrt{\frac{\pi}{q}}\right\} .\label{condition cusp for vanishing}
\end{equation}

Then for every $z$ satisfying (\ref{condition cusp for vanishing}),
any $m\in\mathbb{N}_{0}$, and any analytic function $h:\,\mathbb{C}\longmapsto\mathbb{C}$, the following relation takes
place:
\begin{equation}
\lim_{\omega\rightarrow\frac{\pi}{4}^{-}}\,\frac{d^{m}}{d\omega^{m}}\left\{ h(\omega)\,\psi_{f,p/q}\left(e^{2i\omega},z\right)\right\} =0.\label{limiting form character case}
\end{equation}
\end{lemma}

\begin{proof}
The proof is instantaneous because we already have the transformation
(\ref{direct identity cusp form twisted}) in the right template.
Indeed,
\[
\psi_{f,p/q}(i+\delta,z)=(k-1)!\,2^{k-1}(\sqrt{i+\delta}\,z)^{1-k}\left(\frac{2\pi}{q}\right)^{\frac{1-k}{2}}\,\sum_{n=1}^{\infty}a_{f}(n)\,e^{\frac{2\pi i\,(p-1)}{q}n}\,n^{\frac{1-k}{2}}\,e^{-\frac{2\pi n}{q}\delta}\,J_{k-1}\left(\sqrt{\frac{2\pi n}{q}\,(i+\delta)}\,z\right).
\]

Applying (\ref{direct identity cusp form twisted}) and replacing
there $x$ by $\delta$ and $z$ by $\frac{\sqrt{i+\delta}}{\sqrt{\delta}}\,z$
and $p$ by $p-1$, we get:
\begin{align*}
\psi_{f,p/q}(i+\delta,z) & =(k-1)!\,2^{k-1}(\sqrt{i+\delta}\,z)^{1-k}\left(\frac{2\pi}{q}\right)^{\frac{1-k}{2}}\,\frac{e^{-\frac{(i+\delta)z^{2}}{4\delta}}}{\delta}\,(-1)^{k/2}\times\\
 & \times\,\sum_{n=1}^{\infty}a_{f}(n)\,e^{-\frac{2\pi i\,R\,n}{q}}\,n^{\frac{1-k}{2}}\,e^{-\frac{2\pi n}{q\delta}}\,I_{k-1}\left(\sqrt{\frac{2\pi n\,(i+\delta)}{q}}\,\frac{z}{\delta}\right),
\end{align*}
where $R$ is such $R\,(p-1)\equiv1\mod q$. From the bound (\ref{Bound Bessel})
and the work done in Lemmas \ref{Vanishing Theta lemma} and \ref{vanishing with analytic}, the proof of (\ref{limiting form character case})
reduces once more to show that
\[
\lim_{\delta\rightarrow0^{+}}\delta^{-1}\exp\left(-\frac{2\pi n}{q\delta}\,+\sqrt{\frac{2\pi n}{q}}\,\frac{|\text{Re}(z)-\text{Im}(z)|}{\sqrt{2}\delta}+\frac{\text{Re}(z)\,\text{Im}(z)}{2\delta}\right)=0.
\]

But a sufficient condition for this to hold is that $z$ satisfies
(\ref{condition cusp for vanishing}).
\end{proof} 

\bigskip{}

\bigskip{}

\bigskip{}

\textit{Disclosure Statement:} The authors declare that they have no conflict of interest. 

\bigskip{}

\textit{Acknowledgements:} Both authors were partially supported by CMUP, member of LASI, which is financed by national funds through FCT - Fundação para a Ciência e a Tecnologia, I.P., under the projects with reference UIDB/00144/2020 and UIDP/00144/2020.

The first author also acknowledges the support from FCT (Portugal) through the PhD scholarship 2020.07359.BD.  
\newpage{}

\end{document}